\documentclass[12pt,a4paper]{article}

\usepackage{amsmath,bm}    % 'bm': $bold math symbol, incl. \boldsymbol{\oint,\bigotimes}$,$\bm{\sum,\bigcup}$, $\pmb{\iint,\bigcap}$
\usepackage{amssymb}

\usepackage{amsthm}

\theoremstyle{plain}

\theoremstyle{definition}

\theoremstyle{remark}

\usepackage[english]{babel}
\usepackage{fontenc}
\usepackage[ansinew]{inputenc}
\usepackage{graphicx}
\usepackage{latexsym}
\usepackage{appendix}
\usepackage{amsthm}
\usepackage{subcaption}
\usepackage{caption}
\usepackage{verbatim}
\usepackage{bbm}  % Indicator function in Latex, methold (1), using the grama: \[ \mathbbm{1} \]
\usepackage{dsfont}% Indicator function in Latex, methold (2), using the grama: \[ \mathds{1} \]
\usepackage{siunitx} % insert degree Celsius centigrade
\usepackage{relsize} % adjust \sum size

\usepackage[T1]{fontenc}
\usepackage{siunitx}
\usepackage{eurosym}
\usepackage{tikz}
\usepackage{pgfplots}
\pgfplotsset{width=15cm,compat=1.9}
\usepgfplotslibrary{external}
\usetikzlibrary{decorations.pathreplacing}
\usepackage{subcaption} %draw multiple figures

\usepackage{enumerate}% http://ctan.org/pkg/enumerate

\usepackage{multirow} % for creating table
\usepackage{bigstrut} % for creating table

\usepackage{tcolorbox} %put a framed box around text

\usepackage{mathrsfs} %for \mathscr

\usepackage{mathtools} % for  A \underset{\mathclap{m + n}}{=} B 

\newtheorem{theorem}{Theorem}
\newtheorem{corollary}{Corollary}
\newtheorem{lemma}{Lemma}

\newtheorem*{definition*}{Definition}

\usepackage{titlesec}  % define subsubsubsection by paragraph
\titleformat{\paragraph}
{\normalfont\normalsize\bfseries}{\theparagraph}{1em}{}
\titlespacing*{\paragraph}
{0pt}{3.25ex plus 1ex minus .2ex}{1.5ex plus .2ex}
\titleformat{\subparagraph}
{\normalfont\normalsize\bfseries}{\thesubparagraph}{1em}{}
\titlespacing*{\subparagraph}
{0pt}{3.25ex plus 1ex minus .2ex}{1.5ex plus .2ex}

\newcommand*{\rom}[1]{\expandafter\@slowromancap\romannumeral #1@} % add Roma number I, II, III....

\title{\Large\textbf{Varentropy Estimation via Nearest Neighbor Graphs
}}

\begin{document}

%\author{Nikolai Leonenko \\ Cardiff University \and Yu Sun\\ London School of Economics \and Emanuele Taufer \\ University of Trento 
%}

%%\and Yu Sun\thanks{\scriptsize Department of Economics, London School of Economics and Political Science, UK. Email: y.sun47@lse.ac.uk}\textsuperscript{,} \thanks{\scriptsize CORE and Economic School of Louvain, Universit\'e catholique de Louvain, Belgium. } }

\author{Nikolai Leonenko\thanks{\scriptsize School of Mathematics, Cardiff University, UK. Email: LeonenkoN@cardiff.ac.uk}   
\and Yu Sun\thanks{\scriptsize Systemic Risk Centre, London School of Economics and Political Science, UK. Email: y.u.sun@outlook.com} \textsuperscript{,}\footnotemark[3] 
\and Emanuele Taufer \thanks{\scriptsize Department of Economics and Management, University of Trento, Italy. Email: emanuele.taufer@unitn.it} 
}

%%\author{Simon Dietz\thanks{\scriptsize ESRC Centre for Climate Change Economics and Policy, and Grantham Research Institute on Climate Change and the Environment, London School of Economics and Political Science, UK. Email: s.dietz@lse.ac.uk}    \textsuperscript{,}\thanks{\scriptsize Department of Geography and Environment, London School of Economics and Political Science}
%%\and Yu Sun\thanks{\scriptsize Department of Economics, London School of Economics and Political Science, UK. Email: y.sun47@lse.ac.uk}\textsuperscript{,} \thanks{\scriptsize CORE and Economic School of Louvain, Universit\'e catholique de Louvain, Belgium. } }

%\line numbers

%\date{}

\maketitle \thispagestyle{empty}

%%\emph{\hspace{2cm} Preliminary version, and please do not cite or circulate.}

\begin{abstract}
The Varentropy is a measure of the variability of the information content of random vector and it is invariant under affine transformations. We introduce the statistical estimate of varentropy of random vector based on the nearest neighbor graphs (distances). The asymptotic unbiasedness and $L^{2}$-consistency of the estimates are established.
%Many models in statistical analysis, including reliability theory and survival analysis, rely critically on the estimation of heavy-tailed behaviour. In comparison with Kurtosis measure, Varentropy has the advantage to estimate the heavy-tailed distributions, especially when the conventional tailedness of the probability distribution is not measurable. We establish a non-parametric Varentropy estimator for an unknown distribution $ f $ in $ \mathbb{R}^{d} $, based on the $ k $-th nearest neighbourhood distances computed from a sample of $ N $ identical and independent distributed (i.i.d.) vectors. Under some well-established conditions and plausible assumptions, the asymptotic unbiasedness and consistency of the estimator are demonstrated.

\vspace{20pt}
\vspace{10pt}
%\noindent \textbf{JEL Classification}: 
\end{abstract}

\newpage
\setcounter{page}{1}
%\section{Syllabus}
\section{Introduction}
Let $ (\Omega , \mathscr{F}, P) $ be a probability space and let $ X\in\mathbb{R}^{d}$ be a random vector defined on it. Suppose that the joint distribution of $ X $ has a density $ f(x) $ with respect to the Lebesgue measure $ dx $, with the support $\mathcal{S}=\mathcal{S}(f)=\{x\in\mathbb{R}^{d}: f(x)>0\} $. 
We consider the random variable 
\begin{eqnarray}\label{log variable}
h(X)&=&-\log f(X)\,,
%&=& \mathrm{E}(I^{S}(X)) \nonumber
\end{eqnarray}
which can be thought of as the (random) information content of $ X $ (or as log-likelihood function). The average value of information content of $ X $ is known as Shannon (or Boltzmann-Gibbs) entropy. 
\begin{eqnarray}\label{shannon entropy}
H &=& H(X):=\mathsf{E}\big[h(X)\big]=-\int_{\mathcal{S}}{f(x)\log f(x)}dx \,.
%&=& \mathrm{E}(I^{S}(X)) \nonumber
\end{eqnarray}
The variance of the information content of $ X $ is defined as follows.
\begin{eqnarray}\label{Varentropy}
V=V(X):=\mathsf{Var}\big[h(X)\big]=\int_{\mathcal{S}}{f(x)\big[\log f(x)}\big]^{2}dx- \bigg[\int_{\mathcal{S}}{f(x)\log f(x)}dx\bigg]^{2}
\end{eqnarray}
and known as Shannon varentropy. It may be thought as the measure of the variability of information content (or information varentropy), see e.g. Bobkov and Madiman (2011).
Let us emphasize that the distribution of the difference $ h(X)-H(X) $ is stable under all affine transformation of the space, i.e.,
$
h(\mathcal{A}X)-H(\mathcal{A}X)=h(X)-H(X)
$
for all invertible affine maps $ \mathcal{A}: \mathbb{R}^{d} \longrightarrow \mathbb{R}^{d} $. Thus, the varentropy represents an affine invariant: $ V(\mathcal{A}X+\mathcal{B})=V(X) $, $ \mathcal{B}\in\mathbb{R}^{d} $. In other words, the varentropy is independent of the mean and covariance matrix of the random vector X. In particular, all non-generate normal vectors have varentropy $ V(X)=\frac{d}{2} $. This property can be useful for a goodness-of-fit test of normality.
Note that Shannon entropy is shift invariant only (but not scale invariant). It is easy to see that for $ d=1 $, the varentropy of any normal, exponential and uniform distribution are equal to $ \frac{1}{2}$, $ 1 $, $ 0 $, correspondingly, and hence, using varentropy (or its statistical counterpart), one can separate normal, exponential and uniform distributions.

Song (2001) introduces varentropy and its statistical counterpart as an distinct measure of distribution shape, which can be an alternative to the kurtosis measure. When the traditional kurtosis measure is not measurable, such as Student's distributions with degrees of freedom less than four and Pareto distributions, the varentropy is a measure that can be used to compare the heavy-tailed distributions instead of kurtosis measure, see also Zogrofus (2008) and Batsidis and Zografos (2013) for the study of an empirical varentropy in the class of elliptically contoured distributions. %%Maadani et al. (2022) studied the properties of order statistics by using the Varentropy as an alternative of the kurtosis measure.
Recently, some applications of varentropy in survival analysis has been studied by Di Crescenzo and Paolillo (2021). %and Di Crescenzo, Paolillo and Su\'arez-Llorens (2021). 

We consider the problem of estimating varentropy in Eq.(\ref{Varentropy}) for random vector, for an $ N$ independent and identically distributed (i.i.d.) sample $ X_{1},\ldots ,X_{N} $, $ N\geq 2 $. This estimator are based on the nearest-neighbor (NN) distances computed from the sample. We show that varentropy in Eq.(\ref{Varentropy}) can be estimated consistently with minimal assumption on the density $ f $.
We consider a new estimator of varentropy based on the approach proposed by Kozachenko and Leonenko (1987), who consider the estimator of entropy in Eq.(\ref{shannon entropy}), see also Goria et al.(2005), Tsybakov and Van der Meulen (1996), Leonenko, N. N., Pronzato, L., Savani, V (2008), Berrett, Samworth and Yuan (2019), among others. To prove our results, we exploit the new approach to study limit theorems for point processes related to the k-th nearest-neighbor graphs proposed by Kozachenko and Leonenko (1987), see also their references.

Bulinski and Dimitrov (2019) proposed to use the analogous of the Hardy-Littlewood maximal functions for an investigation of unbiasedness and consistency of the NN estimators of entropies.
Berrett, Samworth and Yuan (2019) proposed an efficient multivariate entropy estimates via $k$th-NN distances, assuming that $ k(N)\rightarrow\infty $, $ N\rightarrow\infty $ with special rate.
Recently, Devroye and Gy\"orfi (2022) relaxed the smoothness conditions and prove the strong consistency of the NN estimators of entropy.
For some other recent development of the NN estimates and their applications to goodness-of-fit tests, see e.g. Penrose and Yukich (2011, 2013), Biau, G., Devroye, L. (2015), Delattre and Fournier (2017), Gao, Oh and Viswanath (2017), Leonenko, Makogin and Cadirci (2021), Cadirci et al. (2022), Berrett and Samworth (2023).

The paper is organised as follows: In Section 2, we define general notations and briefly introduce the NN estimate of varentropy and formulate the condition on density for its asymptotic unbiasedness and consistency. The proofs are collected in Section 3 and 4.
%%%%%%%%%%%%%%%%%%%%%%%%%%%%%%%%%%%%%%%%%%%%%%%%%%%%%%%%%%%%%%%%%%%%%%%%%%%%%%%%%%%%%%%%%%%%%%%%%%%%%%%%%%%%%%%%%%%%%%%%%%%%%%%%%
%\section{Estimation of Entropy and Varentropy}
\section{Main Results}
%%%%%%%%%%%%%%%%%%%%%%%%%%%%%%%%%%%%%%%%%%%%%%%%%%%%%%%%%%%%%%%%%%%%%%%%%%%%%%%%%%%%%%%%%%%%%%%%%%%%%%%%%%%%%%%%%%%%%%%%%%%%%%%%%%%%%%%%%%%%
Let $ X_{1},\ldots ,X_{N} $ be i.i.d. random vectors having the same law as the random vector $ X\in\mathbb{R}^{d} $. Assume that $ X $ has a density $ f $ with respect to the lebesgue measure $ dx $ in $ \mathbb{R}^{d} $. For each $ i=1,\ldots , N $ and $ N\geq 2 $, set $ \rho_{i}=\min\big\{\rho(X_{i},X_{j}), \,\forall i\in\{1\ldots N\}\,,\forall j\in\{1,\ldots,N\}\setminus\{i\} \big\} $, where $ \rho(x,y)=\parallel x-y\parallel $ denote the \textit{Euclidean distance} between $ x,y\in\mathbb{R}^{d} $. In other words, $\rho_{i}$ is the distance from $ X_{i} $ to its nearest neighbor (NN) in the sample $ \{X_{1},\ldots,X_{N}\}\setminus\{X_{i}\} $, and $ \{\rho_{i}, i=1,\ldots , N\} $ defines the NN random graph.
Recall the NN estimate (or Kozachenko-Leonenko estimator 1987) of entropy $ H $ defined in Eq.(\ref{shannon entropy}), is provided by the formula 
\begin{eqnarray}
&&\hspace{.4cm} H_{N}=\dfrac{1}{N}\sum\limits_{i=1}^{N}\zeta_{i}(N) \label{HN zeta} \\
&&\hspace{-1cm}\text{with}\,\,\zeta_{i}(N)=\log \big[{\rho_{i}}^{d}V_{d}\,e^{\gamma}(N-1)\big]\,,\,\, i=1,\ldots, N. \label{zeta_i N}
%,\,\text{and the constant}\,\,\widetilde{\gamma}=e^{\gamma}\approx 1.781 \nonumber \\
%&&\hspace{-1cm}\text{with \textit{Euler constant}}\,\gamma=-\int_{0}^{\infty} e^{-t}\log t\,d{t} \approx 0.5772 \label{Euler constant}
\end{eqnarray}
where $ V_{d}:={\pi^{(d/2)}}/{\Gamma\big(\frac{d}{2}+1\big)}$, $ \gamma=-\int_{0}^{\infty} e^{-t}\log t\,d{t} \approx 0.5772 $ are the volume of a unit ball in $ \mathbb{R}^{d} $ and the  \text{Euler-Mascheroni constant}, respectively.
It is worth to remark that we consider the NN distances for simplicity only. Similar results can be obtained  for the estimates of entropy and varentropy based on the $k$-th NN graphs, constructed by distances  from $ X_{i} $ to its $k$-th NN. This will be done elsewhere.

We define the NN estimate of the varentropy in Eq.(\ref{Varentropy}) as follows.
\begin{eqnarray}
&&\hspace{.6cm} V_{N}:=S^{2}_{N}- H_{N}^{2} \label{V_N} \\
&&\text{with}\,\, S_{N}^{2}= \dfrac{1}{N}\sum\limits_{i=1}^{N}\bigg[\zeta_{i}^{2}(N)-\dfrac{\pi^{2}}{6}\bigg]  \label{S_N} 
%\\
%&&\hspace{-.6cm} \text{and the non-negative constant}\,\, \sigma^{2}=\int^{\infty}_{0}\log^{2}v\, e^{-v}dv - \Big(\int^{\infty}_{0}\log v\, e^{-v}dv\Big)^{2} \label{sigma squared}%\,\,\text{and}\,\, \mathsf{E}_{N}(\cdot):=\dfrac{1}{N}\sum\limits_{i=1}^{N}(\cdot) 
\end{eqnarray}
where $ H_{N} $ and $ \zeta_{i}(N) $ are defined in Eqs.(\ref{HN zeta})-(\ref{zeta_i N}). The nature of the constant ${\pi^{2}}/{6}  $ follows from the identity:
\begin{align}\label{pi^2/6}
\int^{\infty}_{0}\log^{2}v\, e^{-v}dv - \Big(\int^{\infty}_{0}\log v\, e^{-v}dv\Big)^{2} =\frac{\pi^{2}}{6} \approx 1.6449
\end{align}
Let $ B(x,r)=\{y\in\mathbb{R}^{d}: \rho(x,y)< r\} $ be a ball of a radius $ r>0 $ with a center $ x\in\mathbb{R}^{d} $. Clearly, $ \mid B(x,r)\mid = \mu\big(B(x,r)\big)=r^{d}\,V_{d} $. Set $ G(t) $ a monotonically increasing function on $ [0,\infty) $:
\begin{eqnarray}\label{G(t)}
  G(t) = \begin{cases}
     0, & 0 \leq t < 1, \\[2\jot]
    t\log t, & t\geq 1\,.
  \end{cases}
\end{eqnarray}
Following Bulinski and Dimitrov (2019), we introduce the following functionals:
\begin{eqnarray}
&&\hspace{1.6cm} I_{f}(x,r)=\dfrac{\int_{B(x,r)}f(y)dy}{r^{d}V_{d}}\label{I_f(x,r)} \\
&& M_{f}(x,R)=\sup\limits_{r\in(0,R]}I_{f}(x,r) \quad \text{and}
\quad m_{f}(x,R)=\inf\limits_{r\in(0,R]}I_{f}(x,r) \label{M_f e m_f}
\end{eqnarray}
where $ f$ is the probability density.
It is known (see Bulinski and Dimitrov 2019) that the function $   I_{f}(x,r)$ is continuous in $ (x,r)\in \mathbb{R}^{d}\times (0,\infty) $, while for each $ R>0 $, the functions $ m_{f}(\cdot,R) $ and $ M_{f}(\cdot,R) $ are upper semicontinuous and lower semicontinuous, respectively.
Hence, these non-negative functions are Borel measurable. Clearly, for each $ x\in\mathbb{R}^{d} $,  $ m_{f}(\cdot,R) $ is nonincreasing and $ M_{f}(\cdot,R) $ is nondecreasing. Note in passing that, changing $ \sup_{r\in(0,R]} $ by $ \sup_{r\in(0,\infty)} $ in the definition of $M_{f}(x,R)  $ leads to the celebrated Hardy-Littlewood maximal function $ M_{f}(x) $ widely used in the harmonic analysis.

For a probability density $ f $ in $ \mathbb{R}^{d} $, positive
$ \varepsilon_{i}$, $R_{j}$, where $ i=0,1,2 $ and $ j=1,2$, we define the following functions with values in $ [0,\infty) $. % in $ [0,\infty] $. %where $ \varepsilon_{i}>0, R_{j}>0 $, with $ i=0,1,2 $ and $ j=1,2 $. {\color{red} $ \alpha = 1,2,3,4 $}
\begin{eqnarray}
&&\hspace{-.6cm} K_{f,{\color{black} \alpha}}(\varepsilon_{0}):=\int_{\mathbb{R}^{d}}\bigg(\int_{\mathbb{R}^{d}\setminus\{x\}}G(\mid\log^{ {\color{black} \alpha}}\rho(x,y) \mid)f(y)dy\bigg)^{1+\varepsilon_{0}}f(x)dx\,,\label{condition K}\\
&&\hspace{-.6cm}\text{where}\,\, {\color{black} \alpha = 1,2,3,4 }  \nonumber\\
&&\hspace{-.6cm} Q_{f}(\varepsilon_{1},R_{1}):=\int_{\mathbb{R}^{d}}M^{\varepsilon_{1}}_{f}(x,R_{1})f(x)dx  \label{condition Q}\\
&&\hspace{-.6cm} T_{f}(\varepsilon_{2},R_{2}):=\int_{\mathbb{R}^{d}}m^{-\varepsilon_{2}}_{f}(x,R_{2})f(x)dx \,.\label{condition T}
\end{eqnarray}
The main results are as follows.
%%\begin{tcolorbox}
\begin{theorem}\label{result I}
Assume that, for some positive $ \varepsilon_{i},\,R_{j}$, where $ i= 0,1,2 $ and $ j=1,2 $, the functionals appearing of Eqs.(\ref{condition K}), (\ref{condition Q}), (\ref{condition T}) are finite for $ {\color{black} \alpha=1,2} $, that is $ K_{f,\,{\color{black} \alpha}}(\varepsilon_{0})<\infty,\,{\color{black} \alpha=1,2}$, $ Q_{f}(\varepsilon_{1},R_{1})<\infty $ and $ T_{f}(\varepsilon_{2},R_{2})<\infty $. Then the estimate of varentropy in Eq.(\ref{V_N}) is asymptotically unbiased, i.e.
\begin{eqnarray}
\mathsf{E}_{}(V_{N})\rightarrow V\,,\qquad N\rightarrow\infty \label{Varentropy exp eq1} 
\end{eqnarray}
\end{theorem}
%%\end{tcolorbox}
%%\begin{tcolorbox}
\begin{theorem}\label{result II}
Assume that, for some positive $ \varepsilon_{i},\,R_{j}$, where $ i= 0,1,2 $ and $ j=1,2 $, the functionals appearing in Eqs.(\ref{condition K}), (\ref{condition Q}), (\ref{condition T}) are finite for $ {\color{black} \alpha=1,2} $, that is $ K_{f,\,{\color{black} \alpha}}(\varepsilon_{0})<\infty,\,{\color{black} \alpha=1,2,3,4}$, $ Q_{f}(\varepsilon_{1},R_{1})<\infty $ and $ T_{f}(\varepsilon_{2},R_{2})<\infty $. Then the estimate of varentropy in Eq.(\ref{V_N})is $ L^{2}-$consistent, i.e.
\begin{eqnarray}\label{Varentropy var eq1}
\mathsf{E}\big(V_{N}-V\big)^{2} \,\xrightarrow[\text{}]{}\,0 \,,\quad N\rightarrow \infty\,.
%\mathsf{Var}(V_{N}^{2})\,\xrightarrow[\text{}]{N\rightarrow \infty}\,0 \,,\quad x\in A\,.
\end{eqnarray}
\end{theorem}
%%\end{tcolorbox}
The proofs of \textit{Theorem} \ref{result I} and \textit{Theorem} \ref{result II} are given in Section 3 and 4, respectively.
Let us consider the following conditions, see Kozachenko-Leonenko (1987), Bulinski and Dimitrov (2019).
%\begin{enumerate}[(A)]
(A) For some $ p>1 $,
\begin{align} \label{int log_rho power p}
\int_{\mathbb{R}^{d}}\int_{\mathbb{R}^{d}}|\log\rho(x,y)|^{p}f(x)f(y)dxdy<\infty
\end{align}
(B) There exists a version of density $ f $ such that, for some $ M>0 $,
\begin{align}  \label{f<=M}
f(x)\leq M<\infty\,,\quad x\in\mathbb{R}^{d}
\end{align}
(C1) There exists a version of density $ f $ such that, for some $ m>0 $,
\begin{align}  \label{f<=M}
f(x)\geq m>0,\quad x\in\mathcal{S}(f)
\end{align}
(C2) For a fixed $ R>0 $, there exists a constant $ c>0 $ and a version of density $ f $ such that
\begin{align}\label{m_f>cf}
m_{f}(x,R)\geq c\, f(x),\quad x\in\mathcal{S}(f)
\end{align}
%\end{enumerate}
Note that if, for some positive $ \varepsilon,\,R,\, c $, condition (C2) is true and 
\begin{align}\label{int f^(1-varepsilon)}
\int_{\mathbb{R}^{d}} [f(x)]^{1-\varepsilon}dx <\infty\,,
\end{align}
then $ T_{f}(\varepsilon , R)<\infty $.
\begin{corollary} \label{corollary 1}
Any assumption of \textit{Theorem} \ref{result I} concerning the finiteness of functionals (\ref{condition K}), (\ref{condition Q}), (\ref{condition T}) can be replaced by conditions (A), (B), (C1), respectively, and then (\ref{Varentropy exp eq1}) will be true as well. Moreover, if (B) and (C1) are satisfied, then (\ref{Varentropy exp eq1}) holds whenever $ f $ has a bounded
support.
\end{corollary}
\begin{corollary} \label{corollary 2}
Any assumption of \textit{Theorem} \ref{result II} concerning the finiteness of functionals (\ref{condition K}), (\ref{condition Q}), (\ref{condition T}) can be replaced, respectively, by the following ones: (\ref{int log_rho power p}) is valid for some $ p>2 $, (B) and (C1). Then (\ref{Varentropy var eq1}) will be true. Moreover, if (B) and (C1) are satisfied, then (\ref{Varentropy var eq1}) holds whenever $ f $ has a bounded support.
\end{corollary}
\begin{corollary} \label{corollary 3}
If, for some positive $ \varepsilon,\,R$ and $c $, condition (C2) is true and (\ref{int f^(1-varepsilon)}) holds. Then in \textit{Theorems} \ref{result I} and \ref{result II}, we can employ (C2) and (\ref{int f^(1-varepsilon)}) instead of assumptions concerning $ T_{f}(\varepsilon , R) $. 
In particular, relations (\ref{Varentropy exp eq1})and (\ref{Varentropy var eq1}) hold for Gaussian random variable $ \mathcal{N}(\nu , \Sigma) $ in $\mathbb{R}^{d} $ with $ H=\frac{1}{2}\log \det(2\pi e\Sigma) $.
\end{corollary}
The proofs of \textit{Corollaries} 1-3 are similar to Bulinski and Dimitrov (2019).
%%%%%%%%%%%%%%%%%%%%%%%%%%%%%%%%%%%%%%%%%%%%%%%%%%%%%%%%%%%%%%%%%%%%%%%%%%%%%%%%%%%%%%%%%%%%%%%%%%%%%%%%%%%%%%%%%%%%%%%%%%%%%%%%%%%%%%%%%%%%
%%\subsection{Outline of the Proofs}
%%%%%%%%%%%%%%%%%%%%%%%%%%%%%%%%%%%%%%%%%%%%%%%%%%%%%%%%%%%%%%%%%%%%%%%%%%%%%%%%%%%%%%%%%%%%%%%%%%%%%%%%%%%%%%%%%%%%%%%%%%%%%%%%%%%%%%%%%%%%
%%In Section \ref{section Proof of Theorem 1}, we show the asymptotic convergence of $ S^{2}_{N} $ in Lemma 1, and together with the asymptotic convergence of $ H_{N}^{2} $, we obtain to the proof of Theorem 1.\\
%%In Section \ref{section Proof of Theorem 2}, we show that the $ L^{2} $-consistence of $ V_{N}^{2} $ is guaranteed by $ \mathsf{Var}(S^{2}_{N})\xrightarrow[\text{}]{}\,  0$ and $ \mathsf{Var}(H^{2}_{N}) \,\xrightarrow[\text{}]{} 0$, as $ N\rightarrow \infty$. The Proof of $ \mathsf{Var}(S^{2}_{N})\xrightarrow[\text{}]{}\,  0$ is divided into the proof of Variance and Covariance matrix. 
%%%%%%%%%%%%%%%%%%%%%%%%%%%%%%%%%%%%%%%%%%%%%%%%%%%%%%%%%%%%%%%%%%%%%%%%%%%%%%%%%%%%%%%%%%%%%%%%%%%%%%%%%%%%%%%%%%%%%%%%%%%%%%%%%%%%%%%%%%%%
\section{Proof of Theorem 1}\label{section Proof of Theorem 1}
%%%%%%%%%%%%%%%%%%%%%%%%%%%%%%%%%%%%%%%%%%%%%%%%%%%%%%%%%%%%%%%%%%%%%%%%%%%%%%%%%%%%%%%%%%%%%%%%%%%%%%%%%%%%%%%%%%%%%%%%%%%%%%%%%%%%%%%%%%%%
According to the Lebesgue differentiation theorem, if $ f\in L^{1}(\mathbb{R}^{d}) $ then, for almost all $ x\in\mathbb{R}^{d} $ with respect to Lebesgue measure, the following result holds.
\begin{align} \label{limit fy-fx}
\lim\limits_{r\rightarrow 0+}\dfrac{1}{|B(x,r)|}\int_{B(x,r)}|f(y)-f(x)|dy=0
\end{align}
Let $ \Lambda(f) $ stands for a set of all the Lebesgue points of a density f, i.e. for $ x\in\mathbb{R}^{d} $ satisfying Eq.(\ref{limit fy-fx}).
Clearly $ \Lambda(f) $ depends on the chosen version of $ f $ belonging to the class of equivalent functions
from $ L^{1}(\mathbb{R}^{d}) $. For each version of $ f $, we have that Lebesgue measure of the set $ \{\mathbb{R}^{d}\setminus\Lambda(f)\} $ is equal to zero.

Note that the random variables $ \zeta_{1}(N),\ldots , \zeta_{N}(N) $ are identically distributed since $ X_{1},\ldots , X_{N} $ are i.i.d. random vectors. Thus, $ \mathsf{E}H_{N}=\mathsf{E}[\zeta_{1}(N)] $. Using Eq.(\ref{limit fy-fx}), one can prove that under the conditions of \textit{Theorem} \ref{result I}, $ \mathsf{E}_{}(V_{N})\rightarrow V,\, N\rightarrow\infty $. We recall the main steps of this proof (see, Kozachenko-Leonenko 1987 and Bulinski and Dimitrov 2019 for details).\\
%\textbf{Step 1}\quad 
First, for $ x\in\mathbb{R}^{d} $ and $ u>0 $, as $ N\longrightarrow \infty $
\begin{align} \label{F_Nx(u) limit}
F_{N,x}(u)=&\mathsf{Pr}\big(e^{\zeta_{1}(N)}\leq u \big|X_{1}=x\big)=\mathsf{Pr}(\xi_{N,x}\leq u)   \nonumber\\
\rightarrow &\mathsf{Pr}(\xi_{x}\leq u)=F_{x}(u)  =1-\exp\big\{-u\cdot f(x)\,e^{-\gamma}\big\}\,,
\end{align}
where $\xi_{x}  $ is an exponential random variable with the mean $ e^{\gamma}/f(x) $. Relation (\ref{F_Nx(u) limit}) means that 
\begin{align}\label{xi limit}
\xi_{N,x}=(N-1)V_{d}\,e^{\gamma}\min\limits_{j=2,\ldots ,N}\rho^{d}(x, X_{j})\,\xrightarrow[]{\text{law}}\,\xi_{x} \in\mathcal{A}=\mathcal{S}(f)\cap \Lambda(f),\,N\rightarrow\infty\,.
\end{align}
%\textbf{Step 2}\quad 
Second, for almost every $ x\in\mathcal{S}(f) $ as $ N\rightarrow\infty $
\begin{align} \label{log_xi limit}
 \mathsf{E}\big(\log\xi_{N,x}\big)\rightarrow \mathsf{E}\big(\log\xi_{x}\big)=-\log f(x)\,.
\end{align}
%\textbf{Step 3}\quad 
Last, for almost all $ x\in\mathcal{S}(f) $, the family of $\{\log\xi_{N,x}\}_{N\geq N_{0}(x)} $ is uniformally integrable. That is for a function $ G $ defined in Eq.(\ref{G(t)})
\begin{eqnarray} 
\sup\limits_{N\geq N_{0}(x)} \mathsf{E}G(\mid \log\xi_{N,x}\mid)\leq C_{0}(x)<\infty  \label{uniform integrability} 
\end{eqnarray}
for almost every $ x\in\mathcal{S}(f) $, and hence
\begin{align}\label{HN zeta X1=x}
\mathsf{E}_{}\big(H_{N}\big)=\mathsf{E}_{}\big[\zeta_{1}(N)\big]=&\mathsf{E}_{}\big[\mathsf{E}_{}\big(\zeta_{1}(N)\big | X_{1}=x\big)\big]        \nonumber\\
=&\int_{\mathbb{R}^{d}}\bigg[\int_{-\infty}^{\infty} w \, d\mathsf{P}\big({\zeta_{1}(N)}\leq w\big |{X_{1}=x}\big)\bigg]f(x)dx \nonumber\\
=& \int_{\mathbb{R}^{d}}\bigg[\int_{0}^{\infty}(\log u) d\mathsf{P}(\xi_{N,x}\leq u)\bigg]f(x)dx  \nonumber\\
\rightarrow &\int_{\mathbb{R}^{d}}\big[\mathsf{E}_{}\big(\log\xi_{x}\big)\big]f(x)dx,\quad\text{as}\,\, N\rightarrow\infty \nonumber\\
=& \int_{\mathbb{R}^{d}}\big[-\log f(x)\big]f(x)dx=H \,.
\end{align}
We have to prove that $ \mathsf{E}_{}(V_{N})\rightarrow V$, $ N\rightarrow\infty $ in Eq.(\ref{Varentropy exp eq1}), where $ V_{N} $ is defined by Eq.(\ref{V_N}).
From Kozachenko-Leonenko (1987) and Bulinski and Dimitrov (2019), we has
\begin{align}
\mathsf{E}_{}(H_{N}-H)\rightarrow 0,\,\,\mathsf{E}_{}(H_{N}-H)^{2}\rightarrow 0,\,\,\text{as}\,\, N\rightarrow\infty \,.\nonumber 
\end{align}
Under conditions of \textit{Theorem} \ref{result I}. Thus, $\lim\limits_{N\rightarrow\infty} \mathsf{E}_{} (H^{2}_{N})= H^{2} $ and for Eq.(\ref{Varentropy exp eq1}), we need to prove that
\begin{align}\label{S_N limit}
\mathsf{E}_{} (S^{2}_{N}) \rightarrow \mathsf{E}_{}\big[\log^{2} f(x)\big],\quad N\rightarrow\infty\,,
\end{align}
which follows from the following statement.
%%\begin{tcolorbox}[colback=white, grow to left by=-6mm]
\begin{lemma}\label{result I.1}
Assume that, for some positive $ \varepsilon_{i},\, i= 0,1,2 $ and $ R_{j},\, j=1,2 $, the functionals $ K_{f,\,{\color{black} \alpha}}(\varepsilon_{0})<\infty,\,{\color{black} \alpha=1,2}$,\,$ Q_{f}(\varepsilon_{1},R_{1})<\infty $ and $ T_{f}(\varepsilon_{2},R_{2})<\infty $, then %the estimates of $ S^{2}_{N}  $ is asymptotically unbiased estimator of $ \int_{A}\big[\log f(x)\big]^{2}\,f(x)dx $, i.e.
\begin{equation}\label{SN limit}
\lim\limits_{N\rightarrow\infty} \mathsf{E}_{} (S^{2}_{N})=\int_{}\big[\log f(x)\big]^{2}\,f(x)dx \,.
%\,,\quad x\in A\,.
\end{equation}
\end{lemma}
%%\end{tcolorbox}
%%%%%%%%%%%%%%%%%%%%%%%%%%%%%%%%%%%%%%%%%%%%%%%%%%%%%%%%%%%%%%%%%%%%%%%%%%%%%%%%%%%%%%%%%%%%%%%%%%%%%%%%%%%%%%%%%%%%%%%%%%%%%%%%%%%%%%%%%%%%
\subsection{Proof of \textbf{Lemma} \ref{result I.1}}
%%%%%%%%%%%%%%%%%%%%%%%%%%%%%%%%%%%%%%%%%%%%%%%%%%%%%%%%%%%%%%%%%%%%%%%%%%%%%%%%%%%%%%%%%%%%%%%%%%%%%%%%%%%%%%%%%%%%%%%%%%%%%%%%%%%%%%%%%%%%
The random variables $ \zeta_{1}^{2}(N),\ldots , \zeta_{N}^{2}(N) $ are identically distributed, and hence,
\begin{eqnarray}\label{SN zeta X1=x}
&\mathsf{E}_{}\big( S_{N}^{2}\big)=\mathsf{E}_{x}\big[\zeta_{1}^{2}(N)\big]-\dfrac{\pi^{2}}{6} \,.
\end{eqnarray}
\begin{eqnarray} \label{S_N,x def}
&&\text{Let}\,\,\, S^{2}_{N,x}=\bigg\{\zeta^{2}_{1}(N)-\frac{\pi^{2}}{6}\,\bigg | X_{1}=x\bigg\}=\log^{2}\xi_{N,x}-\frac{\pi^{2}}{6}\,,  \\
&&\text{then}\,\,\, \big\{\zeta^{2}_{1}(N)\,\big | X_{1}=x\big\}=\log^{2}\xi_{N,x}\,,\,\text{where $ \xi_{N,x} $ is defined in Eq.(\ref{xi limit}).}  \nonumber
\end{eqnarray}
Note that Eq.(\ref{SN limit}) follows from one of the following equivalent statements as $ N\rightarrow \infty $.
\begin{eqnarray}
&& \mathsf{E}_{}\big[\zeta_{1}^{2}(N)\big]-\frac{\pi^{2}}{6}\,\xrightarrow[\text{}]{}\, \mathsf{E}_{}\big[\log ^{2}f(x)\big] \,, \\
&& \mathsf{E}_{}\Big\{\mathsf{E}_{}\big[\big(\zeta^{2}_{1}(N)-\dfrac{\pi^{2}}{6}\big)\big | X_{1}=x\big]\Big\} \,\xrightarrow[\text{}]{}\, \mathsf{E}_{}\big[\log ^{2}f(x)\big]\,, \label{E_x+E_zeta power2} \\
&& \mathsf{E}_{}\big[\mathsf{E}_{}\big(S^{2}_{N,x}\big)\big] \,\xrightarrow[\text{}]{N\rightarrow \infty}\, \mathsf{E}_{}\big[\log ^{2}f(x)\big] \,. \label{E_x+E_zeta power2 eq2}
\end{eqnarray}
Note that Eq.(\ref{E_x+E_zeta power2}) holds true if we prove that
\begin{enumerate}[(A)]
\item  $\mathsf{E}\big(S_{N,x}\big)-\dfrac{\pi^{2}}{6}\,\xrightarrow[]{}\, \log^{2} f(x),\quad N\rightarrow\infty $.
\item The family $ \big\{S_{N,x}\big\}_{N\geq N_{0}} $ with $ N_{0}\in\mathbb{N} $ is uniformly integrable.
\end{enumerate}
From \textit{Theorem} 2.8 of Bulinski and Dimitrov (2019), we obtain
\begin{eqnarray}\label{uniform integrability 1.1}
\sup_{N\geq N_{0(x)}} \mathsf{E} G \big[\mid \log^{2}\xi_{N,x}\mid\big]\leq C_{0}(x)<\infty 
\end{eqnarray}
and since $ \xi_{N,x} \xrightarrow[\text{}]{\text{law}}\xi_{x}$,  $ x\in \mathcal{A}$, $N\rightarrow \infty $ and convergence in law of random variable is preserved under continuous mapping, we have 
\begin{align} \label{convergence log_xi}
\log^{2}\xi_{N,x}\,\xrightarrow[]{\text{law}}\, \log^{2}\xi_{x}\,,\qquad N\rightarrow\infty\,.
\end{align}
The statement of (A) is proved by Eqs.(\ref{uniform integrability 1.1}) and (\ref{convergence log_xi}). The statement (B) follows from Eq.(\ref{uniform integrability 1.1}), see the Eq.(4.2) on page 32 of Bulinski and Dimitrov (2019) for details.
%%%%%%%%%%%%%%%%%%%%%%%%%%%%%%%%%%%%%%%%%%%%%%%%%%%%%%%%%%%%%%%%%
\section{Proof of \text{Theorem} \ref{result II}}\label{section Proof of Theorem 2}
%%%%%%%%%%%%%%%%%%%%%%%%%%%%%%%%%%%%%%%%%%%%%%%%%%%%%%%%%%%%%%%%%
%%%%%%%%%%%%%%%%%%%%%%%%%%%%%%%%%%%%%%%%%%%%%%%%%%%%%%%%%%%%%%%%%
%\subsection{Preliminary}
We show a stronger statement that the $ L^{2} $-consistence of $ V_{N}^{2} $ is guaranteed by $ \mathsf{Var}(S^{2}_{N})\xrightarrow[\text{}]{}\,  0$ and $ \mathsf{Var}(H^{2}_{N}) \,\xrightarrow[\text{}]{} 0$, as $ N\rightarrow \infty$.
\subsection{A Stronger Statement}
%%%%%%%%%%%%%%%%%%%%%%%%%%%%%%%%%%%%%%%%%%%%%%%%%%%%%%%%%%%%%%%%%%%%%%%%%%%%%%%%%%%%%%%%%%%%%%%%%%%%%%%%%%%%%%%%%%%%%%%%%%%%%%%%%%%%%%%%%%%%
As $ N\rightarrow\infty $, we have
\begin{eqnarray}\label{Varentropy var eq2}
&&\hspace{-1.5cm} \mathsf{E}(V_{N}-V)^{2}=\mathsf{E}(V_{N}-\mathsf{E}V_{N})^{2}+R_{N}^{(0)},\quad R_{N}^{(0)}=\big(\mathsf{E}V_{N}-V\big)^{2}\,\xrightarrow[\text{}]{}\, 0\,.
\end{eqnarray}
Moreover,
\begin{eqnarray}
\mathsf{E}(V_{N}-\mathsf{E}V_{N})^{2}&=&\mathsf{E}\big[(S^{2}_{N}-\mathsf{E}S^{2}_{N})- (H_{N}^{2}-\mathsf{E} H_{N}^{2})\big]^{2}  \nonumber\\
&=&\mathsf{Var}(S^{2}_{N})+\mathsf{Var}(H^{2}_{N})-2\mathsf{E}(S^{2}_{N}-\mathsf{E}S^{2}_{N})(H_{N}^{2}-\mathsf{E} H_{N}^{2}) \,,\label{EV_N-V power2}
\end{eqnarray}
where by the Jensen's inequality, we have
\begin{eqnarray}\label{Jensen E_SH power2}
-\mathsf{E}[(S^{2}_{N}-\mathsf{E}S^{2}_{N})(H_{N}^{2}-\mathsf{E} H_{N}^{2})]&\leq &\big|\mathsf{E}[(S^{2}_{N}-\mathsf{E}S^{2}_{N})(H_{N}^{2}-\mathsf{E} H_{N}^{2})]\big| \nonumber\\
&\leq & \mathsf {E} [|(S^{2}_{N}-\mathsf{E}S^{2}_{N})(H_{N}^{2}-\mathsf{E} H_{N}^{2})|]\,.
\end{eqnarray}
From the H\"older's inequality, we have
\begin{eqnarray}\label{Holder E_SH power2}
\mathsf {E}\big|(S^{2}_{N}-\mathsf{E}S^{2}_{N})(H_{N}^{2}-\mathsf{E} H_{N}^{2})\big| &\leq & \sqrt{\mathsf{Var}(S^{2}_{N})}\,\sqrt{\mathsf{Var}(H^{2}_{N})}\,.
\end{eqnarray} 
From the Eqs.(\ref{EV_N-V power2}), (\ref{Jensen E_SH power2}) and (\ref{Holder E_SH power2}), we obtain
\begin{eqnarray}
\mathsf{E}(V_{N}-\mathsf{E}V_{N})^{2}
&\leq &\bigg[ \sqrt{\mathsf{Var}(S^{2}_{N})}+\sqrt{\mathsf{Var}(H^{2}_{N})}\bigg]^{2} \,.\label{EV_N-V power2 eq2}
\end{eqnarray}
Therefore, the $ L^{2} $-consistence in Eq.(\ref{Varentropy var eq1}) follows from the following two statements.
\begin{eqnarray} \label{Varentropy var eq3}
 \mathsf{Var}(S^{2}_{N}) \,\xrightarrow[\text{}]{}\, 0 \,,\quad \mathsf{Var}(H^{2}_{N}) \,\xrightarrow[\text{}]{}\, 0,\quad\text{as}\,\, N\rightarrow \infty\,.
\end{eqnarray}
Thus, proof of \textit{Theorem} \ref{result II} is guaranteed by the statement in Eq.(\ref{Varentropy var eq3}). 
%%%%%%%%%%%%%%%%%%%%%%%%%%%%%%%%%%%%%%%%%%%%%%%%%%%%%%%%%%%%%%%%%%%%%%%%%%%%%%%%%%%%%%%%%%%%%%%%%%%%%%%%%%%%%%%%%%%%%%%%%%%%%%%%%%%%%%%%%%%%
\subsection{Proof of $\mathsf{Var}(S^{2}_{N}) \,\xrightarrow[\text{}]{}\, 0  $ }
%%%%%%%%%%%%%%%%%%%%%%%%%%%%%%%%%%%%%%%%%%%%%%%%%%%%%%%%%%%%%%%%%%%%%%%%%%%%%%%%%%%%%%%%%%%%%%%%%%%%%%%%%%%%%%%%%%%%%%%%%%%%%%%%%%%%%%%%%%%%
Let us introduce random variables $ \tilde{\zeta}_{i}^{2}(N)=\zeta_{i}^{2}(N)-\frac{\pi^{2}}{6} $, $ i=1,\ldots , N $. Then
\begin{eqnarray}\label{SNx expansion}
\mathsf{Var}\big(S_{N}^{2}\big)&=&\mathsf{Var}\bigg[ \dfrac{1}{N}\sum\limits_{i=1}^{N}\tilde{\zeta}_{i}^{2}(N)\bigg] \nonumber\\
&=&\dfrac{1}{N}\mathsf{Var}\big[ \tilde{\zeta}_{1}^{2}(N)\big] +\dfrac{2}{N^{2}} \sum\limits_{1\leq i < j\leq N}^{N}\mathsf{Cov}\big[ \tilde{\zeta}_{i}^{2}(N)\,,\,\tilde{\zeta}_{j}^{2}(N)\big]\,.
\end{eqnarray}
We analyse the first and second terms in Eq.(\ref{SNx expansion}) seperately.
%%%%%%%%%%%%%%%%%%%%%%%%%%%%%%%%%%%%%%%%%%%%%%%%%%%%%%%%%%%%%%%%%%%%%%%%%%%%%%%%%%%%%%%%%%%%%%%%%%%%%%%%%%%%%%%%%%%%%%%%%%%%%%%%%%
\subsubsection{Boundedness of $ \mathsf{Var}\big[ \tilde{\zeta}_{1}^{2}(N)\big] $}
%%%%%%%%%%%%%%%%%%%%%%%%%%%%%%%%%%%%%%%%%%%%%%%%%%%%%%%%%%%%%%%%%%%%%%%%%%%%%%%%%%%%%%%%%%%%%%%%%%%%%%%%%%%%%%%%%%%%%%%%%%%%%%%%%%
We will prove that
\begin{align}\label{Varentropy var eq1}
\lim\limits_{N\rightarrow\infty}\dfrac{1}{N}\mathsf{Var}\big[ \tilde{\zeta}_{1}^{2}(N)\big]=0\,.
\end{align}
Note that 
%\begin{eqnarray}
%\mathsf{E}^{2}\big[ \tilde{\zeta}_{1}^{2}(N)\big]=\Big\{\mathsf{E} \big[\zeta_{i}^{2}(N)-\sigma^{2}\big]\Big\}^{2}=\log^{4} f(x)
%\end{eqnarray}
%We have the variance term:
\begin{eqnarray}
&&\mathsf{Var}\big[ \tilde{\zeta}_{1}^{2}(N)\big]=\mathsf{E}\big[ \tilde{\zeta}_{1}^{4}(N)\big]-\mathsf{E}\big[ \tilde{\zeta}_{1}^{2}(N)\big]^{2}    \label{zeta tilde var 2}\\
&&\hspace{-.6cm}\text{and}\,\,\,\lim\limits_{N\rightarrow\infty}\mathsf{E}\big[ \tilde{\zeta}_{1}(N)\big]^{2}=\log^{4} f(x) \label{zeta tilde var 2.2}\,.
\end{eqnarray}
Using  $ \tilde{\zeta}_{1}^{4}(N)=\zeta_{1}^{4}(N)-\frac{\pi^{2}}{3}\zeta_{1}^{2}(N)+\frac{\pi^{4}}{36}  $, we have
%\begin{eqnarray}\label{zeta tilde power 4}
%\mathsf{E}\big[ \tilde{\zeta}_{1}^{4}(N)\big]%&=&\mathsf{E}\big[\zeta_{1}^{4}(N)-2\sigma^{2}\zeta_{1}^{2}(N)+\sigma^{4}\big]\nonumber\\
%&=&\mathsf{E}[\tilde{\zeta}_{1}^{4}(N)]-2\sigma^{2}\mathsf{E}[\tilde{\zeta}_{1}^{2}(N)]+\sigma^{4} \nonumber\\
%&=&\mathsf{E}\big[\mathsf{E}(\zeta_{1}^{4}(N)|X_{1}=x)]-2\sigma^{2}\mathsf{E}\big[\mathsf{E}(\zeta_{1}^{2}(N)|X_{1}=x)]+\sigma^{4}\nonumber\\
%&=&\mathsf{E}(\log^{4}\xi_{N,x})-2\sigma^{2}\mathsf{E}(\log^{2}\xi_{N,x})+\sigma^{4}
%\end{eqnarray}
\begin{eqnarray}\label{zeta tilde power 4}
\mathsf{E}\big[ \tilde{\zeta}_{1}^{4}(N)\big |X_{1}=x\big]%&=&\mathsf{E}\big[\zeta_{1}^{4}(N)-2\sigma^{2}\zeta_{1}^{2}(N)+\sigma^{4}\big]\nonumber\\
%%&=&\mathsf{E}\big[\zeta_{1}^{4}(N)|X_{1}=x\big]-2\sigma^{2}\mathsf{E}\big[\zeta_{1}^{2}(N)|X_{1}=x\big]+\sigma^{4}\nonumber\\
&=&\mathsf{E}(\log^{4}\xi_{N,x})-\dfrac{\pi^{2}}{3} \mathsf{E}(\log^{2}\xi_{N,x})+\dfrac{\pi^{4}}{36} \,,
\end{eqnarray}
where $ \log\xi_{N,x} $ is defined in Eq.(\ref{xi limit}). By the \text{Law of total expectation}, we have 
\begin{eqnarray}
\mathsf{E}\big[ \tilde{\zeta}_{1}^{4}(N)\big]=\mathsf{E}\Big\{\mathsf{E}\big[ \tilde{\zeta}_{1}^{4}(N)\big |X_{1}=x\big]\Big\}%&=&\int_{A}\mathsf{E}\big[ \tilde{\zeta}_{1}^{4}(N)\big |X_{1}=x\big] f(x)dx \nonumber\\
&=& \int_{\mathcal{A}}\mathsf{E}\big(\log^{4}\xi_{N,x}\big) f(x)dx \nonumber\\
&\leq & \max\Big\{\mathsf{E}\big(\log^{4}\xi_{N,x}\big)\Big\}\int_{\mathcal{A}}f(x)dx \nonumber\\
&=& \max_{x\in\mathcal{A}}\Big\{\mathsf{E}\big(\log^{4}\xi_{N,x}\big)\Big\}\,.
\end{eqnarray}
We recall that the random variable $ \xi_{x} $ defined in Eq.(\ref{xi limit}) has an exponential distribution with the mean $ e^{\gamma}/f(x)=\lambda^{-1}>0 $, and hence the random variable $ \eta=\log\xi_{x}  $ follows a Gumbel distribution with location parameter $\alpha=-\log\lambda $ and the scale parameter $ \beta=1 $, and hence
\begin{align}
\mathsf{E}\big(\eta^{2}\big)&=\mathsf{E}(\log^{2}\xi_{x})=\log^{2} f(x)+\frac{\pi^{2}}{6},          \\
\text{where}\quad \frac{\pi^{2}}{6}&=\int^{\infty}_{0}\log^{2}t\, e^{-t}dt - \Big(\int^{\infty}_{0}\log t\, e^{-t}dt\Big)^{2} \,.
\end{align}
It follows from the properties of Gumbel distribution that $ \mathsf{E}(\log^{n}\xi_{x})<\infty $ for $ n=1,2,3,4 $. It is sufficient to show that $ \mathsf{E}[\zeta_{1}^{4}(N)] <\infty $ or $ \mathsf{E}(\log^{4}\xi_{N,x}) <\infty $, and to prove that
\begin{align}
\lim\limits_{N\rightarrow\infty}\mathsf{E}(\log^{4}\xi_{N,x})=\mathsf{E}(\log^{4}\xi_{x}) <\infty    \nonumber
\end{align}
by using two steps:
\begin{enumerate}[(1)]
\item $ \log^{4}\xi_{N,x}\xrightarrow[\text{}]{\text{law}}\,\log^{4}\xi_{x}\,,\quad x\in\mathcal{A}\,,\,N\rightarrow \infty $.
\item family $ \big\{\log^{4}\xi_{N,x}\big\}_{N\geq N_{0}(x)} $ is uniformly integrable.
\end{enumerate}
Note that convergence in law follows from Eq.(\ref{xi limit}), since
\begin{eqnarray}
\log^{4}\xi_{N,x} \xrightarrow[\text{}]{law}\,\log^{4}\xi_{x}\,,\quad N\rightarrow \infty \,.
\end{eqnarray}
From the de la Valle Poussin theorem on the uniform integrability of $ \big\{\big(\log\xi_{N,x}\big)^{4}\big\}_{N\geq N_{0}(x)} $ 
for step (2), we need to prove the following: for almost all $ x\in \mathcal{S}(f) $, there is a positive $ C_{0}(x) $  and $ N_{0}(x)\geq 1 $, such that
\begin{eqnarray}\label{uniform integrability variance power 4}
\sup\limits_{N\geq N_{0}(x)} \mathsf{E}G(\mid \log^{4}\xi_{N,x}\mid)\leq C_{0}(x)<\infty  \,,
\end{eqnarray}
where $ G(t) $ is defined in Eq.(\ref{G(t)}) and is an monotonically increasing function on $ [0,\infty) $, such that $ {G(t)}/{t} \rightarrow \infty $ as $ t\rightarrow \infty $.
%%%%%%%%%%%%%%%%%%%%%%%%%%%%%%%%%%%%%%%%%%%%%%%%%%%%%%%%%%%%%%%%%%%%%%%%%%%%%%%%%%%%%%%%%%%%%%%%%%%%%%%%%%%%%%%%%%%%%%%%%%%%%%%%%%%%
%\subsection{Step 3: Proof of the validity of Eq.(\ref{uniform integrability variance power 4})}
%%It is worth to remark that following two statement and the proof is also provided.
\begin{lemma} \label{lemma power 3}
Let $ F(u) $, $ u\in\mathbb{R} $ be a cumulative distribution function such that $ F(0)=0 $. Then
\begin{eqnarray}
&&\hspace{-1.5cm}(1)\,\,\int\limits_{(0,\frac{1}{e}]}(-\log^{3} u)\,\log (-\log u) d F(u)=3\int\limits_{(0,\frac{1}{e}]}F(u) \dfrac{\log^{2} u}{u}\bigg[\log(-\log u)+\dfrac{1}{3}\bigg]du \,,\\
&&\hspace{-1.5cm}(2)\,\,\int\limits_{[e,\infty)}(\log^{3} u)\,\log (\log u) d F(u)=3\int\limits_{[e,\infty)}[1-F(u)] \dfrac{\log^{2} u}{u}\bigg[\log(\log u)+\dfrac{1}{3}\bigg]du  \,.
\end{eqnarray}
\end{lemma}
\begin{proof}[Proof of Lemma \ref{lemma power 3}] Provided in the Appendix A.1.
\end{proof}
\begin{lemma}\label{lemma power 4}
For $ u\in\mathbb{R} $, let $ F(u) $ be a cumulative distribution function and $ F(0)=0 $. Then
\begin{eqnarray}
&&\hspace{-1.5cm}(1)\,\,\int\limits_{(0,\frac{1}{e}]}(\log^{4} u)\,\log (-\log u) d F(u)=4\int\limits_{(0,\frac{1}{e}]}F(u) \dfrac{-\log^{3} u}{u}\bigg[\log(-\log u)+\dfrac{1}{4}\bigg]du \,,\\
&&\hspace{-1.5cm}(2)\,\,\int\limits_{[e,\infty)}(\log^{4} u)\,\log (\log u) d F(u)=4\int\limits_{[e,\infty)}[1-F(u)] \dfrac{\log^{3} u}{u}\bigg[\log(\log u)+\dfrac{1}{4}\bigg]du  \,.
\end{eqnarray}
\end{lemma}
\begin{proof}[Proof of Lemma \ref{lemma power 4}] Provided in the Appendix A.1.
\end{proof}
%%The proof \textit{Lemma} \ref{lemma power 4} of is analogous to the proof of \textit{Lemma} \ref{lemma power 3}.%and thus we skipped the detailed proof here.\\
For $ x\in\mathbb{R}^{d} $ and $ N\geq 2 $, we use the \textit{Lemma} \ref{lemma power 4} %into $ \mathsf{E}G(\mid\big(\log\xi_{N,x}\big)^{4}\mid) $ 
to prove (\ref{uniform integrability variance power 4}). Note that $ \forall u\in(\frac{1}{e},e] $, the function $ \mid\log u\mid\leq 1 $, and consequently $G(\mid\log u\mid)=0 $, with $ G(t) $ defined in Eq.(\ref{G(t)}).
\begin{eqnarray}
&&\hspace{-.3cm}\mathsf{E}G(\mid\log^{4}\xi_{N,x}\mid)\nonumber\\
&&\hspace{-0.9cm}=\int\limits_{0}^{\infty}\mid \log^{4}u\mid\cdot\log\big(\mid \log^{4}u\mid\big)dF_{N,x}(u) \nonumber\\
&&\hspace{-0.9cm}=\int\limits_{0}^{\frac{1}{e}} \log^{4} u\cdot\log(-\log u)^{4}dF_{N,x}(u)+\int\limits_{e}^{\infty} (\log^{4} u)\cdot\log\big(\log^{4} u)dF_{N,x}(u)\nonumber\\
&&\hspace{-0.9cm}=16\underbrace{\int\limits_{(0,\frac{1}{e}]}F_{N,x}(u) \dfrac{-\log^{3} u}{u}\bigg[\log(-\log u)+\dfrac{1}{4}\bigg]du}_{:=I_{1}(N,x)}+16\underbrace{\int\limits_{[e,\infty)}[1-F_{N,x}(u)] \dfrac{\log^{3} u}{u}\bigg[\log(\log u)+\dfrac{1}{4}\bigg]du}_{:=I_{2}(N,x)} \nonumber\\
&&\hspace{-0.9cm}=16\big[I_{1}(N,x)+I_{2}(N,x)\big]\,. \label{EG log4 }
\end{eqnarray}
%%%%%%%%%%%%%%%%%%%%%%%%%%%%%%%%%%%%%%%%%%%%%%%%%%%%%%%%%%%%%%%%%%%%%%%%%%%%%%%%%%%%%%%%%%%%%%%%%%%%%%%%%%%%%%%%%%%%%%%%%%%%%%%%%%%%%%%%%%%%
%\subsubsection{Proof of $ I_{1}(N,x)<\infty $}\label{section I1}
\subsubsection{Finiteness of $ I_{1}(N,x) $}\label{section I1}
%%%%%%%%%%%%%%%%%%%%%%%%%%%%%%%%%%%%%%%%%%%%%%%%%%%%%%%%%%%%%%%%%%%%%%%%%%%%%%%%%%%%%%%%%%%%%%%%%%%%%%%%%%%%%%%%%%%%%%%%%%%%%%%%%%%%%%%%%%%%
We introduce the cumulative distribution function:
\begin{eqnarray}
%\hspace{-.8cm} F_{N,x}(u)&=&\mathsf{P}(\xi_{N,x}\leq u)\,,\quad \xi_{N,x}=(N-1)V_{d}\tilde{\gamma}\min\limits_{j=2,\ldots ,N}\rho^{d}(x, X_{j})\,,\,\,i=1,\ldots, N \label{F_Nx(u)}\\ 
\hspace{-.8cm} F_{N,x}(u)&=&\mathsf{P}(\xi_{N,x}\leq u)%\,,\quad \xi_{N,x}=(N-1)V_{d}\tilde{\gamma}\min\limits_{j=2,\ldots ,N}\rho^{d}(x, X_{j})\,,\,\,i=1,\ldots, N 
=1-\mathsf{P}(\xi_{N,x}> u)
\label{F_Nx(u)} \nonumber\\
%&=&1-\mathsf{P}(\xi_{N,x}> u) \nonumber\\
&=&1-\mathsf{P}\Big(\min\limits_{j=2,\ldots ,N}\rho^{d}(x, X_{j})>r_{N}(u)\Big) %\,,\quad r_{N}(u)=\bigg[\dfrac{u}{(N-1)V_{d}\tilde{\gamma}}\bigg]^{1/d}
\nonumber\\
&=&1-\big[1-\mathsf{P}\big(y\in B(r_{N}(u)\big)\big]^{N-1} \nonumber\\
%&=&1-\bigg[1-\underbrace{\int_{B\big(x,r_{N}(u)\big)}f(y)dy}_{:=\mathsf{P}_{N,x}(u)}\bigg]^{N-1} \nonumber\\
&:=&1-\big[1-\mathsf{P}_{N,x}(u)\big]^{N-1} \,,\label{F_Nx(u) 2}
\end{eqnarray}
with $ \xi_{N,x} $ defined in Eq.(\ref{xi limit}) and
\begin{align} \label{rNu e PNx}
r_{N}(u)=\bigg[\dfrac{u}{(N-1)V_{d}\,e^{\gamma}}\bigg]^{1/d}\,,\,\,\text{and}\,\,\,\mathsf{P}_{N,x}(u)=\int_{B\big(x,r_{N}(u)\big)}f(y)dy\,.
\end{align}
Clearly, $\mathsf{P}_{N,x}(u)\in [0,1] $. For $ u\in(0, 1/e] $, one has $ r_{N}(u)\in(0, R_{1}] $ with %$ R_{1}= r_{N}(1/e)=\Big[\frac{1}{(N-1) V_{d}\,e^{(\gamma+1)}}\Big]^{1/d}$.
\begin{align}\label{R_1}
R_{1}= r_{N}(1/e)=\Big[\frac{1}{(N-1) V_{d}\,e^{(\gamma+1)}}\Big]^{1/d}\,.
\end{align}
%%$ r_{N}(u)>0$ for $ u\in(0, 1/e] $ . We define
%%\begin{eqnarray}\label{r_N(u) with N1}
%%R_{1}：=\max\limits_{u\in(0, 1/e]}\big\{ r_{N}(u)\big\}=\bigg[\dfrac{1}{(N-1) V_{d}\,e^{(\gamma+1)}}\bigg]^{1/d}
%%\end{eqnarray}
%%From Eq.(\ref{r_N(u) with N1}), 
We rearrange $ N $ as a function of $ R_{1} $: $ N(R_{1})=\frac{1}{R_{1}^{d}V_{d}e^{(\gamma+1)}}+1 $. Thus, we define
\begin{eqnarray}\label{N_1}
N_{1}=N_{1}(R_{1})=\bigg\lceil \frac{1}{R_{1}^{d}V_{d}e^{(\gamma+1)}} \bigg\rceil+1 \,,
\end{eqnarray}
where $\lceil t\rceil =\min\big\{n\in\mathbb{Z}: n\geq t\big\} $ denotes the smallest integer more than or equal to $ t $. 
${N}_{1} $ does not depend on $ x $.
%%\begin{eqnarray}\label{R_1}
%%N_{1}:=N(R_{1})=\dfrac{1}{R_{1}^{d}V_{d}e^{(\gamma+1)}}+1
%%\end{eqnarray}
%%Note that $ \mathsf{P}_{N,x}(u)\in [0,1]$ and 
%%where $ \xi_{N,x} $ is defined in Eq.(\ref{xi limit}) and $ \mathsf{P}_{N,x}(u):=\int_{B\big(x,r_{N}(u)\big)}f(y)dy $. We note that $r_{N}(u)  $ is a monotonically increasing function of $ u $. Therefore, for $ u\in(0, 1/e] $, we have
%%\begin{eqnarray}\label{r_N(u) with N1}
%%r_{N}(u)=\bigg[\dfrac{u}{(N-1)V_{d}\tilde{\gamma}}\bigg]^{1/d}\leq \bigg[\dfrac{1}{(N-1)e V_{d}\,e^{\gamma}}\bigg]^{1/d}:=R_{1}
%%\end{eqnarray}
%%From the right-hand side of Eq.(\ref{r_N(u) with N1}), we derive $ N $ as a function of $ R_{1} $ as follows.
Note that both $r_{N}(u)  $ and $R_{1} $ decrease as $ N $ increases, and for $ N\geq N_{1}  $, we have %$r_{N}(u)\leq R_{1}(N)\leq R_{1}(N_{1}) $, according to Eq.(\ref{r_N(u) with N1}) and monotonicity. For $ x\in\mathbb{R}^{d} $ have
\begin{eqnarray}\label{M_f(x,R_1)}
\dfrac{\mathsf{P}_{N,x}(u)}{\mid B\big(x,r_{N}(u)\big)\mid}=\dfrac{\int_{B\big(x,r_{N}(u)\big)}f(y)dy}{r_{N}^{d}(u)V_{d}}\leq\sup\limits_{r\in(0,R_{1}]}\dfrac{\int_{B\big(x,r_{N}(u)\big)}f(y)dy}{r^{d}V_{d}}:=M_{f}(x,R_{1}) \,.
\end{eqnarray}
In Appendix A.2, we show that $ 1-(1-x)^{N}\leq (Nx)^{\varepsilon_{1}} $ for any $ x\in[0,1] $, $ \varepsilon_{1}\in(0,1] $ and $ N \geq 1 $. Using the result in Eq.(\ref{F_Nx(u) 2}), we have
%%We define $ \varepsilon_{1} \in(0,1] $, 
%\end{proof}
\begin{eqnarray}
F_{N,x}(u)&\leq & \big[(N-1)\mathsf{P}_{N,x}(u)\big]^{\varepsilon_{1}} ,\quad \varepsilon_{1} \in(0,1] \nonumber\\
%&\leq & \big[(N-1)r_{N}^{d}(u)V_{d}M_{f}(x,R_{1})\big]^{\varepsilon_{1}} \nonumber\\
&=& \dfrac{M_{f}^{\varepsilon_{1}}(x,R_{1})}{e^{\gamma\varepsilon_{1}}}u^{\varepsilon_{1}}\,.  \label{F_Nx(u) 3}
\end{eqnarray}
Substituting Eq.(\ref{F_Nx(u) 3}) into $ I_{1}(N,x) $ in (\ref{EG log4 }), we obtain
\begin{eqnarray} \label{I_1(N,x)}
I_{1}(N,x) \leq \dfrac{M_{f}^{\varepsilon_{1}}(x,R_{1})}{e^{\gamma\varepsilon_{1}}}\,L(\varepsilon_{1})\,,
\end{eqnarray}
with 
\begin{eqnarray}\label{Lv1}
L(\varepsilon_{1})&=&\int\limits_{(0,\frac{1}{e}]}\frac{(-\log u)^{3}\big[\log(-\log u)+\frac{1}{4}\big]}{u^{1-\varepsilon_{1}}}du %%\nonumber\\
%%&=&
=\int\limits_{[1,\infty)}\dfrac{v^{3}\big(\log v+\frac{1}{4}\big)}{e^{\varepsilon_{1}v}}dv<\infty.%\,,\,\,\text{with}\,v=-\log u \nonumber\\
%&<&\infty \,.
\end{eqnarray}
We conclude $ L(\varepsilon_{1})<\infty $ since for $ \varepsilon_{1} \in(0,1] $ that
\begin{eqnarray}\label{Lv1 2}
\dfrac{v^{3}\big(\log v+\frac{1}{4}\big)}{e^{\varepsilon_{1}v}}\xrightarrow[\text{}]{}\,0\,\,,\quad \text{as}\,\, v\rightarrow \infty\,.%\,\text{and}\,\,\varepsilon_{1} \in(0,1].
\end{eqnarray}
Eq.(\ref{Lv1 2}) holds as the denominator (exponential function) grows much faster than the numerator (power and logarithm functions). %when $ v $ increases. 
Substituting Eq.(\ref{Lv1}) into (\ref{I_1(N,x)}), we obtain
\begin{eqnarray}\label{I_1 bounded}
I_{1}(N,x) <\infty\,.
\end{eqnarray}
%%Note that we conclude $L(\varepsilon_{1})<\infty  $ as for $ \varepsilon_{1} \in(0,1],\,\frac{v^{3}\big(\log v+\frac{1}{4}\big)}{e^{\varepsilon_{1}v}}\xrightarrow[\text{}]{v\rightarrow \infty\,} 0 $, which obviously holds as the denominator  grows much faster than the numerator (power and logarithm functions).
%%%%%%%%%%%%%%%%%%%%%%%%%%%%%%%%%%%%%%%%%%%%%%%%%%%%%%%%%%%%%%%%%%%%%%
%\subsubsection{Proof of $ I_{2}(N,x)<\infty $}\label{section I2}
\subsubsection{Finiteness of $ I_{2}(N,x)$}\label{section I2}
%%%%%%%%%%%%%%%%%%%%%%%%%%%%%%%%%%%%%%%%%%%%%%%%%%%%%%%%%%%%%%%%%%%%%%%%%%%%%%%%%%%%%%%%%%%%%%%%%%%%%%%%%%%%%%%%%%%%%%%%%%%%%%%%%%%%%%%%%%%
We split integral $ \int_{[e,\infty)} $ into $\int_{[e,\sqrt{N-1}]}+\int_{(\sqrt{N-1},\infty)}$ as follows.
\begin{eqnarray}\label{I_2(N,x) 1}
I_{2}(N,x)
%&=&\Bigg(\int\limits_{[e,\sqrt{N-1}]}+\int\limits_{(\sqrt{N-1},\infty)}\Bigg)[1-F_{N,x}(u)] \dfrac{\log^{3} u}{u}\bigg[\log(\log u)+\frac{1}{4}\bigg]du \nonumber\\
&:=& J_{1}(N,x)+J_{2}(N,x)
\end{eqnarray}
with
\begin{eqnarray}
J_{1}(N,x)&=&\int_{[e,\sqrt{N-1}]}[1-F_{N,x}(u)] \dfrac{\log^{3} u}{u}\big(\log\log u+\tfrac{1}{4}\big)du\,, \label{J_1} \\
J_{2}(N,x)&=&\int_{(\sqrt{N-1},\infty)}[1-F_{N,x}(u)] \dfrac{\log^{3} u}{u}\big(\log\log u+\tfrac{1}{4}\big)du \,.\label{J_2}
\end{eqnarray}
%%%%%%%%%%%%%%%%%%%%%%%%%%%%%%%%%%%%%%%%%%%%%%%%%%%%%%%%%%%%%%%%%%%%%%
%%%%%%%%%%%%%%%%%%%%%%%%%%%%%%%%%%%%%%%%%%%%%%%%%%%%%%%%%%%%%%%%%%%%%%
%\paragraph{Proof of $J_{1}(N,x)<\infty  $}
\paragraph{Finiteness of $J_{1}(N,x)$}
%%%%%%%%%%%%%%%%%%%%%%%%%%%%%%%%%%%%%%%%%%%%%%%%%%%%%%%%%%%%%%%%%%%%%%
%%%%%%%%%%%%%%%%%%%%%%%%%%%%%%%%%%%%%%%%%%%%%%%%%%%%%%%%%%%%%%%%%%%%%%
For $ u\in[e,\sqrt{N-1}] $, one has $ r_{N}(u)\in[r_{N}(e), R_{2}] $ with %$ R_{2}= r_{N}(\sqrt{N-1})=\Big[\frac{1}{\sqrt{N-1} V_{d}e^{\gamma}}\Big]^{1/d}$. 
\begin{align}\label{R_2}
R_{2}= r_{N}(\sqrt{N-1})=\Big[\frac{1}{\sqrt{N-1} V_{d}e^{\gamma}}\Big]^{1/d}\,.
\end{align}
Clearly, $N$ can be rearranged as a function of $ R_{2} $: $ N(R_{2})=\frac{1}{({R_{2}^{d}V_{d}e^{\gamma}})^{2}}+1 $. Note that the integral interval of ${[e,\sqrt{N-1}]} $ implies $ e\leq \sqrt{N-1} $ or $N\geq \big\lceil 1+e^{2} \big\rceil=9  $.
%%For $ u\in[e,\sqrt{N-1}] $, we have
%%\begin{eqnarray}\label{R2}
%% r_{N}(u)=\bigg[\dfrac{u}{(N-1)V_{d}e^{\gamma}}\bigg]^{1/d}\leq  \bigg[\dfrac{1}{\sqrt{N-1} V_{d}e^{\gamma}}\bigg]^{1/d}:=R_{2}
%%\end{eqnarray}
%%From the second half of Eq.(\ref{R2}), we derive $ N $ as a function of $ R_{2} $ and denote it as $ N_{2} $ as follows
%%\begin{eqnarray}
%%N_{2}=1+\dfrac{1}{({R_{2}^{d}V_{d}\tilde{\gamma}})^{2}}
%%\end{eqnarray}
%Let us define $  N_{2}:=N_{2}(u)\Big|_{u\in[e,\sqrt{N-1}]}$ as the integer part of $ N $, i.e. $ N_{2} $=$ \lfloor N\rfloor =\max\{n\in \mathbb {Z} \mid n\leq N\} $, where $\mathbb {Z}  $ represents the set of integers, and we have
\begin{align}\label{N_2 R_2}
N_{2}=N_2(R_{2})=\max\bigg\{\bigg\lceil \frac{1}{({R_{2}^{d}V_{d}e^{\gamma}})^{2}}\bigg\rceil+1,\, 9\bigg\}\,.
\end{align} 
Clearly, ${N}_{2} $ does not depend on $ x $, and 
%%For $ N\in \mathbb {N} $, we set $$ N\geq \max\{N_{2}, 10\} $$.
%%\begin{eqnarray}\label{N>N2>10}
%%10\leq N_{2}\leq N <\infty
%%\end{eqnarray}
we consider $ N\geq {N}_{2} $ in this section. For a density $ f $ and $ R_{2}>0 $, we define $\mathcal{D}_{f}(R_{2})=\{x\in \mathcal{S}(f): m_{f}(x,R_{2})>0\}  $ with $ m_{f}(x,R_{2}) $ defined in Eq.(\ref{M_f e m_f}). We have $ \mu\big(\mathcal{S}(f)\setminus \mathcal{D}_{f}(R_{2})\big)=0 $, see, Lemma 3.2 of Bulinski and Dimitrov (2019).
In Appendix A.3, we show that
\begin{eqnarray}\label{1-F_Nx(u) 2}
1-F_{N,x}(u) &\leq & \bigg[\frac{u}{e^{\gamma}}\,m_{f}(x, R_{2})\bigg]^{-\varepsilon}\,,\quad \varepsilon\in(0,e]\,,
\end{eqnarray}
and
%for $ x\in\mathcal{D}_{f}(R_{2})$, $N\geq N_{2}(R_{2})>10  $ and $ u\in [e,\sqrt{N-1}) $. 
%Substituting the Eq.(\ref{1-F_Nx(u) 2}) into Eq.(\ref{J_1}), we obtain:
\begin{eqnarray}
J_{1}(N,x)&\leq &  \dfrac{e^{\gamma\varepsilon}}{m_{f}^{\varepsilon}(x,R_{2})}\,\int\limits_{[e,\sqrt{N-1}]}\dfrac{\log^{3} u\big(\log\log u+\tfrac{1}{4}\big)}{u^{1+\varepsilon}}du  \label{J_1 part2 eq1} \nonumber\\
\hspace{-3cm} &\leq &  \dfrac{e^{\gamma\varepsilon}}{m_{f}^{\varepsilon}(x,R_{2})}\,\int\limits_{[e,\infty]}\dfrac{\log^{3} u\big(\log\log u+\tfrac{1}{4}\big)}{u^{1+\varepsilon}}du \label{J_1 part2 eq2} \nonumber\\
%&=& \dfrac{\tilde{\gamma}^{\varepsilon}}{m_{f}^{\varepsilon}(x,R_{2})}\int\limits_{[1,\infty)}\dfrac{v^{3}\Big(\log v+\frac{1}{4}\Big)}{e^{\varepsilon v}}dv\,,\quad\text{with}\,v=\log u \nonumber\\
&=& \dfrac{e^{\gamma\varepsilon}}{m_{f}^{\varepsilon}(x,R_{2})}\,L(\varepsilon)<\infty \,. \label{J_1 part2 eq3}
\end{eqnarray}
The second equality holds since $ \frac{\log^{3} u\big(\log\log u+\tfrac{1}{4}\big)}{u^{1+\varepsilon}} >0 $ for $ u>e $. $ L(\cdot) $ is defined in Eq.(\ref{Lv1}) (its argument is $ \varepsilon $ instead of $ \varepsilon_{1} $), and hence,
%\begin{eqnarray}\label{Lv}
%%$  L(\varepsilon)=\int\limits_{[1,\infty)}\frac{v^{3}\Big(\log v+\frac{1}{4}\Big)}{e^{\varepsilon_{1}v}}dv$.
%\end{eqnarray}
we have $L(\varepsilon)<\infty  $; thus, $J_{1}(N,x)<\infty  $.
%%%%%%%%%%%%%%%%%%%%%%%%%%%%%%%%%%%%%%%%%%%%%%%%%%%%%%%%%%%%%%%%%%%%%%
%%%%%%%%%%%%%%%%%%%%%%%%%%%%%%%%%%%%%%%%%%%%%%%%%%%%%%%%%%%%%%%%%%%%%%
\paragraph{Finiteness of $J_{2}(N,x) $}
%%%%%%%%%%%%%%%%%%%%%%%%%%%%%%%%%%%%%%%%%%%%%%%%%%%%%%%%%%%%%%%%%%%%%%
%%%%%%%%%%%%%%%%%%%%%%%%%%%%%%%%%%%%%%%%%%%%%%%%%%%%%%%%%%%%%%%%%%%%%%
%Using the Eq.(\ref{1-F_Nx(u) 1}) and the inequality $ 1-\mathsf{P}_{N,x}(u)\leq 1-\mathsf{P}_{N,x}(\sqrt{N-1}) $ for $ u\geq N-1 $, we have the following:
Since $ 1-\mathsf{P}_{N,x}(u)\leq 1-\mathsf{P}_{N,x}(\sqrt{N-1}) $ for $ u\geq N-1 $, we have:
\begin{eqnarray}
1-F_{N,x}(u)=\big[1-\mathsf{P}_{N,x}(u)\big]^{N-1}%&=&\big[1-\mathsf{P}_{N,x}(u)\big]^{N-2}\big[1-\mathsf{P}_{N,x}(u)\big]\nonumber\\
&\leq & \big[1-\mathsf{P}_{N,x}(u)\big]\big[1-\mathsf{P}_{N,x}(\sqrt{N-1})\big]^{N-2} \,. \label{1-F_Nx(u) v2 1}
\end{eqnarray}
From Eqs.(\ref{J_2}) and (\ref{1-F_Nx(u) v2 1}), we obtain
%Substituting Eq.(\ref{1-F_Nx(u) v2 1}) into $J_{2}(N,x)$ in Eq.(\ref{J_2}), we obtain
\begin{eqnarray}\label{J_2 2}
J_{2}(N,x)&\leq & %\big[1-\mathsf{P}_{N,x}(\sqrt{N-1})\big]^{N-2}\int_{(\sqrt{N-1},\infty)}\big[1-\mathsf{P}_{N,x}(u)\big] \dfrac{(\log u)^{3}\big[\log(\log u)+\frac{1}{4}\big]}{u}du\nonumber\\
%&=& 
J_{2.1}(N,x)\cdot J_{2.2}(N,x)
\end{eqnarray}
with
\begin{eqnarray}
J_{2.1}(N,x)&:=&\big[1-\mathsf{P}_{N,x}(\sqrt{N-1})\big]^{N-2}\,,\label{J_2 2.1 1}\\
J_{2.2}(N,x)&:=&\int_{(\sqrt{N-1},\infty)}\big[1-\mathsf{P}_{N,x}(u)\big] \dfrac{\log^{3}u\,\big(\log\log u+\frac{1}{4}\big)}{u}du \,. \label{J_2 2.2 1}
\end{eqnarray}
%%%%%%%%%%%%%%%%%%%%%%%%%%%%%%%%%%%%%%%%%%%%%%%%%%%%%%%%%%%%%%%%%%%%%%
%%%%%%%%%%%%%%%%%%%%%%%%%%%%%%%%%%%%%%%%%%%%%%%%%%%%%%%%%%%%%%%%%%%%%%
\subparagraph{Investigation of $J_{2.1}(N,x)$}\label{subparagraph J_2 2.1}
%%%%%%%%%%%%%%%%%%%%%%%%%%%%%%%%%%%%%%%%%%%%%%%%%%%%%%%%%%%%%%%%%%%%%%
%%%%%%%%%%%%%%%%%%%%%%%%%%%%%%%%%%%%%%%%%%%%%%%%%%%%%%%%%%%%%%%%%%%%%%
%From Eqs.(\ref{(1-t)N<e-tN}) and (\ref{J_2 2.1 1}), we obtain:
Note that $(1-t)^{N}\leq e^{-t\,N} $ for $t\in[0,1] $ and $ N\geq 1 $. We deduce from Eq.(\ref{J_2 2.1 1}) that
\begin{eqnarray}\label{J_2 2.1 4}
J_{2.1}(N,x)
&\leq & e^{-(N-2)\,\mathsf{P}_{N,x}(\sqrt{N-1})} \label{J_2 2.1 2} \nonumber\\
%&=& \exp\left\{-(N-2)\dfrac{\mathsf{P}_{N,x}(\sqrt{N-1})}{\mid B\big(x,r_{N}(\sqrt{N-1})\big)\mid}\,\big| B\big(x,r_{N}(\sqrt{N-1})\big)\big|\right\}               \nonumber\\
&=& \exp\left\{-\dfrac{N-2}{N-1}\dfrac{\sqrt{N-1}}{\tilde{\gamma}}\,\dfrac{\mathsf{P}_{N,x}(\sqrt{N-1})}{\mid B\big(x,r_{N}(\sqrt{N-1})\big)\mid}\right\}   \nonumber\\
&\leq & \exp{-\frac{1}{2}\frac{\sqrt{N-1}}{e^{\gamma}}\,m_{f}(x,R_{2})} \label{J_2 2.1 4} \nonumber\\
%&\leq & \bigg[\frac{\sqrt{N-1}}{2\tilde{\gamma}}\,m_{f}(x,R_{2})\bigg]^{-\varepsilon}\,,\qquad \varepsilon\in(0,e]  \label{J_2 2.1 5}\nonumber\\
%&=& 
&\leq & \dfrac{(2\,e^{\gamma})^{\varepsilon}}{m_{f}(x,R_{2})^{\varepsilon}}\dfrac{1}{(N-1)^{\varepsilon/2}}<\infty\,,\qquad \varepsilon\in(0,e]\,. \label{J_2 2.1 6}
\end{eqnarray}
%%with 
%%$ %t=
%%\frac{\sqrt{N-1}}{2e^{\gamma}}\,m_{f}(x,R_{2})\in(0,\infty) $. %in Eq.(\ref{J_2 2.1 5}). 
The detailed derivation and proofs are provided in Appendix A.4.
%, we used together with Eq.(\ref{e_-t < t_-delta}).
%%%%%%%%%%%%%%%%%%%%%%%%%%%%%%%%%%%%%%%%%%%%%%%%%%%%%%%%%%%%%%%%%%%%%%%%%%%%%%%%%%%%%%%%%%%%%%%%%%%%%%%%%%%%%%%%%%%%%%%%%%%%%%%%%%
\subparagraph{Investigation of $J_{2.2}(N,x)$}\label{subparagraph J_2 2.2}
%%%%%%%%%%%%%%%%%%%%%%%%%%%%%%%%%%%%%%%%%%%%%%%%%%%%%%%%%%%%%%%%%%%%%%%%%%%%%%%%%%%%%%%%%%%%%%%%%%%%%%%%%%%%%%%%%%%%%%%%%%%%%%%%%%%%%%%%%%%%
Let us define $ w={u}/{(N-1)}$ (or $ u=(N-1)w $). For $ N\geq 9 $ (or $ N\geq N_{2} $), one has
\begin{enumerate}[(i)]
\item $ u\in[\sqrt{N-1},\infty)\quad \Longleftrightarrow \quad w\in \Big[\frac{1}{\sqrt{N-1}},\infty\Big)\,,%\quad N>10 
$
\item $ r_{N}(u)=\bigg[\dfrac{u}{(N-1)V_{d}\tilde{\gamma}}\bigg]^{1/d} = \bigg[\dfrac{w}{V_{d}\tilde{\gamma}}\bigg]^{1/d}=r_{2}(w)\,,\quad u=(N-1)w $\,,
\item $ \mathsf{P}_{N,x}(u)= \mathsf{P}_{2,x}(w)$, based on (ii): $  r_{N}(u)=r_{2}(w) $ and $ \mathsf{P}_{N,x}(u)=\int_{B(x, r_{N}(u))}f(y)dy $,
\item $ 1-F_{N,x}(u)=1-F_{2,x}(w)=1-\mathsf{P}_{2,x}(w)$, see Eq.(\ref{F_Nx(u) 2}).
\end{enumerate} 
We substitute (i)-(iv) into Eq.(\ref{J_2 2.2 1}) and 
%From Eq.(\ref{J_2 2.2 1}) and (i)-(iv) into, we 
split the interval of integration as follows. %into $ (\frac{1}{\sqrt{N-1}},e] \cup  (e,\infty) $.
\begin{eqnarray}
&&\hspace{-1.4cm} J_{2.2}(N,x) %\nonumber\\
%&&\hspace{-2cm}
=\int_{(\frac{1}{\sqrt{N-1}},\infty)}\big[1-\mathsf{P}_{2,x}(w)\big] \mathbb{M}(N,w)\dfrac{dw}{w} \nonumber\\
%%&&\hspace{-2cm}=\Bigg(\int\limits_{(\frac{1}{\sqrt{N-1}},e]}+\int\limits_{(e,\infty)}\Bigg)\big[1-F_{2,x}(w)\big] \dfrac{(\log (N-1)w)^{3}\big[\log(\log (N-1)w)+\frac{1}{4}\big]}{w}dw  \label{J_2 2.2 2} \nonumber \\
&&\hspace{0.3cm}
:= J_{2.2.1}(N,x)+J_{2.2.2}(N,x) \label{J_2 2.2 3}
\end{eqnarray}
with $\mathbb{M}(N,w)=\big(\log^{3} [(N-1)w]\big)\big(\log\log [(N-1)w]+\frac{1}{4}\big) $ and
\begin{eqnarray}
\hspace{-1cm} J_{2.2.1}(N,x)&=&\int_{\big(\frac{1}{\sqrt{N-1}},e\big]}\big[1-F_{2,x}(w)\big] \mathbb{M}(N,w)\dfrac{dw}{w} \,,\label{J_2 2.2.1 eq1} \\
\hspace{-1cm} J_{2.2.2}(N,x)&=&\int_{(e,\infty)}\big[1-F_{2,x}(w)\big]\mathbb{M}(N,w) \dfrac{dw}{w}\,. \label{J_2 2.2.2 eq1}   
\end{eqnarray}
%%%%%%%%%%%%%%%%%%%%%%%%%%%%%%%%%%%%%%%%%%%%%%%%%%%%%%%%%%%%%%%%%%%%%%%%%%%%%%%%%%%%%%%%%%%%%%%%%%%%%%%%%%%%%%%%%%%%%%%%%%%%%%%%%%
%%\textbf{(I) Analytical treatment of $J_{2.2.1}(N,x)  $}\\
\textbf{(I) Investigation of $J_{2.2.1}(N,x)  $}\\
%%%%%%%%%%%%%%%%%%%%%%%%%%%%%%%%%%%%%%%%%%%%%%%%%%%%%%%%%%%%%%%%%%%%%%%%%%%%%%%%%%%%%%%%%%%%%%%%%%%%%%%%%%%%%%%%%%%%%%%%%%%%%%%%%%
%It is worth nothing that the cumulative distribution function $ F_{2,x}(w)\in[0,1]\,,\,\forall w\in[0,\infty)$ and thus we have $  1-F_{2,x}(w)$ is non-negative and satisfies
Note that $F_{2,x}(w)  $ is the cumulative distribution function, and thus,
\begin{eqnarray}\label{sup 1-f_2,x N eq1}
\sup\limits_{w\in\big(\frac{1}{\sqrt{N-1}},e\big]}\{1-F_{2,x}(w)\}\leq \sup\limits_{w\in[0,\infty)}\{1-F_{2,x}(w)\}= 1 \,.
\end{eqnarray}
From the monotonicity of logarithmic and power functions, we have
\begin{eqnarray}\label{sup Polynomials log(N-1)w eq1}
%%\hspace{-.9cm}\sup\limits_{w\in\big(\frac{1}{\sqrt{N-1}},e\big]}\hspace{-.6cm}\big\{ \log^{3} [(N-1)w]\big(\log\log [(N-1)w]+\tfrac{1}{4}\big)\big\}=\log^{3} [(N-1)e]\big(\log\log [(N-1)e]+\tfrac{1}{4}\big).
\sup\limits_{w\in\big(\frac{1}{\sqrt{N-1}},e\big]}\mathbb{M}(N,w)&=&\big(\log^{3} [(N-1)e]\big)\big(\log\log [(N-1)e]+\tfrac{1}{4}\big)   \nonumber\\
&:=&\mathbb{M}_{e}^{\ast}(N) \,.
\end{eqnarray}
Substituting Eq.(\ref{sup 1-f_2,x N eq1}) and (\ref{sup Polynomials log(N-1)w eq1}) into (\ref{J_2 2.2.1 eq1}), we obtain
\begin{eqnarray}
\hspace{-0.6cm} J_{2.2.1}(N,x)
&\hspace{0cm}\leq &\mathbb{M}_{e}^{\ast}(N)\int_{\big(\frac{1}{\sqrt{N-1}},e\big]} \dfrac{dw}{w} \nonumber\\
&=&\hspace{0cm} \mathbb{M}_{e}^{\ast}(N)\big[1+\tfrac{1}{2}\log (N-1)\big]\,.  \label{J_2 2.2.1 eq2 final}
\end{eqnarray}
%%%%%%%%%%%%%%%%%%%%%%%%%%%%%%%%%%%%%%%%%%%%%%%%%%%%%%%%%%%%%%%%%%%%%%%%%%%%%%%%%%%%%%%%%%%%%%%%%%%%%%%%%%%%%%%%%%%%%%%%%%%%%%%%%%%%%%%%%%%%
\textbf{(II) Investigation of $J_{2.2.2}(N,x)  $}\\
%%%%%%%%%%%%%%%%%%%%%%%%%%%%%%%%%%%%%%%%%%%%%%%%%%%%%%%%%%%%%%%%%%%%%%%%%%%%%%%%%%%%%%%%%%%%%%%%%%%%%%%%%%%%%%%%%%%%%%%%%%%%%%%%%%%%%%%%%%%%
We set $ \Delta$ an arbitrage small positive number, and split integral $ \int\limits_{(e,\infty)}$ into $\int\limits_{(e,e^{1+\Delta}]}+\int\limits_{(e^{1+\Delta},\infty)}$:
\begin{eqnarray}
J_{2.2.2}(N,x)%%&&=\int\limits_{(e,\infty)}\big[1-F_{2,x}(w)\big] \dfrac{(\log (N-1)w)^{3}\big[\log(\log (N-1)w)+\frac{1}{4}\big]}{w}dw \nonumber\\
%%&&\hspace{-3cm}= \bigg(\int\limits_{(e,e^{1+\Delta}]}+\int\limits_{(e^{1+\Delta},\infty)} \bigg)\dfrac{(\log (N-1)w)^{3}\big[\log(\log (N-1)w)+\frac{1}{4}\big]}{w}dw \nonumber\\
%%&&\hspace{-3cm}
=J_{2.2.2}^{(1)}(N,x)+J_{2.2.2}^{(2)}(N,x)  \label{J_2 2.2.2 eq2}
\end{eqnarray}
with
\begin{eqnarray}
&&\hspace{-1.6cm} J_{2.2.2}^{(1)}(N,x)=\int_{(e,e^{1+\Delta}]}[1-F_{2,x}(w)\big]\mathbb{M}(N,w) \dfrac{dw}{w} \,,\label{J_2 2.2.2 (1)}\\
&&\hspace{-1.6cm} J_{2.2.2}^{(2)}(N,x)=\int_{(e^{1+\Delta},\infty)}[1-F_{2,x}(w)\big]\mathbb{M}(N,w)\dfrac{dw}{w} \,. \label{J_2 2.2.2 (2)}
\end{eqnarray}
%%%%%%%%%%%%%%%%%%%%%%%%%%%%%%%%%%%%%%%%%%%%%%%%%%%%%%%%%%%%%%%%%%%%%%%%%%%%%%%%%%%%%%%%%%%%%%%%%%%%%%%%%%%%%%%%%%%%%%%%%%%%%%%%%%
%%\textbf{(II.1) Analytical treatment of $J_{2.2.2}^{(1)}(N,x) $}\\
\textbf{(II.1) Investigation of $J_{2.2.2}^{(1)}(N,x) $}\\
%%%%%%%%%%%%%%%%%%%%%%%%%%%%%%%%%%%%%%%%%%%%%%%%%%%%%%%%%%%%%%%%%%%%%%%%%%%%%%%%%%%%%%%%%%%%%%%%%%%%%%%%%%%%%%%%%%%%%%%%%%%%%%%%%%
From the monotonicity of logarithmic and power functions, we obtain %for $ w\in(e,e^{1+\Delta}] $ 
that
%%\begin{eqnarray}
%%0< \big\{\log^{3} [(N-1)w]\big(\log\log [(N-1)w]+\tfrac{1}{4}\big)\big\}\leq M_{J_{2.2.2}^{(1)}, P1}
%%\end{eqnarray}
%%with
\begin{eqnarray}
\sup\limits_{w\in(e,e^{1+\Delta}]}\mathbb{M}(N,w)  &=&\log^{3} [(N-1)e^{1+\Delta}]\big(\log\log [(N-1)e^{1+\Delta}]+\tfrac{1}{4}\big)   \nonumber\\
%%M_{J_{2.2.2}^{(1)}, P1}&:=&\sup\limits_{w\in(e,e^{1+\Delta}]}\big\{\log^{3} [(N-1)w]\big(\log\log [(N-1)w]+\tfrac{1}{4}\big)\big\}\nonumber\\
&:=&\mathbb{M}_{e^{1+\Delta}}^{\ast}(N)\,.    \label{sup J_2.2.2 (1) part1}
\end{eqnarray}
Substituting 
%$\sup\limits_{w\in[0,\infty)}\{1-F_{2,x}(w)\} =1 $ 
in Eqs.(\ref{sup 1-f_2,x N eq1}) and (\ref{sup J_2.2.2 (1) part1}) into Eq.(\ref{J_2 2.2.2 (1)}), we obtain
\begin{eqnarray}
 J_{2.2.2}^{(1)}(N,x)&\leq & %% M_{J_{2.2.2}^{(1)}, P1}
 \mathbb{M}_{e^{1+\Delta}}^{\ast}(N) \int\limits_{(e,e^{1+\Delta}]} \dfrac{1}{w}dw \nonumber\\
 &=& \log^{3} [(N-1)e^{1+\Delta}]\big(\log\log [(N-1)e^{1+\Delta}]+\tfrac{1}{4}\big)\,\Delta \nonumber \\
 &：=& J_{2.2.2}^{(1),\ast}(N)\,. \label{J_2 2.2.2 eq3 (1) final}
\end{eqnarray}
%%%%%%%%%%%%%%%%%%%%%%%%%%%%%%%%%%%%%%%%%%%%%%%%%%%%%%%%%%%%%%%%%%%%%%%%%%%%%%%%%%%%%%%%%%%%%%%%%%%%%%%%%%%%%%%%%%%%%%%%%%%%%%%%%%%%%%%%%%%%
%%\textbf{(II.2) Analytical treatment of $J_{2.2.2}^{(2)}(N,x)$}\\
\textbf{(II.2) Investigation of $J_{2.2.2}^{(2)}(N,x)$}\\
%%%%%%%%%%%%%%%%%%%%%%%%%%%%%%%%%%%%%%%%%%%%%%%%%%%%%%%%%%%%%%%%%%%%%%%%%%%%%%%%%%%%%%%%%%%%%%%%%%%%%%%%%%%%%%%%%%%%%%%%%%%%%%%%%%%%%%%%%%%%
We introduce the following \text{Lemma}, which can provide an intermediate result. 
\begin{lemma}\label{corollary K(Delta)}
For an arbitrage small number $ \Delta>0 $, and $ w\in[e^{1+\Delta},\infty) $, the ratio $ \frac{\log(1+\log w)}{\log(\log w)} $ is bounded and precisely, we have
\begin{eqnarray}
&&\dfrac{\log(1+\log w)}{\log(\log w)}\leq \frac{\log(2+\Delta)}{\log(1+\Delta)} :=\widetilde{\Delta} \,.\label{K(Delta)} 
\end{eqnarray}
$ \widetilde{\Delta} $ decreases with $ {\Delta} $, and
%Moreover, we have 
$1\leq \widetilde{\Delta} < \infty $. Equality holds if and only if $ \Delta\rightarrow \infty  $.
\end{lemma}
\begin{proof}[Proof of Lemma \ref{corollary K(Delta)}] Provided in the Appendix A.5.
\end{proof}
The \textit{Lemma} \ref{corollary K(Delta)} implies 
\begin{eqnarray}\label{J_2 2.2.2 (2) eq4}
 J_{2.2.2}^{(2)}(N,x)< \big[4\log\log (N-1)+\widetilde{\Delta}\big]\,%M_{J_{2.2.2}^{(2)},P1}
%% \widetilde{\mathbb{M}}_{}(N,x)
\Psi(N,x)\,,
\end{eqnarray}
%%with
%\begin{eqnarray}
%&&M_{J_{2.2.2}^{(2)},P1}=\int\limits_{(e^{1+\Delta},\infty)}\big[1-F_{2,x}(w)\big]\bigg[\log(\log w)+\frac{1}{4}\bigg] \dfrac{(\log (N-1)w)^{3}}{w}dw \label{M_J 2.2.2 (2) P1}
%\\
%&&M_{J_{2.2.2}^{(2)},P2}=\int\limits_{(e^{1+\Delta},\infty)}\big[1-F_{2,x}(w)\big] \dfrac{(\log (N-1)w)^{3}}{w}dw \label{M_J 2.2.2 (2) P2}
%\end{eqnarray}
%We expand the 1st part of the numerator of Eq.(\ref{J_2 2.2.2 (2)}) as follows.
%\begin{eqnarray}
%&&(\log (N-1)w)^{3}=(\log (N-1)+\log w)^{3} \nonumber\\
%&&\hspace{-.6cm}=\log^{3}w+3\log^{2}w\log(N-1)+3\log w\log^{2}(N-1)+\log^{3}(N-1)\label{log(N-2)w power3}
%\end{eqnarray}
%Substituting Eq.(\ref{log(N-2)w power3}) into Eq.(\ref{M_J 2.2.2 (2) P1}), we obtain
\begin{align} \label{M_J 2.2.2 (2) P1 eq2}
\hspace{-1cm}\text{with}\quad \Psi(N,x)=& \sum\limits_{\alpha=1}^{4}\Bigg\{\theta_{\alpha}\Big[\log^{(4-\alpha)}(N-1)\Big]\cdot \nonumber\\
& \int\limits_{(e^{1+\Delta},\infty)}\dfrac{\log^{(\alpha-1)} w}{w}\big[\log\log w+\tfrac{1}{4}\big] \big[1-F_{2,x}(w)\big]dw\Bigg\}
\end{align}
with $ \theta_{1}=\theta_{4}=1 $ and $ \theta_{2}=\theta_{3}=3 $.
%%\begin{eqnarray}\label{M_J 2.2.2 (2) P1 eq2}
%%M_{J_{2.2.2}^{(2)},P1}&=&\hspace{-.6cm}\int\limits_{(e^{1+\Delta},\infty)}\dfrac{\log^{3} w}{w}\bigg[\log(\log w)+\frac{1}{4}\bigg] \big[1-F_{2,x}(w)\big]dw \nonumber\\
%%&&+3\log(N-1)\int\limits_{(e^{1+\Delta},\infty)}\dfrac{\log^{2} w}{w}\bigg[\log(\log w)+\frac{1}{4}\bigg] \big[1-F_{2,x}(w)\big]dw \nonumber\\
%%&&+3\log^{2}(N-1)\int\limits_{(e^{1+\Delta},\infty)}\dfrac{\log w}{w}\bigg[\log(\log w)+\frac{1}{4}\bigg] \big[1-F_{2,x}(w)\big]dw \nonumber\\
%%&&+\log^{3}(N-1)\int\limits_{(e^{1+\Delta},\infty)}\dfrac{1}{w}\bigg[\log(\log w)+\frac{1}{4}\bigg] \big[1-F_{2,x}(w)\big]dw 
%%\end{eqnarray}
%Substituting the result (2) of \textit{Lemma} \ref{lemma power 3} and result (2) of \textit{Lemma} \ref{lemma power 4} into Eq.(\ref{M_J 2.2.2 (2) P1 eq2}), we obtain
By \textit{Lemma} \ref{lemma power 3} and \textit{Lemma} \ref{lemma power 4}, we have %the 1st, 2nd, 3rd and 4th part of RHS of Eq.(\ref{M_J 2.2.2 (2) P1 eq2}) satisfy: 
\begin{align}\label{EGlog power}
\int\limits_{(e^{1+\Delta},\infty)}\dfrac{\log^{(\alpha-1)} w}{w}\big[\log\log w+\tfrac{1}{4}\big] \big[1-F_{2,x}(w)\big]dw\leq \frac{R^{(\alpha)}(x)}{\alpha^{2}},\,\,\,\alpha=1,2,3,4.
\end{align}
%%\begin{align}
%%&(1)\quad \int\limits_{(e^{1+\Delta},\infty)}\dfrac{1}{w}\bigg[\log(\log w)+\frac{1}{4}\bigg] \big[1-F_{2,x}(w)\big]dw
%%\leq R^{(1)}(x)\label{EGlog power1} \\
%%&(2)\quad \int\limits_{(e^{1+\Delta},\infty)}\dfrac{\log^{} w}{w}\bigg[\log(\log w)+\frac{1}{4}\bigg] \big[1-F_{2,x}(w)\big]dw
%%\leq\frac{1}{4}R^{(2)}(x)\label{EGlog power2} \\
%%&(3)\quad \int\limits_{(e^{1+\Delta},\infty)}\dfrac{\log^{2}w}{w}\bigg[\log(\log w)+\frac{1}{4}\bigg] \big[1-F_{2,x}(w)\big]dw
%%\leq\frac{1}{9}R^{(3)}(x) \label{EGlog power3}\\
%%&(4)\quad \int\limits_{(e^{1+\Delta},\infty)}\dfrac{\log^{3}w}{w}\bigg[\log(\log w)+\frac{1}{4}\bigg] \big[1-F_{2,x}(w)\big]dw \leq \frac{1}{16}R^{(4)}(x)\label{EGlog power4} 
%%\end{align}
with
\begin{align}
R^{(\alpha)}(x):=\mathsf{E} G(\mid\log w\mid^{\alpha})=\int\limits_{\mathbb{R}^{d}}\mid\log w\mid^{\alpha}\log\mid\log w\mid^{\alpha}dF_{2,x}(w) \,. %,\quad \alpha=1,2,3,4.
\end{align}
%%The detailed deductions of Eqs.(\ref{J_2 2.2.2 (2) eq4})-(\ref{EGlog power}) and are given in Appendix A.1.5.
Note that $ u $ %is used to 
represents $ \xi_{N,x}$ in this section, and based on definition of $ w $, we have 
\begin{eqnarray} \label{w<xi_2,x}
w=\dfrac{\xi_{N,x}}{N-1} &=& V_{d}e^{\gamma}{\min_{j=2,\ldots,N}\rho^{d}(x, X_{j})}  \nonumber\\
&\leq &  V_{d}e^{\gamma}{\rho^{d}(x, X_{2})} =\xi_{2,x} \,.
\end{eqnarray}
Thus, for $ w\in (e,\infty) $, one has $ |\log w|\leq  |\xi_{2,x}| $ and $ \mathsf{E}G|\log^{\alpha} w|\leq \mathsf{E} G|\xi_{2,x}^{\alpha}|$, $\alpha\geq 0 $.
%%\begin{eqnarray}
%% \mathsf{E}G|\log w|^{\alpha}\leq \mathsf{E} G|\xi_{2,x}|^{\alpha}\,,\quad \alpha\geq 0
%%\end{eqnarray}
\begin{lemma}\label{G_a+b}
For $ x\in\mathbb{R}^{d} $ and $ \alpha=1,2,3,4 $, there exists constants $ a,b\geq 0 $, such that:
\begin{align}
G(\mid\log^{\alpha}\xi_{2,x}\mid)\leq a\,G(\mid\log^{\alpha}{\rho}(x,y)\mid)+b\,.
\end{align}
\end{lemma}
\begin{proof}[Proof of Lemma \ref{G_a+b}] provided in the Appendix A.6.
\end{proof}
The \textit{Lemma} \ref{G_a+b} implies that $  \mathsf{E}G(\mid\log^{\alpha}\xi_{2,x}(x,y)\mid)\leq a\,\mathsf{E}G(\mid\log^{\alpha}{\rho}(x,y)\mid)+b$\,. %, or
%\begin{align}\label{G_a+b exp}
%\int_{\mathbb{R}^{d}}G(\mid\log^{\alpha}\widetilde{\rho}(x,y)\mid)f(y)dy\leq a\int_{\mathbb{R}^{d}}G(\mid\log^{\alpha}{\rho}(x,y)\mid)f(y)dy+b
%\end{align}
Note that $K_{f,{\color{black} \alpha}}(\varepsilon_{0})<\infty $ is defined in Eq.(\ref{condition K}), thus %and from Eq.(\ref{G_a+b exp}), we can conclude 
\begin{align}\label{G_a+b bounded}
K_{f,{\color{black} \alpha}}(\varepsilon_{0}) <\infty\,\Longrightarrow \mathsf{E}G(\mid\log^{\alpha}{\rho}(x,y)\mid)<\infty\,.
\end{align}
Set $ \mathcal{A}_{f,\alpha}:=\{x\in\mathcal{S}(f)\}: \mathsf{E}G|\xi_{2,x}|^{\alpha}<\infty,\,\,\alpha=1,2,3,4.\}$, where $\mathcal{S}(f) $ is support of $ f $. According to the \textit{Lemma} \ref{G_a+b} and (\ref{G_a+b bounded}), one has $ \mu\big(\mathcal{S}(f)\setminus\mathcal{A}_{f,\alpha}\big)=0 $. Let us define $\mathcal{A}:=\mathcal{A}_{f,\alpha}(G)\cap \Lambda (f)\cap \mathcal{S}(f)\cap\mathcal{D}_{f}(R) $, where $\mu\big(\mathcal{S}(f)\setminus\mathcal{A}\big)=0  $, and for $ x\in\mathcal{A}$, we have
\begin{eqnarray}\label{R_alpha < infty}
R^{(\alpha)}(x)<\infty\,,\qquad\alpha=1,2,3,4.\,\,%\text{and}\, x\in  \mathcal{A}_{}
\end{eqnarray}
Substituting Eqs.(\ref{EGlog power})  % and (\ref{EGlog power3}),(\ref{EGlog power2}) and (\ref{EGlog power1}) 
and (\ref{M_J 2.2.2 (2) P1 eq2}) into (\ref{J_2 2.2.2 (2) eq4}),  we obtain
\begin{align}
\Psi(N,x)
%%&\hspace{.cm} M_{J_{2.2.2}^{(2)},P1}
\leq  \widetilde{R}(N,x) \label{M_J 2.2.2 (2) P1 eq3}
%%\\
%%&\hspace{-.3cm}\text{} \widetilde{R}(N,x)= \frac{R^{(4)}(x)}{16}+\frac{R^{(3)}(x)\log(N-1)}{3}+\frac{3R^{(2)}(x)\log^{2}(N-1)}{4} +R^{(1)}(x) \log^{3}(N-1)\nonumber \\
%%&\hspace{.cm} \text{and}\quad J_{2.2.2}^{(2)}(N,x)<\Big\{K(\Delta)+4\log[\log (N-1)]\,\Big\}P_{1}R(N,x)  \label{J_2 2.2.2 (2) eq5 final}
\end{align}
with
\begin{align}
 \widetilde{R}(N,x)= \frac{R^{(4)}(x)}{16}+\frac{R^{(3)}(x)\log(N-1)}{3}+\frac{3R^{(2)}(x)\log^{2}(N-1)}{4} +R^{(1)}(x) \log^{3}(N-1).\nonumber
\end{align}
From Eqs.(\ref{J_2 2.2.2 (2) eq4}) and (\ref{M_J 2.2.2 (2) P1 eq3}), we have
\begin{align}
 J_{2.2.2}^{(2)}(N,x)<\big(4\log\log (N-1)+\widetilde{\Delta}\big)\,\widetilde{R}(N,x)\,.  \label{J_2 2.2.2 (2) eq5 final}
\end{align}
%\begin{eqnarray}\label{J_2 2.2.2 (2) eq5 final}
% J_{2.2.2}^{(2)}(N,x)< \Big\{K(\Delta)+4\log[\log (N-1)]\,\Big\}P_{1}R(N,x)
%\end{eqnarray}
Recall that $R^{(\alpha)}(x)<\infty$ in Eq.(\ref{R_alpha < infty}), and this implies that
\begin{eqnarray}\label{P_1R}
\widetilde{R}(N,x) = \mathcal{O}\big(\log^{3}(N-1)\big)\,.
%%\widetilde{R}(N,x) \longrightarrow  \mathcal{O}([\log(N-1)]^{3})\,,\quad N\rightarrow\infty
\end{eqnarray}
%%Substituting $ J_{2.2.2}^{(1)}(N,x) $ in Eq.(\ref{J_2 2.2.2 eq3 (1) final}) and $ J_{2.2.2}^{(2)}(N,x) $ in Eq.(\ref{J_2 2.2.2 (2) eq5 final}) into $J_{2.2.2}(N,x)$ in Eq.(\ref{J_2 2.2.2 eq2}), we obtain
We substitute $ J_{2.2.2}^{(1)}(N,x) $ in (\ref{J_2 2.2.2 eq3 (1) final}) and $ J_{2.2.2}^{(2)}(N,x) $ in (\ref{J_2 2.2.2 (2) eq5 final}) into $J_{2.2.2}(N,x)$ in (\ref{J_2 2.2.2 eq2}):
\begin{eqnarray}\label{J_2 2.2.2 final}
&&J_{2.2.2}(N,x)=J_{2.2.2}^{(1)}(N,x)+J_{2.2.2}^{(2)}(N,x)\nonumber\\
&&\hspace{-.0cm}<\log^{3}[(N-1)e^{1+\Delta}]\big(\log\log[(N-1)e^{1+\Delta}]+\tfrac{1}{4}\big)\,\Delta \hspace{.cm}+\big(4\log\log (N-1)+\widetilde{\Delta}\big)\widetilde{R}(N,x)  \nonumber\\
&&=\mathcal{O}\big(\log^{3}(N-1)\,\log\log (N-1)\big)\,. %\,,\quad N\rightarrow\infty\,.
\end{eqnarray} 
We substitute $J_{2.2.1}(N,x) $ in (\ref{J_2 2.2.1 eq2 final}) and $J_{2.2.2}(N,x)  $ in (\ref{J_2 2.2.2 final}) into $  J_{2.2}(N,x)$ in (\ref{J_2 2.2.2 eq2}):
%%Substituting $J_{2.2.1}(N,x) $ in Eq.(\ref{J_2 2.2.1 eq2 final}) and $J_{2.2.2}(N,x)  $ in Eq.(\ref{J_2 2.2.2 final}) into $  J_{2.2}(N,x)$ in Eq.(\ref{J_2 2.2.2 eq2}), we obtain
\begin{eqnarray}
J_{2.2}(N,x)&=&J_{2.2.1}(N,x)+J_{2.2.2}(N,x) \nonumber\\
&<&\hspace{0cm}\log^{3}[(N-1)e]\big(\log(\log [(N-1)e]+\tfrac{1}{4}\big)\big[1+\tfrac{1}{2}\log (N-1)\big] \nonumber\\
&&+\log^{3}[(N-1)e^{1+\Delta}]\big(\log\log [(N-1)e^{1+\Delta}]+\tfrac{1}{4}\big)\,\Delta \nonumber\\
&&\hspace{.cm}+\big[4\log\log (N-1)+\widetilde{\Delta}\,\big]\widetilde{R}(N,x)\\
&&\hspace{-.6cm}=\mathcal{O}\big([\log(N-1)]^{4}\log\log(N-1)\big)\,.%\quad N\rightarrow\infty\,. 
\label{J_2 2.2.2 eq2 final}
\end{eqnarray}
Substituting $ J_{2.2}(N,x) $ in (\ref{J_2 2.2.2 eq2 final}) and $ J_{2.1}(N,x) $ in (\ref{J_2 2.1 6}) into $  J_{2}(N,x) $ in (\ref{J_2 2}), we obtain
\begin{eqnarray}
&&\hspace{-.9cm} J_{2}(N,x)\leq  J_{2.1}(N,x)\cdot J_{2.2}(N,x)\nonumber\\
&&\hspace{-.9cm}\leq \dfrac{(2e^{\gamma})^{\varepsilon}}{m_{f}(x,R_{2})^{\varepsilon}}\dfrac{1}{(N-1)^{\varepsilon/2}} 
\cdot\Big\{\log^{3}[(N-1)e]\big(\log\log[(N-1)e]+\tfrac{1}{4}\big) \big[1+\tfrac{1}{2}\log (N-1)\big]\nonumber\\
&&\hspace{-.5cm}+\log^{3}[(N-1)e^{1+\Delta}]\big(\log\log[(N-1)e^{1+\Delta}]+\tfrac{1}{4}\big)\,\Delta \nonumber\\
&&\hspace{-.5cm}+\big[4\log\log (N-1)+\widetilde{\Delta}\,\big]\widetilde{R}(N,x)\Big\}\,,\qquad \varepsilon\in(0,e] \nonumber\\
&&\hspace{-.9cm}:= R_{J_{2}}(N,x,\Delta)\,. \label{J_2 2 final}
\end{eqnarray}
Note that logarithm functions grow slower than power functions, and this implies that
\begin{eqnarray}\label{R_j_2}
R_{J_{2}}(N,x,\Delta)=\dfrac{\mathcal{O}\big([\log^{4}(N-1)]\log\log (N-1)\big)}{\mathcal{O}((N-1)^{\varepsilon/2})}\longrightarrow 0\,,\quad N\rightarrow\infty \,. %%\,,\,\, \varepsilon\in(0,e]\,.
\end{eqnarray}
Thus, we have
\begin{eqnarray}\label{J_2 2 final+1}
J_{2}(N,x)\leq R_{J_{2}}(N,x,\Delta) <\infty \,.
\end{eqnarray}
Note that we have proved $ J_{1}(N,x) <\infty $ in \text{Section (\ref{section I2})}, and thus, we proved $ I_{2}(N,x)=J_{1}(N,x)+J_{2}(N,x) <\infty $. 
Combing $ I_{1}(N,x)<\infty $ in \text{Section (\ref{section I1})} and $ I_{2}(N,x)<\infty  $ in \text{Section (\ref{section I2})} into Eq.(\ref{EG log4 }), we obtain that 
\begin{eqnarray}
\mathsf{E}G\big(\mid\log^{4}\xi_{N,x}\mid\big)=16\big[I_{1}(N,x)+I_{2}(N,x)\big]<\infty \label{EG log4 final}
\end{eqnarray}
given the constants $ \varepsilon= \dfrac{\varepsilon_{0}\varepsilon_{2}}{1+\varepsilon_{0}} $, $ \varepsilon_{2} \leq \frac{1+\varepsilon_{0}}{\varepsilon_{0}}e$, $ \varepsilon\,,\varepsilon_{0}\,,\varepsilon_{2}>0 $ and $ x\in\mathcal{A}$, %$ x\in\mathcal{A}=\mathcal{A}_{f,\alpha}(G)\cap \Lambda (f)\cap \mathcal{S}(f)\cap\mathcal{D}_{f}(R) $, 
$ N\geq \overline{N}=\max\big\{N_{1}, N_{2}\big\} $. We proved the uniform integrability of $ \big\{\log^{4}\xi_{N,x}\big\} $ in (\ref{uniform integrability variance power 4}), and hence:
\begin{eqnarray}
\mathsf{Var}\big[ \tilde{\zeta}_{1}^{2}(N)\big]<\infty\,.
\end{eqnarray}
%%%%%%%%%%%%%%%%%%%%%%%%%%%%%%%%%%%%%%%%%%%%%%%%%%%%%%%%%%%%%%%%%%%%%%%%%%%%%%%%%%%%%%%%%%%%%%%%%%%%%%%%%%%%%%%%%%%%%%%%%%%%%%%%%%%%
\subsection{Proof of $ \mathsf{Cov}\big[ \tilde{\zeta}_{i}^{2}(N)\,,\,\tilde{\zeta}_{j}^{2}(N)\big]\xrightarrow[\text{}]{N\rightarrow \infty}\,0  $ }
%%%%%%%%%%%%%%%%%%%%%%%%%%%%%%%%%%%%%%%%%%%%%%%%%%%%%%%%%%%%%%%%%%%%%%%%%%%%%%%%%%%%%%%%%%%%%%%%%%%%%%%%%%%%%%%%%%%%%%%%%%%%%%%%%%%%
%\subsection{CDF and expectations of joint probability distributions}
%%%%%%%%%%%%%%%%%%%%%%%%%%%%%%%%%%%%%%%%%%%%%%%%%%%%%%%%%%%%%%%%%%%%%%%%%%%%%%%%%%%%%%%%%%%%%%%%%%%%%%%%%%%%%%%%%%%%%%%%%%%%%%%%%%%%
Note that $ \tilde{\zeta}_{i}^{2}(N)={\zeta}_{i}^{2}(N)-\frac{\pi^{2}}{6} $ in Eq.(\ref{SNx expansion}). For $ i\neq j $ and $ N\geq 3 $, we have
\begin{eqnarray}
\mathsf{Cov}\big[ \tilde{\zeta}_{i}^{2}(N)\,,\,\tilde{\zeta}_{j}^{2}(N)\big]\,.
%&=&\mathsf{E}\big\{\big[\big({\zeta}_{i}^{2}(N)-\sigma^{2}\big)-\mathsf{E}\big({\zeta}_{i}^{2}(N)-\sigma^{2}\big)\big]\big[\big({\zeta}_{j}^{2}(N)-\sigma^{2}\big)-\mathsf{E}\big({\zeta}_{j}^{2}(N)-\sigma^{2}\big)\big]\big\}\nonumber\\
%&=&\mathsf{E}\big\{\big[{\zeta}_{i}^{2}(N)-\mathsf{E}{\zeta}_{i}^{2}(N)\big]\big[{\zeta}_{j}^{2}(N)-\mathsf{E}{\zeta}_{j}^{2}(N)\big]\big\}\nonumber\\
&=&\mathsf{Cov}\big[ {\zeta}_{i}^{2}(N)\,,\,{\zeta}_{j}^{2}(N)\big]
\end{eqnarray}
For the investigation of $  \mathsf{Cov}\big[ {\zeta}_{i}^{2}(N)\,,\,{\zeta}_{j}^{2}(N)\big]$, we define the joint cumulative distribution function (CDF) $ F^{i,j}_{N,x,y}(u,w) $ as follows.
\begin{eqnarray}\label{CDF}
\hspace{-.6cm} F^{i,j}_{N,x,y}(\tilde{u},\tilde{w})&=&\mathsf{Pr}({\zeta}_{i}^{2}(N)\leq\tilde{u}^{2}\,,\,{\zeta}_{j}^{2}(N))\leq\tilde{w}^{2} \big| X_{i}=x\,,X_{j}=y)\nonumber\\
\hspace{-.6cm} &=&\mathsf{Pr}(|\zeta_{i}(N)|\leq\tilde{u}\,,\,|\zeta_{j}(N)|\leq\tilde{w} \big| X_{i}=x\,,X_{j}=y)\,,\quad \tilde{u},\,\tilde{w}\geq 0.
\end{eqnarray}
Note that $ \zeta_{i}(N) $ is defined in Eq.(\ref{zeta_i N}) which can be rewritten as $\zeta_{i}(N)=\log\xi_{N,X_{i}} $ with 
\begin{align}\label{xi_X_i}
\xi_{N,X_{i}} = {\rho^{d}_{i}}V_{d}e^{\gamma}(N-1)\,\,\,\text{with}\,\,\, \rho_{i}:=\min\limits_{k\in\{1,\ldots ,N\}\setminus\{i\}}\rho(X_{i}, X_{k}).
\end{align}
Clearly, $0 <\rho_{i}\leq \rho(X_{i},X_{k})|_{X_{i}=x,X_{k}=y}:=\rho(x,y)$.
We set $ u=e^{\tilde{u}}$, and hence:
\begin{eqnarray}\label{zeta xi inequality}
|\zeta_{i}(N)|\leq\tilde{u} %%\, &\Leftrightarrow & |\log\xi_{N,X_{i}}|\leq\tilde{u} %\,,\quad \tilde{u}\geq 0 \nonumber\\
\quad &\Leftrightarrow &\quad \frac{1}{u} \leq \xi_{N,X_{i}}\leq u  %,\,\quad {u}\geq 1
%\nonumber\\
\quad \Leftrightarrow \quad \dfrac{1}{u}\leq \rho_{i}^{d}V_{d}e^{\gamma}(N-1)\leq u\,.
\end{eqnarray}
%%with $ u=e^{\tilde{u}}\in[1,\infty) $ as $ \tilde{u}\geq 0 $. Eq.(\ref{zeta xi inequality}) implies
%%\begin{eqnarray}
%%\dfrac{1}{u}\leq{\min\limits_{i\in\{1,\ldots ,N\}\setminus\{j\}}\rho^{d}(X_{i}, X_{j})}V_{d}e^{\gamma}(N-1)\leq u\,,\quad u\geq 1
%%\end{eqnarray}
For $ \tilde{u}\geq 0 $, one has $u \geq 1 $. Note that $ \xi_{N,X_{i}} $ increases in proportion to $ N $, and this implies $ \xi_{N,X_{i}}\rightarrow \infty $ as $ N\rightarrow \infty $. Hence, there exits a positive $ {N}_{0} $, such that for $ N\geq {N}_{0} $, one has $ \xi_{N,X_{i}} \geq 1 $ or $ \log\xi_{N,X_{i}}\geq 0$.
From Eq.(\ref{zeta xi inequality}), one has $ \xi_{N,X_{i}}\geq 1 \Leftrightarrow N \geq \frac{1}{\rho_{i}^{d}V_{d}e^{\gamma}} +1  $. 
Note that $ N\geq 3 $ is required for deriving the covariance matrix, and we define $ N_{0} $:
\begin{align}\label{N_0}
N_{0}=\max\bigg\{\bigg\lceil \frac{1}{\rho_{i}^{d}V_{d}e^{\gamma}} \bigg\rceil+1,\, 3\bigg\}\,.
\end{align}
%%\begin{eqnarray}\label{N_0 tilde}
%%\xi_{N,X_{i}}\geq 1 \quad \Leftrightarrow \quad N \geq \dfrac{1}{\rho_{i}^{d}V_{d}e^{\gamma}} +1 
%%\end{eqnarray}
%%\begin{eqnarray}\label{N_0 tilde}
%%{N}_{0}:=\bigg\lfloor \dfrac{1}{\rho_{i}^{d}V_{d}e^{\gamma}}\bigg\rfloor +2 \geq \dfrac{1}{\rho_{i}^{d}V_{d}e^{\gamma}} +1
%%\end{eqnarray}
%%$\lceil t\rceil =\min\big\{n\in\mathbb{Z}: n\geq t\big\} $ denotes the smallest integer more than or equal to $ t $. 
%%\begin{eqnarray}\label{N_0 tilde}
%%{N}^{\ast}&\geq & \dfrac{1}{\rho_{i}^{d}V_{d}\tilde{\gamma}}+1 \nonumber\\
%%&\geq & N_{\rho} +1
%%\end{eqnarray}
%%\begin{eqnarray}\label{N_0 tilde}
%%N&\geq & \dfrac{1}{\rho_{i}^{d}V_{d}\tilde{\gamma}}+1 \nonumber\\
%%&\geq & N_{\rho} +1 := {N}^{\ast}\,, \\
%%&&\hspace{-3cm}\text{with}\,\,\, N_{\rho} = \bigg\lfloor \dfrac{1}{{\min\limits_{i\neq j}\rho^{d}(X_{i}, X_{j})}V_{d}e^{\gamma}}\bigg\rfloor =\max\bigg\{n\in \mathbb {Z} \mid n\leq \dfrac{1}{\rho_{i}V_{d}e^{\gamma}}\bigg\} \,. \nonumber
%%\end{eqnarray}
%%where $  N_{\rho}$ is defined as the integer part of $ \frac{1}{{\min\limits_{i\neq j}\rho^{d}(X_{i}, X_{j})}V_{d}\tilde{\gamma}} $, %i.e. 
%%\begin{eqnarray}
%%N_{\rho} = \bigg\lfloor \dfrac{1}{{\min\limits_{i\neq j}\rho^{d}(X_{i}, X_{j})}V_{d}\tilde{\gamma}}\bigg\rfloor =\max\bigg\{n\in \mathbb {Z} \mid n\leq \dfrac{1}{{\min\limits_{i\neq j}\rho^{d}(X_{i}, X_{j})}V_{d}\tilde{\gamma}}\bigg\} \,, \nonumber
%%\end{eqnarray}
%%and $\mathbb {Z}  $ represents the set of integers. 
${N}_{0} $ does not depend on $ x $ and $ y $. For $ N\geq {N}_{0} $, the intervals of $\xi_{N,X_{i}}$ and ${\zeta}_{i}(N) $ are, respectively, $ [1,u] $ and $ [0,\tilde{u}] $. %, where $ \tilde{u}=\log u,\, \tilde{u}\geq 0\,,u\geq 1 $. Furthermore, since $ 1\leq\xi_{N,X_{i}}\leq u $, we have
This implies $ r_{N}(1) \leq \rho_{i} \leq r_{N}(u)$ and
%\begin{eqnarray}
%&&\hspace{1cm} 
%r_{N}(1) \leq \rho_{i} \leq r_{N}(u) 
%\,,\\
%&&\hspace{-.7cm} \text{with}\,\, r_{N}(u):=\bigg[\dfrac{u}{V_{d}\tilde{\gamma}(N-1)}\bigg]^{1/d} \,\,\text{and}\,r_{N}(1):=\bigg[\dfrac{1}{V_{d}\tilde{\gamma}(N-1)}\bigg]^{1/d}\nonumber\,.
%\end{eqnarray} 
 %When $ N\geq \tilde{N}_{0} $,the probability 
$ \mathsf{Pr}(\xi_{N,X_{i}}<1)=0 \Leftrightarrow \mathsf{Pr}\big(\rho_{i}< r_{N}(1)\big)=0$. Clearly, Eqs.(\ref{xi_X_i})-(\ref{N_0 tilde}) hold for $ \zeta_{j}(N) $ and $ \xi_{N,X_{j}} $ when one replaces $i $ with $j $, and $ u,\tilde{u} $ with $ w,\tilde{w} $. The CDF $ F^{i,j}_{N,x,y}(\tilde{u},\tilde{w}) $ in (\ref{CDF}) is rearranged as:
\begin{eqnarray}
%%&& F^{i,j}_{N,x,y}(\tilde{u},\tilde{w}) %\,,\,\qquad \tilde{u},\,\tilde{w}\geq 0 ,\, N\geq \tilde{N}_{0}  
%%\nonumber\\
&&\hspace{-0.8cm} F^{i,j}_{N,x,y}(\tilde{u},\tilde{w}) =1-\mathsf{Pr}\big(|\zeta_{i}(N)|> \tilde{u} \big| X_{i}=x\big)-\mathsf{Pr}\big(|\zeta_{j}(N)|> \tilde{w} \big| X_{j}=y\big)\nonumber\\
&&\hspace{1.8cm}+\mathsf{Pr}\big(|\zeta_{i}(N)|> \tilde{u}\,,\,|\zeta_{j}(N)|> \tilde{w} \big| X_{i}=x\,,X_{j}=y\big) %\,,\,\quad \tilde{u},\,\tilde{w}\geq 0 ,\, N\geq \tilde{N}_{0} 
\nonumber\\
&=&1-\mathsf{Pr}\big(\zeta_{i}(N)> \tilde{u} \big| X_{i}=x\big)-\mathsf{Pr}\big(\zeta_{j}(N)> \tilde{w} \big| X_{j}=y\big)\nonumber\\
&&+\mathsf{Pr}\big(\zeta_{i}(N)> \tilde{u}\,,\,\zeta_{j}(N)> \tilde{w} \big| X_{i}=x\,,X_{j}=y\big) %\,,\quad \tilde{u},\,\tilde{w}\geq 0,\, N\geq {N}_{0} 
\nonumber\\
%&=&1-\mathsf{Pr}\big(\xi_{N,X_{i}}> u \big| X_{i}=x\big)-\mathsf{Pr}\big(\xi_{N,X_{j}}> w \big| X_{j}=y\big)\nonumber\\
%&&+\mathsf{Pr}\big(\xi_{N,X_{i}}> u\,,\,\xi_{N,X_{j}}> w \big| X_{i}=x\,,X_{j}=y\big),\,\, {u=e^{\tilde{u}}}\geq 1 ,\,{w=e^{\tilde{w}}}\geq 1,\,N\geq \tilde{N}_{0} \quad \\
&=& 1-\mathsf{Pr}\big(\rho_{i}> r_{N}(u) \big| X_{i}=x\big)-\mathsf{Pr}\big(\rho_{j}> r_{N}(w) \big| X_{j}=y\big)\nonumber\\
&&+\mathsf{Pr}\big(\rho_{i}> r_{N}(u)\,,\,\rho_{j}> r_{N}(w) \big| X_{i}=x\,,X_{j}=y\big) \nonumber\\
&=&1-\mathds{1}\big(\rho(x, y)>r_{N}(u)\big)\mathsf{Pr}\big({\min\limits_{k\neq i,j}\rho(x,X_{k})}> r_{N}(u)\big) \nonumber\\
&&-\mathds{1}\big(\rho(x, y)>r_{N}(w)\big)\mathsf{Pr}\big({\min\limits_{k\neq i,j}\rho(y,X_{k})}> r_{N}(w) \big)\nonumber\\
&&\hspace{-.8cm}+\mathds{1}\big(\rho(x, y)>\max\{r_{N}(u),r_{N}(w)\}\big)\mathsf{Pr}\big({\min\limits_{k\neq i,j}\rho(x, X_{k})}> r_{N}(u)\,,\,{\min\limits_{k\neq i,j}\rho(y, X_{k})}> r_{N}(w) \big) \nonumber\\
&:=& F_{N,x,y}(u,w) \,,\label{CDF 2}
\end{eqnarray}
where the subscript of the $ \min\limits_{k\neq i, j}$ denotes $ k\in\{1,\ldots, N\}\setminus \{i,j\}$. $ r_{N}(u)$ is defined in Eq.(\ref{rNu e PNx}). Note that the formula of $ F^{i,j}_{N,x,y}(\tilde{u},\tilde{w}) $ is identical for every $ i, j $, and thus, we neglect the superscripts and simplify $ F^{i,j}_{N,x,y} $ as $ F_{N,x,y}$. 
%to re-denote $ F^{i,j}_{N,x,y}(\tilde{u},\tilde{w}) $, which stays the same $\forall\, i,j\in\{1,2,\ldots ,N\} $, $ i\neq j $.
%%Analogously to analysis on page 35 of Bulinski and Dimitrov 2019, we consider the set $ A=A_{1}\cup A_{2} $ with $ A_{1}=\{(x,y):x\in A, y\in A, \,x\neq y\} $ and $ A_{2}=\{(x,y):x\in A, y\in A,\, x=y\} $.
We set $ \mathcal{A}=\mathcal{A}_{1}\cup \mathcal{A}_{2} $ with $ \mathcal{A}_{1}=\{(x,y):x\in \mathcal{A}, y\in \mathcal{A}, \,x\neq y\} $ and $ \mathcal{A}_{2}=\{(x,y):x\in \mathcal{A}, y\in \mathcal{A},\, x=y\} $. Note that the random variable $ \xi $ has a probability distribution function $ P_{\xi}  $ and a density function $ f_{\xi} $ w.r.t the league measure in $\mathbb{R}^{d} $. The following results hold evidently
\begin{eqnarray}  
\big(P_{\xi}\otimes P_{\xi}\big)\big(\mathcal{A}_{1})=1 \qquad \text{and}\qquad \big(P_{\xi}\otimes P_{\xi}\big)\big(\mathcal{A}_{2})=0\,. \nonumber
\end{eqnarray}
%Referring to the result on the Page 35 of Bulinski and Dimitrov 2019, 
Clearly, we have $ F_{N,x,y}(u,w)=1\,,\,\, {u}\geq 1 ,\,{w}\geq 1 $, when $(x,y)\in \mathcal{A}_{2}$. In case of $(x,y)\in\mathcal{A}_{1}$, $ \forall t\geq 0 $, we have $ r_{N}(t)\rightarrow 0 $, as $ N\rightarrow\infty $.
%\begin{enumerate}[(i)]
%\item When $(x,y)\in A_{2}$, we have $  F_{N,x,y}(u,w)=1\,,\quad {u}\geq 1 ,\,{w}\geq 1$
%\item When $(x,y)\in A_{1}$, $ \forall t\geq 0 $, we have $ r_{N}(t)\rightarrow 0 $, as $ N\rightarrow\infty $. 
%\end{enumerate}
Thus, we can find a finite number $N'_{0}=  N'_{0} (x,y,w,u)$, such that for any $ N\geq N'_{0} $, one has
\begin{eqnarray}\label{rNu,rNw,BxBy}
r_{N}(u)<\dfrac{\rho(x,y)}{2}\,,\quad  r_{N}(w)<\dfrac{\rho(x,y)}{2}\quad \text{and}\quad B(x,r_{N}(u))\cap B(y,r_{N}(w))=\emptyset  \,.
\end{eqnarray}
Moreover, there exits a finite number $\widetilde{N}_{0}  $:
\begin{align}\label{N_0 tilde}
\widetilde{N}_{0}= \max \{N_{0}, N'_{0}\}\,,
\end{align}
such that for any $N\geq\widetilde{N}_{0}:= \max \{N_{0}, N'_{0}\}$, %such that $ \forall N\geq N_{0} $, we have 
and substituting Eq.(\ref{rNu,rNw,BxBy}) into (\ref{CDF 2}), we obtain 
\begin{eqnarray}
F_{N,x,y}(u,w)&=& 1-\mathsf{Pr}\big({\min\limits_{k\neq j,i}\rho(X_{k}, x)}> r_{N}(u)\big) -\big({\min\limits_{k\neq i,j}\rho(X_{k}, y)}> r_{N}(w) \big)\nonumber\\
&&+\mathsf{Pr}\big({\min\limits_{k\neq i,j}\rho(x, X_{k})}> r_{N}(u)\,,\,{\min\limits_{k\neq i,j}\rho(y, X_{k})}> r_{N}(w) \big)%\,,\quad  \forall N\geq N_{0}
\nonumber\\
&=& 1-\prod\limits_{k\neq i,j}\mathsf{Pr}\big(X_{k}\notin B(x,r_{N}(u)) \big) -\prod\limits_{k\neq i,j}\mathsf{Pr}\big(X_{k}\notin B(y,r_{N}(w)) \big)\nonumber\\
&&+\prod\limits_{k\neq i,j}\mathsf{Pr}\big(X_{k}\notin B(x,r_{N}(u))\cup B(y,r_{N}(w)) \big)%\,,\quad  \forall N\geq N_{0}
\nonumber\\
&=& 1-\bigg(1-\int_{B(x,r_{N}(u))}f(z)dz\bigg)^{N-2}-\bigg(1-\int_{B(y,r_{N}(w))}f(z)dz\bigg)^{N-2}\nonumber\\
&&+\bigg(1-\int_{B(x,r_{N}(u))}f(z)dz-\int_{B(y,r_{N}(w))}f(z)dz\bigg)^{N-2}%\,,\,\,  \forall N\geq N_{0}
\,.\qquad\label{CDF 3}
\end{eqnarray}
Note that $ f(x)>0 $ and $ f(y)>0 $ for $ x,y\in\mathcal{S}(f) $. %We consider the limit case of $F_{N,x,y}(u,w)  $ when $ N\rightarrow\infty $.
%When $ N\rightarrow\infty $, 
Moreover, $F_{N,x,y}(u,w)  $ converges to:
\begin{eqnarray}
F_{x,y}(u,w)&=&\lim\limits_{N\rightarrow\infty} F_{N,x,y}(u,w)\nonumber\\
&=& 1-\lim\limits_{N\rightarrow\infty}\bigg(1-\dfrac{uf(x)}{\tilde{\gamma}(N-1)}\bigg)^{N-2}-\lim\limits_{N\rightarrow\infty}\bigg(1-\dfrac{wf(y)}{\tilde{\gamma}(N-1)}\bigg)^{N-2} \nonumber\\
&&+\lim\limits_{N\rightarrow\infty}\bigg(1-\dfrac{uf(x)+wf(y)}{\tilde{\gamma}(N-1)}\bigg)^{N-2} \nonumber\\
&=& 1-e^{\frac{-uf(x)}{\tilde{\gamma}}}-e^{\frac{-wf(y)}{\tilde{\gamma}}}+e^{-\frac{-uf(x)+wf(y)}{\tilde{\gamma}}} \nonumber\\
&=&\Big(1-e^{\frac{-f(x)u}{\tilde{\gamma}}}\Big) \Big(1-e^{\frac{-f(y)w}{\tilde{\gamma}}}\Big)\,\quad  \nonumber\\
&=& F_{x}(u)\cdot F_{y}(w) \,.\label{CDF 4 infty}
\end{eqnarray}
Clearly, $ F_{x,y}(u,w) $ is a joint distribution function of a vector $ \eta_{x,y}=(\eta_{x},\eta_{y}) $, where $ \eta_{x} $ and $ \eta_{y} $ are independent random variables and follow the following exponential distributions with means $ e^{\gamma}/f(x) $ and $e^{\gamma}/f(y) $, respectively.\\
%%\begin{eqnarray} \label{eta x, eta y}
%%\eta_{x}\sim \exp\Big(\dfrac{f(x)}{\tilde{\gamma}}\Big)\,,\quad \eta_{y}\sim \exp\Big(\dfrac{f(y)}{\tilde{\gamma}}\Big)\,,\quad \text{and}\quad \eta_{x}\perp \eta_{y}
%%\end{eqnarray}
%For $ (x,y)\in A_{1}$, i.e. $ x,y\in A $ and $ x\neq y $, and for $ i\neq j $ and $ N\geq 3 $, we define $ \eta^{i,j}_{N,x,y}=\big(\eta^{y,i,j}_{N,x}, \eta^{x,i,j}_{N,y}\big)$ with
%\begin{eqnarray}
%&&\eta^{y,i,j}_{N,x}=(N-1)V_{d}\tilde{\gamma}\min\Big\{\Big[\min\limits_{k\in\{1,\ldots,N\}\setminus\{i,j\}}\rho^{d}(X_{k},x)\Big],\rho^{d}(y,x)\Big\} \label{eta y,i,j N,x 1} \\
%&&\eta^{x,i,j}_{N,y}=(N-1)V_{d}\tilde{\gamma}\min\Big\{\Big[\min\limits_{k\in\{1,\ldots,N\}\setminus\{i,j\}}\rho^{d}(Y_{k},y)\Big],\rho^{d}(x,y)\Big\}
%\end{eqnarray}
Note that $ \zeta_{i}(N) $ in Eq.(\ref{zeta_i N}) is defined on a ball centred at $ X_{i} $. Note also that $\xi_{N,x} $ in Eq.(\ref{xi limit}) satisfies:
\begin{align}
\zeta_{1}(N)\big|_{X_{1}=x}=\log \xi_{N,x}, \quad \text {or}\quad \xi_{N,x} = e^{\zeta_{1}(N)}\big|_{X_{1}=x}. \nonumber
\end{align}  
$\xi_{N,x} $ is defined on a ball centred at $ X_{1}=x $.
For $ N\geq 3 $, a ball centred at $ X_{i} $ and with a given point $ X_{j} $, $ i\neq j $, the definition of $ \zeta_{i}(N) $ can be expanded to $ \zeta_{i\,|\,j}(N) $:
\begin{eqnarray} \label{zeta e xi relation}
\zeta_{i\,|\,j}(N)=\log\Big\{(N-1) V_{d}e^{\gamma} \rho_{i\,|\,j}^{d} \Big\} % \label{Xi(N,x)} 
%\\
%\xi_{N,y}^{(j)} = \xi_{N,X_{j}}\big|_{X_{j}=y}=(N-1) V_{d}\tilde{\gamma} {\min\limits_{k\neq i}\rho^{d}(X_{k},y)}  \label{Xj(N,y)} 
\end{eqnarray}
with %$ \rho_{i\,|\,j} $ defined in Eq.(\ref{Xi(N,x),y}).
\begin{align}
\rho_{i\,|\,j}= \min\limits_{k\in\{1,\ldots,N\}\setminus\{i\}}\rho(X_{i},X_{k})\Big|_{X_{j}} =\min\Big\{\Big[\min\limits_{k\in\{1,\ldots,N\}\setminus\{i,j\}}\rho(X_{i},X_{k})\Big],\rho(X_{i},X_{j})\Big\} \,.       \label{Xi(N,x),y}
\end{align}
For $ N\geq 3 $, $ X_{1}=x $, $X_{2}=y $ and $ (x,y)\in\mathcal{A}_{1}$, we expand $\xi_{N,x} $ to $ \eta^{}_{N,x,y}=\big(\eta^{y}_{N,x}, \eta^{x}_{N,y}\big)$:
\begin{eqnarray}
&\eta^{y}_{N,x}=\xi_{N,x}\big|_{X_{2}=y}=(N-1)V_{d}e^{\gamma}\min\Big\{\Big[\min\limits_{
%k\in\{1,\ldots,N\}\setminus\{i,j\}
k=3,\ldots,N
}
\rho^{d}(x,X_{k})\Big],\rho^{d}(x,y)\Big\}\,,\quad \label{eta y,i,j N,x 1} \\
&\eta^{x}_{N,y}=\xi_{N,y}\big|_{X_{1}=x}=(N-1)V_{d}e^{\gamma}\min\Big\{\Big[\min\limits_{k=3,\ldots,N}\rho^{d}(y,X_{k})\Big],\rho^{d}(x,y)\Big\}\,,\quad \label{eta x,i,j N,y 1}
\end{eqnarray}
%%Let us consider $ X_{i}=x $ and $ X_{j}=y $, we have
%%\begin{eqnarray}\label{Xi(N,x),y}
%%  \min\limits_{k\in\{1,\ldots,N\}\setminus\{i,j\}}\rho(x,X_{k})\Big|_{X_{j}=y} 
%%  &=& \min\Big\{\Big[\min\limits_{k\in\{1,\ldots,N\}\setminus\{i,j\}}\rho(x,X_{k})\Big],\rho(x,y)\Big\}
%%\end{eqnarray}
%%Note that $\xi_{N,x} $ in Eq.(\ref{xi limit}) is defined for a ball centred at $ X_{1}=x $. We generalize this definition as $\xi_{N,X_{i}} =(N-1)V_{d}e^{\gamma}{\min\limits_{k\in\{1,\ldots,N\}\setminus\{i\}}\rho^{d}(X_{i}, X_{k})} $ 
%%\begin{align}
%%\min\limits_{k\in\{2,\ldots,N\}}\rho(x,X_{k}) =\min\Big\{\Big[\min\limits_{k\in\{3,\ldots,N\}}\rho(x,X_{k})\Big],\rho(x,y)\Big\}
%%\end{align}
and analogously to the Eq.(\ref{zeta e xi relation}), the relations between $\zeta_{i\,|\,j}(N)$ and $ \eta^{}_{N,x,y} $ satisfies:
\[
\begin{pmatrix}
    \zeta_{1\,|\,2}(N)\big|X_{1}=x,X_{2}=y \\
    \zeta_{2\,|\,1}(N)\big|X_{2}=y,X_{1}=x
  \end{pmatrix}  
  =
\begin{pmatrix}
   \log \eta^{y}_{N,x}\\
   \log \eta^{x}_{N,y}
  \end{pmatrix}\,.
\]
The distribution function of $ \eta^{y}_{N,x} $ is
\begin{eqnarray} %\label{eta y,i,j N,x 2}
\mathsf{Pr}\big(\eta^{y}_{N,x}\leq u\big)&=&1- \mathsf{Pr}\big(\eta^{y}_{N,x}> u\big) \label{F Nx u 1}\nonumber \\
%&=&1- \mathsf{Pr}\bigg(\min\Big\{\Big[\min\limits_{k\in\{1,\ldots,N\}\setminus\{i,j\}}\rho^{d}(y,Y_{k})\Big],\rho^{d}(y,x)\Big\}>\bigg[\dfrac{u}{(N-1)V_{d}\tilde{\gamma}}\bigg]^{1/d}\bigg)\nonumber\\
&=&1- \mathsf{Pr}\bigg(\min\Big\{\Big[\min\limits_{k\in\{1,\ldots,N\}\setminus\{i,j\}}\rho(x,X_{k})\Big],\rho(x,y)\Big\}>r_{N}(u)\bigg)\,\nonumber\\
%%&&\text{with}\quad r_{N}(u)=\bigg[\dfrac{u}{(N-1)V_{d}\tilde{\gamma}}\bigg]^{1/d} \,,\nonumber \\
%%&=&1- \mathds{1}\bigg[\rho(x, y)>r_{N}(u)\bigg] \mathsf{Pr}\bigg(\Big[\min\limits_{k\in\{1,\ldots,N\}\setminus\{i,j\}}\rho(x,X_{k})\Big]>r_{N}(u)\bigg)\,\nonumber\\
%%&&\hspace{0cm}\text{remark}\,X_{1},\ldots,X_{N}\,\overset{\text{i.i.d.}}{\sim} \,\xi\,, \text{and thus for}\, k\in\{1,\ldots,N\}\setminus\{i,j\}\,,\nonumber \\
%%&& \text{we have}\,\, \min\rho(X_{k},x)>r_{N}(u)\quad \Longleftrightarrow \quad \forall \rho(X_{k},x)>r_{N}(u)\nonumber\\
&=&1- \mathds{1}\Big[\rho(x, y)>r_{N}(u)\Big] \prod\limits_{k\in\{1,\ldots,N\}\setminus\{i,j\}}\mathsf{Pr}\Big[\rho(X_{k},x)>r_{N}(u)\Big]\,\nonumber\\
&=&1- \mathds{1}\Big[\rho(x, y)>r_{N}(u)\Big] \Big[\mathsf{Pr}\Big(\rho(\xi,x)>r_{N}(u)\Big)\Big]^{N-2} \label{F Nx u 2}\nonumber \\ 
%%&=&1- \mathds{1}\Big[\rho(x, y)>r_{N}(u)\Big] \Big[1-\mathsf{Pr}\Big(\xi\in B(x,r_{N}(u))\Big)\Big]^{N-2}\,\nonumber\\
%%&=&1- \mathds{1}\Big[\rho(x, y)>r_{N}(u)\Big] \Big[1-\int_{B(x,r_{N}(u))}f(\xi)d\xi\Big]^{N-2}\,\nonumber\\
&=&1- \mathds{1}\Big[\rho(x, y)>r_{N}(u)\Big] \Big[1-\mathsf{P}_{N,x}(u)\Big]^{N-2}\,\nonumber\\
&:=& F^{y}_{N,x}(u) \,.\label{F Nx u 3}
\end{eqnarray} 
Analogously, we obtain the distribution function of $ \eta^{x}_{N,y} $:
\begin{eqnarray}
F^{x}_{N,y}(w)&:=&\mathsf{Pr}\big(\eta^{x}_{N,y}\leq w\big)= 1- \mathds{1}\Big[\rho(x, y)>r_{N}(w)\Big] \Big[1-\mathsf{P}_{N,y}(w)\Big]^{N-2}  \,.\label{F Ny w}
\end{eqnarray}
$ F^{y}_{N,x}(u) $ and $ F^{x}_{N,y}(w) $ represent the marginal distributions of the random variable $\eta^{y}_{N,x}  $ and  $\eta^{x}_{N,y}  $ respectively. Recall that $ F^{}_{N,x,y}:=F^{}_{N,x,y}(u,w) $ is a cumulative distribution function of $ \eta^{}_{N,x,y}=\big(\eta^{y}_{N,x}, \eta^{x}_{N,y}\big) $, when $ (x,y)\in \mathcal{A}_{1}$. Substituting Eqs.(\ref{F Nx u 3}) and (\ref{F Ny w}) into (\ref{CDF 4 infty}), we obtain $ \eta^{y}_{N,x} \xrightarrow[\text{}]{law}\,\eta^{}_{x} $ and $ \eta^{x}_{N,y} \xrightarrow[\text{}]{law}\,\eta^{}_{y} $ (or $\eta^{}_{N,x,y} \xrightarrow[\text{}]{law}\,\eta^{}_{x,y}$), $ N\rightarrow \infty    $.
%%\begin{eqnarray}
%%\eta^{y}_{N,x} \xrightarrow[\text{}]{law}\,\eta^{}_{x} \,\,\text{and}\,\,\eta^{x}_{N,y} \xrightarrow[\text{}]{law}\,\eta^{}_{y} \,\,\,(\text{or}\,\,\,\eta^{}_{N,x,y} \xrightarrow[\text{}]{law}\,\eta^{}_{x,y})\,,\quad N\rightarrow \infty   \nonumber
%%\end{eqnarray}
Note that a random variable's convergence in law is preserved under the continuous mapping. Thus, we can obtain
\begin{eqnarray}\label{convergency in law joint expectation}
\big(\log^{2}\eta^{y}_{N,x}\big)\big(\log^{2}\eta^{x}_{N,y}  \big)\xrightarrow[\text{}]{law}\,\big(\log^{2}\eta^{}_{x}\big)\big(\log^{2}\eta^{}_{y}\big) \,,\quad N\rightarrow \infty \,.
\end{eqnarray}
Let us rule out a set of zero probability, %where the value of the random variable is zero, 
i.e. $ \eta^{}_{N,x,y}=\big(\eta^{y}_{N,x}, \eta^{x}_{N,y}\big)=(0,0) $.
\begin{lemma}\label{eta>0 both}
For each $ (x,y)\in\mathcal{A} $, $ x\neq y $, one has $ \eta^{y}_{N,x}>0 $ and $ \eta^{x}_{N,y}>0 $ almost sure.
\end{lemma}
\begin{proof}[Proof of Lemma \ref{eta>0 both}] Provided in the Appendix A.7.
\end{proof}
%We apply the joint CDF $ F_{N,x,y}(u,w) $ in Eq.(\ref{CDF 2}) to calculate its expectation. For $ N\geq\widetilde{N}_{0} $, and 
The domain of the joint CDF $ F_{N,x,y}(u,w) $ is $ u,w \in [1,\infty) $, which can be analytically extended on $ u,w \in [0,\infty) $. Using the non-negativity of CDF, we have %the integration on the bigger domain is greater or equal than the region $ [1,\infty ) $, i.e.
\begin{eqnarray}
\int_{(1,\infty)}\int_{(1,\infty)}\big(\log^{2} u\big)\big(\log^{2} w  \big) dF_{N,x,y}(u,w) &\leq& \int_{(0,\infty)}\int_{(0,\infty)}\big(\log^{2} u\big)\big(\log^{2} w  \big)  dF_{N,x,y}(u,w) \nonumber\\
&=& \mathsf{E}\big[\big(\log^{2}\eta^{y}_{N,x}\big)\big(\log^{2}\eta^{x}_{N,y}  \big)\big] \,. \nonumber
\end{eqnarray} 
Therefore, the proof of $ \mathsf{E}\big[\big(\log^{2}\eta^{y}_{N,x}\big)\big(\log^{2}\eta^{x}_{N,y}  \big)\big]<\infty$ is sufficient to guarantee the finiteness of the expectation on the truncated domain. 
From Eqs.(\ref{rNu,rNw,BxBy})-(\ref{CDF 4 infty}), we have the independence of $\big\{\zeta_{i\,|\,j}(N)\big\}_{i,j\in\{1,2,\ldots,N\}}  $ when $ N\geq\widetilde{N}_{0} $. Note that any functions of independent variables are also independent. Thus, we have
\begin{eqnarray}
\mathsf{E}\Big[\big(\log^{2}\eta^{y}_{N,x}\big)\big(\log^{2}\eta^{x}_{N,y}  \big)\Big]
%%&=&\int_{(0,\infty)}\int_{(0,\infty)}\big(\log u\big)^{2}\big(\log w  \big)^{2}  dF_{N,x,y}(u,w) \nonumber\\
%%&=&\mathsf{E}\Big(\big(\log e^{\zeta_{i}(N)}\big)^{2}\big(\log e^{\zeta_{j}(N)}\big)^{2} \Big| X_{i}=x, X_{j}=y \Big) \nonumber\\
&=&\mathsf{E}\Big({\zeta_{1\,|\,2}^{2}(N)}{\zeta_{2\,|\,1}^{2}(N)} \Big| X_{1}=x, X_{2}=y \Big) \nonumber \\
%&=&\mathsf{E}\Big({\zeta_{i}(N)}^{2}\Big| X_{i}=x \Big) \mathsf{E}\Big({\zeta_{j}(N)}^{2} \Big| X_{j}=y \Big) \label{expectation joint N} \\
&=&\mathsf{E}\Big({\zeta_{1\,|\,2}^{2}(N)}\Big| X_{1}=x, X_{2}=y \Big) \mathsf{E}\Big({\zeta_{2\,|\,1}^{2}(N)} \Big| X_{1}=x, X_{2}=y \Big) \nonumber \\
&:=&\mathsf{E}\big(\log^{2}\eta^{y}_{N,x}\big)\mathsf{E}\big(\log^{2}\eta^{x}_{N,y}  \big) \,. \label{expectation joint N}
\end{eqnarray}
%where  $ \zeta_{1}(N) $ and $ \zeta_{2}(N) $ represent $ \big\{{\zeta_{1}(N)}\big| X_{1}=x \big\}$ and $ \big\{{\zeta_{2}(N)}\big| X_{2}=y \big\}$ respectively. 
Clearly, the independence holds for logarithm and power transformations:
%functions are continuous mappings and thus, the %in their domain, and thus based on the independence in Eqs.(\ref{CDF 4 infty}) and (\ref{eta x, eta y}),we can deduce
\begin{eqnarray}
\eta_{x}\perp \eta_{y}\quad \Rightarrow \quad \log \eta_{x}\perp \log \eta_{y}\quad \text{and}\quad\big(\log^{2} \eta_{x}\big)\perp \big(\log^{2} \eta_{y}\big)\,.\label{expectation joint independence}
\end{eqnarray}
Thus, we have
\begin{eqnarray} \label{expectation joint independence 2}
\mathsf{E}\Big[\big(\log^{2}\eta_{x}\big)\big(\log^{2}\eta_{y}  \big)\Big]&=&\mathsf{E}\big(\log^{2}\eta_{x}\big)\mathsf{E}\big(\log^{2}\eta_{y}  \big) \nonumber\\
&=&\bigg(\log^{2}f(x)+\frac{\pi^{2}}{6}\bigg)\bigg(\log^{2}f(y)+\frac{\pi^{2}}{6}\bigg)\,.
\end{eqnarray}
We need to prove that
\begin{eqnarray}\label{expectation joint N limit}
\lim\limits_{N\rightarrow\infty} \mathsf{E}\Big[\big(\log^{2}\eta^{y}_{N,x}\big)\big(\log^{2}\eta^{x}_{N,y}  \big)\Big]=\mathsf{E}\Big[\big(\log^{2}\eta_{x}\big)\big(\log^{2}\eta_{y}  \big)\Big]\,.
\end{eqnarray}
Let us define $ \mathcal{A}_{1,M}:=\{(x,y)\in\mathcal{A}_{1}: \rho(x,y)>M\}$ and to prove the equality in Eq.(\ref{expectation joint N limit}), we need to show that $ \forall\, M>0 $ and all $ (x,y)\in\mathcal{A}_{1,M} $, the following two properties hold.
\begin{enumerate}[(i)]
\item Convergence in law when $N\rightarrow \infty  $, i.e.
\begin{eqnarray} \label{convergence covariance squared 1}
\mathsf{E}\Big[\big(\log\eta^{y}_{N,x}\big)^{2}\big(\log\eta^{x}_{N,y}  \big)^{2}\Big]\xrightarrow[N\rightarrow \infty]{ law}\, \bigg(\log^{2}f(x)+\frac{\pi^{2}}{6}\bigg)\bigg(\log^{2}f(y)+\frac{\pi^{2}}{6}\bigg)\,.\label{convergence in law expectation joint}
\end{eqnarray}
\item Uniform integrability of $ \Big\{\big(\log\eta^{y,i}_{N,x}\big)^{2}\big(\log\eta^{x,i}_{N,y}  \big)^{2}\Big\} $\,.
\end{enumerate}
%%%%%%%%%%%%%%%%%%%%%%%%%%%%%%%%%%%%%%%%%%%%%%%%%%%%%%%%%%%%%%%%%%%%%%%%%%%%%%%%%%%%%%%%%%%%%%%%%%%%%%%%%%%%%%%%%%%%%%%%%%%%%%%%%%%%
\subsubsection{Proof of convergence in law (i)}
%%%%%%%%%%%%%%%%%%%%%%%%%%%%%%%%%%%%%%%%%%%%%%%%%%%%%%%%%%%%%%%%%%%%%%%%%%%%%%%%%%%%%%%%%%%%%%%%%%%%%%%%%%%%%%%%%%%%%%%%%%%%%%%%%%%%
Under the continuous mapping, random variables' convergence in law is preserved. From the convergence in law in Eq.(\ref{convergency in law joint expectation}) and the independence in Eq.(\ref{expectation joint independence}), we can conclude
\begin{eqnarray}\label{E eta_x^y,2 eta_y^x,2}
%\mathsf{E}\Big[\big(\log\eta^{y}_{N,x}\big)^{2}\big(\log\eta^{x}_{N,y}  \big)^{2}\Big]&=&\mathsf{E}\Big[\big(\log\eta^{y}_{N,x}\big)^{2}\big(\log\eta^{x}_{N,y}  \big)^{2}\Big]\,,\quad\quad N\longrightarrow\infty \nonumber\\
\mathsf{E}\Big[\big(\log^{2}\eta^{y}_{N,x}\big)\big(\log^{2}\eta^{x}_{N,y}  \big)\Big]&\xrightarrow[\text{}]{law}\,&\mathsf{E}\Big[\big(\log^{2}\eta_{x}\big)\big(\log^{2}\eta_{y}  \big)\Big]\,,\quad\quad N\longrightarrow\infty \nonumber\\
&=&\mathsf{E}\big(\log^{2}\eta_{x}\big)\mathsf{E}\big(\log^{2}\eta_{y}  \big) \nonumber\\
&=&\bigg(\log^{2}f(x)+\frac{\pi^{2}}{6}\bigg)\bigg(\log^{2}f(y)+\frac{\pi^{2}}{6}\bigg)\,.
\end{eqnarray}
Thus, we proved the Eq.(\ref{convergence in law expectation joint}).
%%%%%%%%%%%%%%%%%%%%%%%%%%%%%%%%%%%%%%%%%%%%%%%%%%%%%%%%%%%%%%%%%%%%%%%%%%%%%%%%%%%%%%%%%%%%%%%%%%%%%%%%%%%%%%%%%%%%%%%%%%%%%%%%%%%%%%%%
\subsubsection{Proof of uniformly integrability (ii)}\label{section UI power4 Cov}
%%%%%%%%%%%%%%%%%%%%%%%%%%%%%%%%%%%%%%%%%%%%%%%%%%%%%%%%%%%%%%%%%%%%%%%%%%%%%%%%%%%%%%%%%%%%%%%%%%%%%%%%%%%%%%%%%%%%%%%%%%%%%%%%%%%%
According to the de la Valle Poussin theorem, proving the uniform integrability of \\
$ \Big\{\big(\log\eta^{y}_{N,x}\big)^{2}\big(\log\eta^{x}_{N,y}  \big)^{2}\Big\} $ is sufficient to show that for $ \mu-$ almost every $ x,y \in S(f) $, a positive $ C_{0}(x,y) $ and $ N_{0}(x,y)\in\mathbb{N} $.
\begin{eqnarray}\label{uniform integrability joint}
\sup_{N\geq N_{0(x,y)}} \mathsf{E} G \Big[\Big|\big(\log^{2} \eta^{y}_{N,x}\big)\big(\log^{2} \eta^{x}_{N,y}  \big)\Big|\Big]\leq C_{0}(x,y)<\infty \,.
\end{eqnarray}
$ G(t) $ is defined in Eq.(\ref{G(t)}), which is a convex and non-decreasing function on $ [0,\infty) $, and $ {G(t)}/{t} \rightarrow \infty $ as $ t\rightarrow \infty $.
%\subsection{Step 3: Proof of the validity of Eq.(\ref{uniform integrability joint})}
Note that $ \forall\, a, b\in\mathbb{R} $, we have $  
\big| ab \big|\leq \dfrac{a^{2}+b^{2}}{2} 
$ 
and according to Jenson Inequality, for convex function $ G(\cdot) $ and together with its expectation, we have:
%%\begin{eqnarray}
%%G\Big(\dfrac{a^{2}+b^{2}}{2}\Big)\leq \dfrac{G(a^{2})+G(b^{2})}{2} 
%%\end{eqnarray}
%%Thus,
\begin{eqnarray}\label{EG eta_x^y,2 eta_y^x,2}
\mathsf{E} G \Big[\Big| \big(\log^{2}\eta^{y}_{N,x}\big)\big(\log^{2}\eta^{x}_{N,y}  \big)\Big|\Big] 
%%&\leq & \mathsf{E} G \bigg[\dfrac{\big(\ln\eta^{y}_{N,x}\big)^{4}+\big(\ln\eta^{x}_{N,y}  \big)^{4}}{2}\bigg] \nonumber\\
 &\leq &\dfrac{\mathsf{E} G \big(\log^{4}\eta^{y}_{N,x}\big)+\mathsf{E} G \big(\log^{4}\eta^{x}_{N,y}  \big)}{2} \,.
\end{eqnarray}
%Let us derive $ \mathsf{E} G \big(\log\eta^{y}_{N,x}\big)^{4} $ and $ \mathsf{E} G \big(\log\eta^{x}_{N,y}  \big)^{4} $ as follows.
From the definition of $ G(\cdot) $ in Eq.(\ref{G(t)}), one has that $  G \big(\log^{4}\eta^{y}_{N,x}\big)= \big(\log^{4}\eta^{y}_{N,x}\big)\log \big(\log^{4}\eta^{y}_{N,x}\big)$ if $ \eta^{y}_{N,x}\in\big(0,\frac{1}{e}\big]\cup \big[e,\infty) $ and $  G \big(\log^{4}\eta^{y}_{N,x}\big)=0$ otherwise.
%%\begin{eqnarray}\label{}
%%  G \big(\log\eta^{y}_{N,x}\big)^{4} = \begin{cases}
%%    0, & \eta^{y}_{N,x} \in \big(\frac{1}{e},e\big); \\[2\jot]
%%     \big(\log\eta^{y}_{N,x}\big)^{4} \ln \Big[ \big(\log\eta^{y}_{N,x}\big)^{4}\Big], & \eta^{y}_{N,x}\in\big(0,\frac{1}{e}\big]\cup \big[e,\infty)  
%%  \end{cases}   \nonumber
%%\end{eqnarray}
We apply the results of \textit{Lemma} \ref{lemma power 4} to derive the expectation of $ G \big(\ln\eta^{y}_{N,x}\big)^{4} $ as follows.
\begin{eqnarray}
\mathsf{E} G \big(\log^{4}\eta^{y}_{N,x}\big)
%%&=&\int_{\big(0,\frac{1}{e}\big]}G \big(\log u\big)^{4} d F^{y}_{N,x}(u)+\int_{[e,\infty)}G \big(\log u\big)^{4} d F^{y}_{N,x}(u)\nonumber\\
&=&4\int_{\big(0,\frac{1}{e}\big]} (\log^{4} u)\log(-\log u) d F^{y}_{N,x}(u)+4\int_{[e,\infty)}(\log^{4} u)\log(\log u) d F^{y}_{N,x}(u) \nonumber \\
&=&4\underbrace{\int_{\big(0,\frac{1}{e}\big]}F^{y}_{N,x}(u) \frac{-\log^{3} u}{u}\bigg[\log(-\log u)+\frac{1}{4}\bigg] du }_{:={I}_{1}^{y}(N,x)} \nonumber\\
&& + 4\underbrace{\int_{[e,\infty)}\big[1-F^{y}_{N,x}(u)\big]\frac{\log^{3} u}{u}\bigg[\log(\log u)+\frac{1}{4}\bigg] d u }_{:={I}_{2}^{y}(N,x)}                \label{I1 y N,x} \\
&:=& 4\big[{I}_{1}^{y}(N,x)+{I}_{2}^{y}(N,x)\big] \,.\label{I1 x X,y}
\end{eqnarray}
%%with
%%\begin{eqnarray}
%%I_{1}^{y}(N,x)&=&\int_{\big(0,\frac{1}{e}\big]}F^{y}_{N,x}(u) \frac{(-\log u)^{3}}{u}\bigg[\log(-\log u)+\frac{1}{4}\bigg] du \label{I1 y N,x}\\
%%I_{2}^{y}(N,x)&=&\int_{[e,\infty)}\big[1-F^{y}_{N,x}(u)\big]\frac{(\log u)^{3}}{u}\bigg[\log(\log u)+\frac{1}{4}\bigg] d u \label{I1 x X,y}
%%\end{eqnarray}
%%%%%%%%%%%%%%%%%%%%%%%%%%%%%%%%%%%%%%%%%%%%%%%%%%%%%%%%%%%%%%%%%%%%%%%%%%%%%%%%%%%%%%%%%%%%%%%%%%%%%%%%%%%%%%%%%%%%%%%%%%%%%%%%%%%%%%%%%%%%
%%\subsubsection{Proof of $ I_{1}^{y}(N,x)<\infty $}\label{section I1 y N,x}
\subsubsection{Finiteness of $ I_{1}^{y}(N,x) $}\label{section I1 y N,x}
%%%%%%%%%%%%%%%%%%%%%%%%%%%%%%%%%%%%%%%%%%%%%%%%%%%%%%%%%%%%%%%%%%%%%%%%%%%%%%%%%%%%%%%%%%%%%%%%%%%%%%%%%%%%%%%%%%%%%%%%%%%%%%%%%%%%%%%%%%%%
%%Recall that $\mathcal{A}_{1,M}:=\{(x,y)\in\mathbb{A}_{1}: \rho(x,y)>M\} $ defined between the Eqs.(\ref{expectation joint N limit}) and (\ref{convergence covariance squared 1}). 
For $ (x,y)\in\mathcal{A}_{1,M} $, one has $\rho(x,y)>M  $. Thus, if $ M\geq r_{N}(u) $, one has $$ \mathds{1}\big[\rho(x, y)>r_{N}(u)\big]=1 $$\,.
%%the indicator functions in Eqs.(\ref{CDF 2}), 
%%(\ref{F Nx u 3}) and (\ref{F Ny w}) take value 1, i.e. $ \mathds{1}\Big[\rho(x, y)>r_{N}(u)\Big]=1 $.
%%\begin{eqnarray}
%%M\geq r_{N}(u)=\bigg[\dfrac{u}{V_{d}e^{\gamma}(N-1)}\bigg]^{d}  \, \Longrightarrow \, \mathds{1}\Big[\rho(x, y)>r_{N}(u)\Big]=1
%%\end{eqnarray}
Note that $ M\geq r_{N}(u) \Leftrightarrow N\geq \frac{u}{M^{d}V_{d}e^{\gamma}}+1 $. %%Clearly, $ N\geq 3$ needs to be satisfied and for $ u\in \big(0, \frac{1}{e}\big] $, we define $ N_{1} $:
We define
\begin{eqnarray}\label{N1 M}
N_{1}^{c}=N_{1}^{c}(M):=\max \bigg\{\bigg\lceil\frac{1}{M^{d}V_{d}e^{(\gamma +1)}}+1\bigg\rceil, \widetilde{N}_{0}\bigg\}\,.
\end{eqnarray}
$ N_{1}^{c}$ does not depend on $ x $ and $ y $. For $ N\geq N_{1}^{c}$, $ F^{y}_{N,x}(u) $ in Eq.(\ref{F Nx u 3}) is rearranged as:
\begin{eqnarray}
F^{y}_{N,x}(u) &=&1-\Big[1-\mathsf{P}_{N,x}(u)\Big]^{N-2} %\,,\qquad N\geq N_{1}  \label{F y Nx u 1} 
\\
&\leq & \big[(N-2)\mathsf{P}_{N,x}(u)\big]^{\varepsilon_{1}}\,,\qquad \varepsilon_{1}\in(0,1]\,,  \label{F y Nx u 2} \\
&\leq & \big[(N-2)r_{N}^{d}(u)V_{d}M_{f}(x,R_{1})\big]^{\varepsilon_{1}}\,,\qquad u\in(0,1/e]\,, \,\, N\geq {\widetilde{N}}_{1}\,, \label{F y Nx u 3} \\
&=& \bigg[\dfrac{N-2}{N-1}\cdot\dfrac{M_{f}(x,R_{1})}{e^{\gamma}}u\bigg]^{\varepsilon_{1}} \label{F y Nx u 4}\\
&\leq & \dfrac{M_{f}^{\varepsilon_{1}}(x,R_{1})}{e^{\gamma\varepsilon_{1}}}u^{\varepsilon_{1}} \label{F y Nx u 5}\,,\qquad \text{with `$=$' holds only when}\,\, N\rightarrow \infty\,,
\end{eqnarray}
where $r_{N}(u)$, $\mathsf{P}_{N,x}(u)$, $R_{1}$ and $ M_{f}(x,R_{1}) $ are defined in Eqs.(\ref{rNu e PNx})-(\ref{M_f(x,R_1)}). $ \widetilde{N}_{1} $ denotes the greater value of $ N_1$ in (\ref{N_1}) and $ N_{1}^{c} $ in (\ref{N1 M}).
\begin{eqnarray}\label{N1 M R1}
\widetilde{N}_{1}=\widetilde{N}_{1}(M, R_{1}):=\max \big\{N_{1}^{c}(M), N_{1}(R_{1})\big\}  \,.
\end{eqnarray}
Note that $ \widetilde{N}_{1} $ does not depend on $ x $ and $ y $. Analogous to the investigation in \text{Section \ref{section I1}}, we obtain
\begin{align}\label{I1 y N,x bounded}
I_{1}^{y}(N,x) <\infty\,.
\end{align}
%%%%%%%%%%%%%%%%%%%%%%%%%%%%%%%%%%%%%%%%%%%%%%%%%%%%%%%%%%%%%%%%%%%%%%%%%%%%%%%%%%%%%%%%%%%%%%%%%%%%%%%%%%%%%%%%%%%%%%%%%%%%%%%%%%%%%%%%%%%%
\subsubsection{Finiteness of $ I_{2}^{y}(N,x)<\infty $}\label{section I2 y N,x}
%%%%%%%%%%%%%%%%%%%%%%%%%%%%%%%%%%%%%%%%%%%%%%%%%%%%%%%%%%%%%%%%%%%%%%%%%%%%%%%%%%%%%%%%%%%%%%%%%%%%%%%%%%%%%%%%%%%%%%%%%%%%%%%%%%%%%%%%%%%%
We split the integral $ \int_{[e,\infty]} $ into $ \int_{[e,\sqrt{N-1}]}+\int_{(\sqrt{N-1},\infty)} $ as follows.
\begin{eqnarray}\label{I_2(N,x) 1}
I^{y}_{2}(N,x)%%&=&\Bigg(\int\limits_{[e,\sqrt{N-1}]}+\int\limits_{(\sqrt{N-1},\infty)}\Bigg)[1-F^{y}_{N,x}(u)] \dfrac{(\log u)^{3}}{u}\bigg[\log(\log u)+\frac{1}{4}\bigg]du \nonumber\\
&=& J^{y}_{1}(N,x)+J^{y}_{2}(N,x)
\end{eqnarray}
with
\begin{eqnarray}
J^{y}_{1}(N,x)&=&\int\limits_{[e,\sqrt{N-1}]}[1-F^{y}_{N,x}(u)] \dfrac{(\log u)^{3}}{u}\bigg[\log(\log u)+\frac{1}{4}\bigg]du\,, \label{J_1 y} \\
J^{y}_{2}(N,x)&=&\int\limits_{(\sqrt{N-1},\infty)}[1-F^{y}_{N,x}(u)] \dfrac{(\log u)^{3}}{u}\bigg[\log(\log u)+\frac{1}{4}\bigg]du\,. \label{J_2 y}
\end{eqnarray}
%%Note that the integral interval $ \int\limits_{[e,\sqrt{N-1}]} $ requires that $ e\leq \sqrt{N-1} $, which indicates $N \geq 10  $.
%%%%%%%%%%%%%%%%%%%%%%%%%%%%%%%%%%%%%%%%%%%%%%%%%%%%%%%%%%%%%%%%
%%%%%%%%%%%%%%%%%%%%%%%%%%%%%%%%%%%%%%%%%%%%%%%%%%%%%%%%%%%%%%%%
%%\paragraph{Proof of $J^{y}_{1}(N,x)<\infty  $}\label{proof of Jy1 N,x}
\paragraph{Finiteness of $J^{y}_{1}(N,x)$}\label{proof of Jy1 N,x}
%%%%%%%%%%%%%%%%%%%%%%%%%%%%%%%%%%%%%%%%%%%%%%%%%%%%%%%%%%%%%%%%
%%%%%%%%%%%%%%%%%%%%%%%%%%%%%%%%%%%%%%%%%%%%%%%%%%%%%%%%%%%%%%%%
From the Eq.(\ref{F Ny w}), we have
\begin{eqnarray}\label{1-F y Nx u}
1- F^{y}_{N,x}(u)&=& \mathds{1}\Big[\rho(x, y)>r_{N}(u)\Big] \Big[1-\mathsf{P}_{N,x}(u)\Big]^{N-2} \nonumber\\ 
&\leq & \Big[1-\mathsf{P}_{N,x}(u)\Big]^{N-2} \,.\label{1-F y Nx u} 
\end{eqnarray}
Note that $ \dfrac{1}{2}\leq \dfrac{N-2}{N-1}< 1 $ for $ N\in[3,\infty) $.
We define a finite number 
\begin{align}\label{N_2 tilde}
\widetilde{N}_{2}=\max\{N_{2},\widetilde{N}_{0}\}  
\end{align}
with $ N_{2} $ and $ \widetilde{N}_{0}  $ defined in Eqs.(\ref{N_2 R_2}) and (\ref{N_0 tilde}). For $ N\geq \widetilde{N}_{2} $, we can derive that
%%\begin{align}
%%\widetilde{N}_{2}=\max\{N_{2},\widetilde{N}_{0}\}
%%\end{align}
%%Substituting the inequality Eq.(\ref{(1-t)N<e-tN}) into (\ref{1-F y Nx u} ), we obtain
\begin{eqnarray}
1-F_{N,x}^{y}(u) %%&\leq & e^{-(N-2)\,\mathsf{P}_{N,x}(u)} \label{1-Fy_Nx(u)}\,,\qquad \text{with}\,\,\mathsf{P}_{N,x}(u)\in[0,1],\,\, N\geq 3 \\
%%&=& e^{\bigg\{-\dfrac{N-2}{N-1}\,\dfrac{u}{\tilde{\gamma}}\,\dfrac{\mathsf{P}_{N,x}(u)}{r_{N}^{d}(u)V_{d}}\bigg\}} \,,\qquad \text{with}\,\, r_{N}(u)=\bigg[\dfrac{u}{(N-1)V_{d}\tilde{\gamma}}\bigg]^{1/d}\nonumber\\
%%&\leq & e^{\Bigg\{-\dfrac{1}{2}\dfrac{u}{\tilde{\gamma}}\,\dfrac{\int_{B\big(x,r_{N}(u)\big)}f(\xi)d\xi}{r_{N}^{d}(u)V_{d}}\Bigg\}} \label{1-Fy_Nx(u) 0}\\
&\leq & \exp\bigg\{-\frac{N-2}{N-1}\dfrac{u}{e^{\gamma}}\,m_{f}(x,R_{2})\bigg\} \nonumber \\
&\leq & e^{-\dfrac{u}{2e^{\gamma}}\,m_{f}(x,R_{2})} %%\,,\qquad  N \geq N_{2}>10 \label{1-Fy_Nx(u) 1} 
\nonumber\\
&\leq & \bigg[\dfrac{u}{2e^{\gamma}}\,m_{f}(x,R_{2})\bigg]^{-\varepsilon}\,,\qquad\text{}\, \varepsilon\in(0,e]\,, \label{1-Fy_Nx(u) 2}
\end{eqnarray}
where $ R_{2} $ and $ m_{f}(x,R_{2}) $ are defined in Eqs.(\ref{R_2}) and (\ref{M_f e m_f}). Substituting Eq.(\ref{1-Fy_Nx(u) 2}) into (\ref{J_1 y}), we obtain
\begin{eqnarray}
J^{y}_{1}(N,x) %&\leq &  \dfrac{(2\tilde{\gamma})^{\varepsilon}}{m_{f}^{\varepsilon}(x,R_{2})}\,\int\limits_{[e,\sqrt{N-1}]}\dfrac{(\log u)^{3}\big[\log(\log u)+\frac{1}{4}\big]}{u^{1+\varepsilon}}du \label{J_1 y part2 eq1} \\
%\hspace{-3cm} &\leq &  \dfrac{(2\tilde{\gamma})^{\varepsilon}}{m_{f}^{\varepsilon}(x,R_{2})}\,\int\limits_{[e,\infty]}\dfrac{(\log u)^{3}\big[\log(\log u)+\frac{1}{4}\big]}{u^{1+\varepsilon}}du \label{J_1 y part2 eq2} \\
& = & \dfrac{(2e^{\gamma})^{\varepsilon}}{m_{f}^{\varepsilon}(x,R_{2})}\,L(\varepsilon)<\infty\,,\qquad\text{}\, \varepsilon\in(0,e]\,, \label{J_1 y part2 eq3}
%\\
%&<&\infty \nonumber   
\end{eqnarray}
where the function $L(\cdot)  $ is defined in Eq.(\ref{Lv1}). 
The detailed derivation is analogously to the analysis of $ J_{1}(N,x) $ in \text{Section \ref{subparagraph J_2 2.1}}.
%%where $ m_{f}(x,R_{2})= \inf\limits_{r\in(0,R_{2}]}\frac{\int_{B\big(x,r\big)}f(y)dy}{r^{d}V_{d}} > 0$ and $ R_{2}=\Big[\frac{1}{\sqrt{N-1} V_{d}\tilde{\gamma}}\Big]^{1/d} >\infty $ are defined in Eqs.(\ref{m_f(x,R_2)}) and (\ref{R2}), respectively. We re-write the definition of $ R_{2} $ to obtain
%%\begin{eqnarray}\label{N_2 eq2}
%% N_{2}=1+\frac{1}{({R_{2}^{d}V_{d}\tilde{\gamma}})^{2}} >10
%%\end{eqnarray}
%%We set $ N_{2}> 10 \Rightarrow e< \sqrt{N_{2}-1}$, which is in consistent with domain of $ J_{1}^{y}(N,x)$, i.e. $ u\in [e,\sqrt{N-1}] $. 
%%We used the following result to obtain Eq.( \ref{1-Fy_Nx(u) 1}).
%%\begin{eqnarray}
%%\dfrac{1}{2}\leq \dfrac{N-2}{N-1}< 1 \,,\qquad \forall N\in \mathbb {N} \,\,\text{and}\,\, 3<N<\infty 
%%\end{eqnarray}
%%The above result can be easily verified as the function $ g(N):=\frac{N-2}{N-1} = 1-\frac{1}{N-1}$ is monotonically increasing. 
%%The Eq.(\ref{1-Fy_Nx(u) 2}) is obtained by applying the \textit{Corollary} \ref{corollary e_-t < t_-delta}, where $R_{2}  $ and $ m_{f}(x,R_{2})>0$ are defined in Eqs.(\ref{R2}) and (\ref{m_f(x,R_2)}), together with their assumptions $ x\in\mathcal{D}_{f}(R_{2})$ and $N\geq N_{2}>10  $. 
%%%%%%%%%%%%%%%%%%%%%%%%%%%%%%%%%%%%%%%%%%%%%%%%%%%%%%%%%%%%%%%%
%%%%%%%%%%%%%%%%%%%%%%%%%%%%%%%%%%%%%%%%%%%%%%%%%%%%%%%%%%%%%%%%
%%\paragraph{Proof of $J_{2}^{y}(N,x)<\infty  $}\label{proof of J2y N,x}
\paragraph{Finiteness of $J_{2}^{y}(N,x)  $}\label{proof of J2y N,x}
%%%%%%%%%%%%%%%%%%%%%%%%%%%%%%%%%%%%%%%%%%%%%%%%%%%%%%%%%%%%%%%%
%%%%%%%%%%%%%%%%%%%%%%%%%%%%%%%%%%%%%%%%%%%%%%%%%%%%%%%%%%%%%%%%
For $ u\in[\sqrt{N-1},\infty) $, one has $\mathsf{P}_{N,x}(\sqrt{N-1})\leq \mathsf{P}_{N,x}(u)  $ and the Eq.(\ref{1-F y Nx u}) satisfies:
%%From the Eqs.(\ref{F Nx u 3}) and (\ref{1-F y Nx u}), we derive $ 1- F^{y}_{N,x}(u) $  for $ u\geq N-1 $ that
\begin{eqnarray}
1- F^{y}_{N,x}(u)%%&=& \mathds{1}\Big[\rho(x, y)>r_{N}(u)\Big] \Big[1-\mathsf{P}_{N,x}(u)\Big]^{N-2} \nonumber\\ 
%%&\leq & \Big[1-\mathsf{P}_{N,x}(u)\Big]^{N-2}  \nonumber\\
%%&= & \Big[1-\mathsf{P}_{N,x}(u)\Big]^{N-3} \Big[1-\mathsf{P}_{N,x}(u)\Big]\,,\qquad u\geq N-1  \nonumber\\
&\leq & \Big[1-\mathsf{P}_{N,x}(\sqrt{N-1})\Big]^{N-3} \Big[1-\mathsf{P}_{N,x}(u)\Big] \,.\label{1-Fy_Nx(u) v2 1}
\end{eqnarray}
Substituting Eq.(\ref{1-Fy_Nx(u) v2 1}) into %$J_{2}^{y}(N,x)$ in Eq.
(\ref{J_2 y}), we obtain
\begin{eqnarray}\label{J_2 2 y}
J_{2}^{y}(N,x)&\leq & %\big[1-\mathsf{P}_{N,x}(\sqrt{N-1})\big]^{N-3}\int_{(\sqrt{N-1},\infty)}\big[1-\mathsf{P}_{N,x}(u)\big] \dfrac{(\log u)^{3}\big[\log(\log u)+\frac{1}{4}\big]}{u}du\nonumber\\
%&=& 
J_{2.1}^{y}(N,x)\cdot J_{2.2}^{y}(N,x)
\end{eqnarray}
with
\begin{eqnarray}
J_{2.1}^{y}(N,x)&:=&\big[1-\mathsf{P}_{N,x}(\sqrt{N-1})\big]^{N-3}\,,\label{J_2 2.1 1 y}\\
J_{2.2}^{y}(N,x)&:=&\int_{(\sqrt{N-1},\infty)}\big[1-\mathsf{P}_{N,x}(u)\big] \dfrac{(\log u)^{3}\big[\log(\log u)+\frac{1}{4}\big]}{u}du\,. \label{J_2 2.2 1 y}
\end{eqnarray}
%%%%%%%%%%%%%%%%%%%%%%%%%%%%%%%%%%%%%%%%%%%%%%%%%%%%%%%%%%%%%%%%%%%
%%%%%%%%%%%%%%%%%%%%%%%%%%%%%%%%%%%%%%%%%%%%%%%%%%%%%%%%%%%%%%%%%%%
\subparagraph{Investigation of $J_{2.1}^{y}(N,x)$}
%%%%%%%%%%%%%%%%%%%%%%%%%%%%%%%%%%%%%%%%%%%%%%%%%%%%%%%%%%%%%%%%%%%
%%%%%%%%%%%%%%%%%%%%%%%%%%%%%%%%%%%%%%%%%%%%%%%%%%%%%%%%%%%%%%%%%%%
Note that the inequality $ \dfrac{1}{3}\leq \dfrac{N-3}{N-1}< 1 $ holds for any $ N\in(3,\infty)$. For $ N\geq\widetilde{N}_{2} $, we can deduce the following results%, which is analogously to the analysis of $ J_{2.1}(N,x) $ in \textbf{Section \ref{subparagraph J_2 2.2}}.
\begin{eqnarray}
J_{2.1}^{y}(N,x)&\leq &  
\underbrace{\dfrac{(3e^{\gamma})^{\varepsilon}}{m_{f}^{\varepsilon}(x,R_{2})}\,(N-1)^{-\varepsilon/2}}_{=J_{2.1}^{y,\ast}(N,x) }
,\qquad\text{}\, \varepsilon\in(0,e] \nonumber \\
&\leq &\dfrac{(3e^{\gamma})^{\varepsilon}}{m_{f}^{\varepsilon}(x,R_{2})}\,(\widetilde{N}_{2}-1)^{-\varepsilon/2}%,\qquad N\geq \widetilde{N}_{2}  \nonumber \\
< \infty\,.  \label{J_2 2.1 1 y final}
\end{eqnarray}
The detailed deduction is analogously to the analysis of $ J_{2.1}(N,x) $ in \text{Section \ref{subparagraph J_2 2.2}}.%% is shown in Appendix A.4.
\subparagraph{Investigation of $J_{2.2}^{y}(N,x)$}
%%%%%%%%%%%%%%%%%%%%%%%%%%%%%%%%%%%%%%%%%%%%%%%%%%%%%%%%%%%%%%%%
%%%%%%%%%%%%%%%%%%%%%%%%%%%%%%%%%%%%%%%%%%%%%%%%%%%%%%%%%%%%%%%%
Note that the settings of  $ J_{2.2}^{y}(N,x) $ is $ N\geq 3 $, which is different from the setting of $ J_{2.2}(N,x) $ in Eq.(\ref{J_2 2.2 1}). The detailed investigation, analogously to the analysis in the \text{Section \ref{subparagraph J_2 2.2}}, which is based on $ N=2 $, can be obtained by defining $ \tilde{u}=\dfrac{2u}{N-1}$ and analysing its associated change of variables as follows.
%% although the function $ J_{2.2}^{y}(N,x) $ in Eq.(\ref{J_2 2.2 1 y}) looks identical to the function $ J_{2.2}(N,x) $ in Eq.(\ref{J_2 2.2 1}), the settings of the distribution function  $ F^{y}_{N,x}(u) $ are different from $ F_{N,x}(u) $. In particular, we require $ N\geq 3 $ for analysing the covariance, while we choose $ N=2 $ in section \textbf{\ref{subparagraph J_2 2.2} Analytical treatment of $J_{2.2}(N,x)$}.
%%We define $ \tilde{u}=\dfrac{2u}{N-1}$ and the change of variables leads to the following remarkable results.
\begin{enumerate}[(i)]
%\item  $  \tilde{u}=\dfrac{2u}{N-1}\in \Big[\frac{2}{\sqrt{N-1}},\infty\Big) \quad \Longleftrightarrow \quad u\in[\sqrt{N-1},\infty)\,,\quad N\geq 3$
\item %%$ \tilde{u}=\dfrac{2u}{N-1} $ and $ \tilde{u}\in \Big[\frac{2}{\sqrt{N-1}},\infty\Big) $, i.e.
%%\begin{eqnarray}\label{tilde u domain}
$
u\in[\sqrt{N-1},\infty)\quad \Longleftrightarrow \quad \tilde{u}=\dfrac{2u}{N-1}\in \Big[\frac{2}{\sqrt{N-1}},\infty\Big)\,,\quad N\geq 3\,,
$
%\end{eqnarray}
\item %%$  r_{N}(u)=r_{3}(\tilde{u})$, i.e.
%%\begin{eqnarray}\label{rN,u=rN,tilde u}
$
r_{N}(u)=\bigg[\dfrac{u}{(N-1)V_{d}\tilde{\gamma}}\bigg]^{1/d} = \bigg[\dfrac{\tilde{u}}{2V_{d}\tilde{\gamma}}\bigg]^{1/d}=r_{3}(\tilde{u})
$\,,
%%\end{eqnarray}
%$ r_{N}(u)=\bigg[\dfrac{u}{(N-1)V_{d}\tilde{\gamma}}\bigg]^{1/d} = \bigg[\dfrac{\tilde{u}}{2V_{d}\tilde{\gamma}}\bigg]^{1/d}=r_{3}(\tilde{u}) $
%\item $ r_{N}(u)=\bigg[\dfrac{u}{(N-1)V_{d}\tilde{\gamma}}\bigg]^{1/d} = \bigg[\dfrac{\tilde{u}}{V_{d}\tilde{\gamma}}\bigg]^{1/d}=r_{2}(\tilde{u})=  2^{1/d}\bigg[\dfrac{\tilde{u}}{2V_{d}\tilde{\gamma}}\bigg]^{1/d}=2^{1/d}r_{3}(\tilde{u}) $
%Therefore, we have $r_{3}(\tilde{u})=\dfrac{r_{2}(\tilde{u})}{2^{1/d}}\leq r_{2}(\tilde{u}) $, since $ 2^{1/d}\geq 1 $
\item %%$\mathsf{P}_{N,x}(u) =\mathsf{P}_{3,x}(\tilde{u}) $, i.e.
%%\begin{eqnarray}\label{PN,u=PN,tilde u}
$\mathsf{P}_{N,x}(u) =\int_{B(x, r_{N}(u))}f(\xi)d\xi=\int_{B(x, r_{3}(\tilde{u}))}f(\xi)d\xi =\mathsf{P}_{3,x}(\tilde{u}) 
%%\end{eqnarray}
$\,,
%\item $ \mathsf{P}_{N,x}(u)=\int_{B(x, r_{N}(u))}f(\xi)d\xi =\int_{B(x, r_{3}(\tilde{u}))}f(\xi)d\xi =\mathsf{P}_{3,x}(\tilde{u})$, with $  r_{3}(\tilde{u}) $ defined in (ii)
%\item $ \mathsf{P}_{N,x}(u)=\int_{B(x, r_{N}(u))}f(\xi)d\xi =\int_{B(x, r_{2}(\tilde{u}))}f(\xi)d\xi =\mathsf{P}_{2,x}(\tilde{u})\geq \mathsf{P}_{3,x}(\tilde{u})=\int_{B(x, r_{3}(\tilde{u}))}f(\xi)d\xi $, with $  r_{N}(u) $ defined in (ii)
%Therefore, $ 0\leq 1- \mathsf{P}_{2,x}(\tilde{u}) \leq 1- \mathsf{P}_{3,x}(\tilde{u}) \leq 1$
\item $ F^{y}_{N,x}(u)=F^{y}_{3,x}(\tilde{u}) $\,.%%, i.e.
%%\begin{eqnarray}\label{F y N=3,x tilde u}
%%F^{y}_{N,x}(u)&=&1- \mathds{1}\Big[\rho(x, y)>r_{N}(u)\Big] \Big[1-\mathsf{P}_{N,x}(u)\Big]^{N-2} \,,\qquad N\geq 3\nonumber\\ 
%%&=&1- \mathds{1}\Big[\rho(x, y)>r_{3}(\tilde{u})\Big] \Big[1-\mathsf{P}_{3,x}(\tilde{u})\Big] \nonumber\\ 
%%&=& F^{y}_{3,x}(\tilde{u})
%%\end{eqnarray}
\end{enumerate} 
We substitute  $ \tilde{u}=\dfrac{2u}{N-1} $ into Eq.(\ref{J_2 2.2 1 y}) and split the integrals as follows.
%%the integral $ (\sqrt{N-1},\infty) $ becomes $(\frac{2}{\sqrt{N-1}},\infty)  $ , which can be split as follows
%%\begin{eqnarray}\label{intergal split J_2 2.2 y}
%%\bigg(\dfrac{2}{\sqrt{N-1}},\infty\bigg)= \bigg(\frac{2}{\sqrt{N-1}},e\bigg]\cup (e,\infty)  
%%\end{eqnarray}
%%The integral split requires $ \frac{2}{\sqrt{N-1}}\leq e $ or in other words $ N\geq 1+ \big(\frac{2}{e}\big)^{2} $. This requirement is clearly satisfied for any $ N\geq 2 $. Substituting Eq.(\ref{intergal split J_2 2.2 y}) the results in (i)-(iv) into Eq.(\ref{J_2 2.2 1 y}), we obtain
\begin{align}
J^{y}_{2.2}(N,x)=\underbrace{\int\limits_{(\frac{2}{\sqrt{N-1}},e]}\big[1-\mathsf{P}_{3,x}(\tilde{u})\big] \widetilde{\mathbb{M}}(N,\tilde{u})\dfrac{d\tilde{u}}{\tilde{u}} }_{:=J_{2.2.1}^{y}(N,x)}+\underbrace{\int\limits_{(e,\infty)}\big[1-\mathsf{P}_{3,x}(\tilde{u})\big] \widetilde{\mathbb{M}}(N,\tilde{u})\dfrac{d\tilde{u}}{\tilde{u}} }_{：=J_{2.2.2}^{y}(N,x)}   \label{J_2 2.2 1 y eq2}
\end{align}
\begin{comment}
\begin{equation} 
\resizebox{.999\textwidth}{!}
{
%
  $
  \begin{aligned} 
  \hspace{-.2cm} J^{y}_{2.2}(N,x) {}
   & =\int_{(\frac{2}{\sqrt{N-1}},\infty)}\big[1-\mathsf{P}_{3,x}(\tilde{u})\big] \widetilde{\mathbb{M}}(N,\tilde{u})\dfrac{d\tilde{u}}{\tilde{u}}    \\
   & =\Bigg(\int\limits_{(\frac{2}{\sqrt{N-1}},e]}+\int\limits_{(e,\infty)}\Bigg)\big[1-\mathsf{P}_{3,x}(\tilde{u})\big]\, \widetilde{\mathbb{M}}(N,\tilde{u})\,\dfrac{d\tilde{u}}{\tilde{u}}    \\
   & := J_{2.2.1}^{y}(N,x)+J_{2.2.2}^{y}(N,x) \label{J_2 2.2 3 y}
  \end{aligned}
  $
%
}
\end{equation}
\end{comment}
with $ \widetilde{\mathbb{M}}(N,\tilde{u})=\log^{3} \dfrac{(N-1)\tilde{u}}{2}\bigg[\log\log\dfrac{(N-1)\tilde{u}}{2}+\dfrac{1}{4}\bigg] $ and
{\normalsize 
\begin{eqnarray} 
&&\hspace{-0cm} J_{2.2.1}^{y}(N,x)=\int_{(\frac{2}{\sqrt{N-1}},e]}\big[1-\mathsf{P}_{3,x}(\tilde{u})\big] \widetilde{\mathbb{M}}(N,\tilde{u})\dfrac{d\tilde{u}}{\tilde{u}}\,,  \label{J_2 2.2.1 eq1 y}\\
&&\hspace{-0cm} J_{2.2.2}^{y}(N,x)=\int_{(e,\infty)}\big[1-\mathsf{P}_{3,x}(\tilde{u})\big] \widetilde{\mathbb{M}}(N,\tilde{u})\dfrac{d\tilde{u}}{\tilde{u}} \,. \label{J_2 2.2.2 eq1 y} 
%%\\
%%&& \widetilde{\mathbb{M}}(N,\tilde{u})=\log^{3} \frac{(N-1)\tilde{u}}{2}\bigg[\log\log\frac{(N-1)\tilde{u}}{2}+\frac{1}{4}\bigg] \label{M,N,u y} 
\end{eqnarray}
}
Note that the integral $ \int_{(\frac{2}{\sqrt{N-1}},e]} $ requires $ \frac{2}{\sqrt{N-1}}\leq e $, and it is satisfied for any $ N\geq 2 $.\\[.3cm]
%%%%%%%%%%%%%%%%%%%%%%%%%%%%%%%%%%%%%%%%%%%%%%%%%%%%%%%%%%%%%%%%
%%%%%%%%%%%%%%%%%%%%%%%%%%%%%%%%%%%%%%%%%%%%%%%%%%%%%%%%%%%%%%%%
\textbf{(I) Investigation of $J_{2.2.1}^{y}(N,x)  $}\\[.3cm]
%%%%%%%%%%%%%%%%%%%%%%%%%%%%%%%%%%%%%%%%%%%%%%%%%%%%%%%%%%%%%%%%
%%%%%%%%%%%%%%%%%%%%%%%%%%%%%%%%%%%%%%%%%%%%%%%%%%%%%%%%%%%%%%%%
The cumulative distribution function $ \mathsf{P}_{3,x}(\tilde{u})$ is bounded between 0 and 1, and hence
\begin{eqnarray}\label{sup 1-f_2,x N eq1 y}
\sup\limits_{\widetilde{u}\in(\frac{2}{\sqrt{N-1}},e]}\big\{1-\mathsf{P}_{3,x}(\tilde{u})\big\} \leq \sup\limits_{\tilde{u}\in[0,\infty)}\big\{1-\mathsf{P}_{3,x}(\tilde{u})\big\}= 1 \,.
\end{eqnarray}
From the monotonicity of logarithmic and power functions, we have
\begin{eqnarray}%\label{sup Polynomials log(N-1)w eq1 y}
\sup\limits_{\widetilde{u}\in(\frac{2}{\sqrt{N-1}},e]} \widetilde{\mathbb{M}}(N,\tilde{u})&=&\log^{3}\frac{e(N-1)}{2}\bigg[\log\log \frac{e(N-1)}{2}+\dfrac{1}{4}\bigg]\nonumber\\
\hspace{-0.6cm}&:=&\widetilde{\mathbb{M}}^{\ast}_{e}(N)\,. \label{M,N,max y}
\end{eqnarray}
Substituting Eq.(\ref{sup 1-f_2,x N eq1 y}) and (\ref{M,N,max y}) into (\ref{J_2 2.2.1 eq1 y}), we obtain
\begin{eqnarray}
\hspace{-.0cm} J_{2.2.1}^{y}(N,x) &\leq &  \mathbb{\widetilde{M}}^{\ast}_{e}(N) \int\limits_{(\frac{2}{\sqrt{N-1}},e]} \dfrac{d\tilde{u}}{\tilde{u}} \nonumber\\
%%\hspace{-0.0cm}&=&\mathbb{\tilde{M}}^{\ast}_{e}(N)\,\bigg[\log \bigg(\dfrac{e\sqrt{N-1}}{2}\bigg)\bigg]  \label{J_2 2.2.1 eq3 y} \\
&=&\log^{3} \frac{e(N-1)}{2}\bigg[\log\log \frac{e(N-1)}{2}+\frac{1}{4}\bigg] \,\log \dfrac{e\sqrt{N-1}}{2}   \label{J_2 2.2.1 eq3 final y} \nonumber \\
&:=& J_{2.2.1}^{y,\ast}(N)\,. \label{J_2 2.2.1 eq3 final y} 
\end{eqnarray}
%%\\[.3cm]
%%%%%%%%%%%%%%%%%%%%%%%%%%%%%%%%%%%%%%%%%%%%%%%%%%%%%%%%%%%%%%%%%
%%%%%%%%%%%%%%%%%%%%%%%%%%%%%%%%%%%%%%%%%%%%%%%%%%%%%%%%%%%%%%%%%
\textbf{(II) Investigation of $J_{2.2.2}^{y}(N,x)  $}\\[.3cm]
%%%%%%%%%%%%%%%%%%%%%%%%%%%%%%%%%%%%%%%%%%%%%%%%%%%%%%%%%%%%%%%%%
%%%%%%%%%%%%%%%%%%%%%%%%%%%%%%%%%%%%%%%%%%%%%%%%%%%%%%%%%%%%%%%%%
Analogously to the investigation of $ J_{2.2.2}(N,x) $ in Section \ref{subparagraph J_2 2.2}, we split the integral $ \int_{(e,\infty)}  $ into $  \int_{(e,e^{1+\Delta}]}+\int_{(e^{1+\Delta},\infty)}$, where $ \Delta$ represents a small positive number.
\begin{eqnarray}
J_{2.2.2}^{y}(N,x) %%&&=\int\limits_{(e,\infty)}\big[1-\mathsf{P}_{3,x}(\tilde{u})\big] \, \widetilde{\mathbb{M}}(N,\tilde{u})\,\dfrac{d\tilde{u}}{\tilde{u}}\nonumber\\
%%&&\hspace{-3cm}= \bigg(\int\limits_{(e,e^{1+\Delta}]}+\int\limits_{(e^{1+\Delta},\infty)} \bigg)\, \tilde{\mathbb{M}}(N,\tilde{u})\,\dfrac{d\tilde{u}}{\tilde{u}} \nonumber\\
%%&&\hspace{-3cm}
=J_{2.2.2}^{y,(1)}(N,x)+J_{2.2.2}^{y,(2)}(N,x)  \label{J_2 2.2.2 eq2 y}
\end{eqnarray}
with
\begin{eqnarray}
&&\hspace{-1.6cm} J_{2.2.2}^{y,(1)}(N,x)=\int_{(e,e^{1+\Delta}]}[1-\mathsf{P}_{3,x}(\tilde{u})\big]\, \widetilde{\mathbb{M}}(N,\tilde{u})\,\dfrac{d\tilde{u}}{\tilde{u}}\,, \label{J_2 2.2.2 (1) y}\\
&&\hspace{-1.6cm} J_{2.2.2}^{y,(2)}(N,x)=\int_{(e^{1+\Delta},\infty)}[1-\mathsf{P}_{3,x}(\tilde{u})\big]\, \widetilde{\mathbb{M}}(N,\tilde{u})\,\dfrac{d\tilde{u}}{\tilde{u}}\,. \label{J_2 2.2.2 (2) y}
\end{eqnarray}
%%%%%%%%%%%%%%%%%%%%%%%%%%%%%%%%%%%%%%%%%%%%%%%%%%%%%%%%%%%%%%%%%%
%%%%%%%%%%%%%%%%%%%%%%%%%%%%%%%%%%%%%%%%%%%%%%%%%%%%%%%%%%%%%%%%%%
\textbf{(II.1) Investigation of $J_{2.2.2}^{y,(1)}(N,x) $}\\[.3cm]
%%%%%%%%%%%%%%%%%%%%%%%%%%%%%%%%%%%%%%%%%%%%%%%%%%%%%%%%%%%%%%%%%%
%%%%%%%%%%%%%%%%%%%%%%%%%%%%%%%%%%%%%%%%%%%%%%%%%%%%%%%%%%%%%%%%%%
From the monotonicity of logarithmic and power functions, we have
%%We define $ \tilde{\mathbb{M}}^{\ast}_{e^{1+\Delta}}(N) $ as follows
\begin{eqnarray}
\sup\limits_{\tilde{u}\in(e,e^{1+\Delta}]}\, \widetilde{\mathbb{M}}(N,\tilde{u}) &=&\log^{3} \frac{e^{1+\Delta}(N-1)}{2}\bigg[\log\log \frac{e^{1+\Delta}(N-1)}{2}+\frac{1}{4}\bigg]  \nonumber\\
&:=&\widetilde{\mathbb{M}}^{\ast}_{e^{1+\Delta}}(N)\,. \label{sup J_2.2.2 (1) part1 y}
\end{eqnarray}
Substituting Eqs.(\ref{sup 1-f_2,x N eq1 y}) and (\ref{sup J_2.2.2 (1) part1 y}) into (\ref{J_2 2.2.2 (1) y}), we obtain
\begin{eqnarray}
 J_{2.2.2}^{y,(1)}(N,x)&\leq & \widetilde{\mathbb{M}}^{\ast}_{e^{1+\Delta}}(N) \int_{(e,e^{1+\Delta}]}\dfrac{d\tilde{u}}{\tilde{u}} \nonumber\\
%% &=& \Delta\,\tilde{\mathbb{M}}^{\ast}_{e^{1+\Delta}}(N) \nonumber \\
  &=&  %%\bigg[\log \bigg(\frac{e^{1+\Delta}}{2}(N-1)\bigg)\bigg]^{3}\bigg[\log\bigg[\log \bigg(\frac{e^{1+\Delta}}{2}(N-1)\bigg)\bigg]+\frac{1}{4}\bigg] \nonumber
  \log^{3} \frac{e^{1+\Delta}(N-1)}{2}\bigg[\log\log \frac{e^{1+\Delta}(N-1)}{2}+\frac{1}{4}\bigg]\,\Delta   \label{J_2 2.2.2 eq3 (1) final y}  \nonumber\\
 &:=& J_{2.2.2}^{y,(1),\ast}(N)\,. \label{J_2 2.2.2 eq3 (1) final y}
\end{eqnarray}
%%%%%%%%%%%%%%%%%%%%%%%%%%%%%%%%%%%%%%%%%%%%%%%%%%%%%%%%%%%%%%%%%%
%%%%%%%%%%%%%%%%%%%%%%%%%%%%%%%%%%%%%%%%%%%%%%%%%%%%%%%%%%%%%%%%%%
%%%%%%%%%%%%%%%%%%%%%%%%%%%%%%%%%%%%%%%%%%%%%%%%%%%%%%%%%%%%%%%%%%
\textbf{(II.2) Investigation of $J_{2.2.2}^{y,(2)}(N,x)$}\\[.3cm]
%%%%%%%%%%%%%%%%%%%%%%%%%%%%%%%%%%%%%%%%%%%%%%%%%%%%%%%%%%%%%%%%%%
%%%%%%%%%%%%%%%%%%%%%%%%%%%%%%%%%%%%%%%%%%%%%%%%%%%%%%%%%%%%%%%%%%
%%%%%%%%%%%%%%%%%%%%%%%%%%%%%%%%%%%%%%%%%%%%%%%%%%%%%%%%%%%%%%%%%%
%%Analogously to the Part (II.2) in Section \ref{subparagraph J_2 2.2}, 
Replacing $ w $, $ P_{2,x}(w) $, $ F_{2,x}(w) $ and $ {\mathbb{M}}^{}(N,x)$ with $ \tilde{u}/2 $, $ P_{3,x}(\tilde{u}) $, $ F_{3,x}(\tilde{u}) $ and $ \widetilde{\mathbb{M}}^{}(N,x)$ in Part (II.2) in Section \ref{subparagraph J_2 2.2}, we obtain that.
\begin{eqnarray}\label{J_2 2.2.2 (2) eq4 y}
 J_{2.2.2}^{y,(2)}(N,x)< \bigg(\widetilde{\Delta}+4\log\log \dfrac{(N-1)}{2}\bigg)\,\widetilde{\Psi}(N,x)\,,
%% J_{2.2.2}^{y,(2)}(N,x)< \bigg\{K(\Delta)+4\log\bigg(\log \dfrac{N-1}{2}\bigg)\,\bigg\}J_{2.2.2}^{y,(2.1)}
\end{eqnarray}
\begin{align} \label{M_J 2.2.2 (2) P1 eq2 y}
\hspace{-1cm}\text{with}\quad \widetilde{\Psi}(N,x)=& \sum\limits_{\alpha=1}^{4}\Bigg\{\theta_{\alpha}\bigg[\log^{(4-\alpha)}\frac{(N-1)}{2}\bigg]\cdot \nonumber\\
& \int_{(e^{1+\Delta},\infty)}\dfrac{\log^{(\alpha-1)} \tilde{u}}{\tilde{u}}\big[\log\log \tilde{u}+\tfrac{1}{4}\big] \big[1-\mathsf{P}_{3,x}(\tilde{u})\big]d\tilde{u}\Bigg\}
\end{align}
with $ \theta_{1}=\theta_{4}=1 $ and $ \theta_{2}=\theta_{3}=3 $.
Recall from ($ \mathrm{iv} $) %in Eq.(\ref{F y N=3,x tilde u}), 
that we have %$1-F^{y}_{N,x}(u)=1-F^{y}_{3,x}(\tilde{u})  $ and
\begin{eqnarray}\label{1-F y N=3,x tilde u}
 1- F^{y}_{3,x}(\tilde{u})&=& \mathds{1}\Big[\rho(x, y)>r_{3}(\tilde{u})\Big] \Big[1-\mathsf{P}_{3,x}(\tilde{u})\Big] \nonumber\\ 
 &=& \begin{cases}
     1-\mathsf{P}_{3,x}(\tilde{u}), & \qquad \rho(x, y)>r_{3}(\tilde{u})\,, \\[2\jot]
    0, &\qquad  \rho(y,x)\leq r_{3}(\tilde{u})\,.
  \end{cases} 
\end{eqnarray}
Therefore, when $ \rho(y,x)\leq r_{3}(\tilde{u}) $, we have $ F^{y}_{3,x}(\tilde{u})=F^{y}_{N,x}(u)=1 $, which indicates $J^{y}_{2}(N,x) $ in Eq.(\ref{J_2 y}) is equal to zero, i.e.
\begin{eqnarray}
J^{y}_{2}(N,x)&=&\int_{(\sqrt{N-1},\infty)}[1-F^{y}_{N,x}(u)] \dfrac{(\log u)^{3}}{u}\bigg[\log(\log u)+\frac{1}{4}\bigg]du  %\nonumber\\
%&=&
=0\,.
\end{eqnarray}
Obviously, %%$J^{y}_{2}(N,x)<\infty  $ holds in this case, and 
the boundedness of $J^{y}_{2}(N,x)$ in Section \ref{proof of J2y N,x} %Proof of $J_{2}^{y}(N,x)<\infty  $ 
is proved in this case. Let us focus on the the case $ \rho(x, y)>r_{3}(\tilde{u}) $, and then we have 
\begin{eqnarray}\label{F_3=P_3}
F^{y}_{3,x}(\tilde{u})=\mathsf{P}_{3,x}(\tilde{u})\,.
\end{eqnarray}
%and substituting it into the RHS of Eq.(\ref{M_J 2.2.2 (2) P1 y}), we obtain the 1st part $ \mathrm{I}_{ J_{2.2.2}^{y,(2.1)}} $ satisfies
Substituting Eq.(\ref{F_3=P_3})  into (\ref{M_J 2.2.2 (2) P1 eq2 y}) and by \textit{Lemma} \ref{lemma power 3}, \textit{Lemma} \ref{lemma power 4},  we obtain %the 1st, 2nd, 3rd and 4th part of RHS of Eq.(\ref{M_J 2.2.2 (2) P1 eq2}) satisfy: 
\begin{align}\label{EGlog power y}
\int_{(e^{1+\Delta},\infty)}\dfrac{\log^{(\alpha-1)}\tilde{u} }{\tilde{u}}\big[\log\log \tilde{u}+\tfrac{1}{4}\big] \big[1-F^{y}_{3,x}(\tilde{u})\big]d\tilde{u}\leq \frac{R^{(\alpha)}_{y}(x)}{\alpha^{2}},\,\,\,\alpha=1,2,3,4.
\end{align}
%%\begin{align}
%%&(1)\quad \int\limits_{(e^{1+\Delta},\infty)}\dfrac{1}{w}\bigg[\log(\log w)+\frac{1}{4}\bigg] \big[1-F_{2,x}(w)\big]dw
%%\leq R^{(1)}(x)\label{EGlog power1} \\
%%&(2)\quad \int\limits_{(e^{1+\Delta},\infty)}\dfrac{\log^{} w}{w}\bigg[\log(\log w)+\frac{1}{4}\bigg] \big[1-F_{2,x}(w)\big]dw
%%\leq\frac{1}{4}R^{(2)}(x)\label{EGlog power2} \\
%%&(3)\quad \int\limits_{(e^{1+\Delta},\infty)}\dfrac{\log^{2}w}{w}\bigg[\log(\log w)+\frac{1}{4}\bigg] \big[1-F_{2,x}(w)\big]dw
%%\leq\frac{1}{9}R^{(3)}(x) \label{EGlog power3}\\
%%&(4)\quad \int\limits_{(e^{1+\Delta},\infty)}\dfrac{\log^{3}w}{w}\bigg[\log(\log w)+\frac{1}{4}\bigg] \big[1-F_{2,x}(w)\big]dw \leq \frac{1}{16}R^{(4)}(x)\label{EGlog power4} 
%%\end{align}
with
\begin{align}
R^{(\alpha)}_{y}(x):=\mathsf{E} G(\mid\log \tilde{u}\mid^{\alpha})=\int_{\mathbb{R}^{d}}\mid\log \tilde{u}\mid^{\alpha}\log\mid\log \tilde{u}\mid^{\alpha}dF^{y}_{3,x}(\tilde{u})\,. %,\quad \alpha=1,2,3,4.
\end{align}
Note that $ u $ represents $ \eta^{y}_{N,x} $ in this section. %for simplifying the notation of each integration. $ \eta^{y}_{N,x} $ is defined in  Eq.(\ref{eta y,i,j N,x 1}), i.e.
%%\begin{eqnarray}
%%u\sim\eta^{y}_{N,x}=\xi_{N,x}\big|_{X_{j}=y}=(N-1)V_{d}\tilde{\gamma}\min\Big\{\Big[\min\limits_{k\in\{1,\ldots,N\}\setminus\{i,j\}}\rho^{d}(X_{k},x)\Big],\rho^{d}(y,x)\Big\}
%%\end{eqnarray}
%%where $ N\geq 3 $ and $ \rho(x,X_{j}) $ represents the \textit{Euclidean distance} between $ X_{j}$ and $ x $. 
Recall that $ \tilde{u} = 2u/(N-1)$ and we have
%%Based on the definition of $ \tilde{u} = 2u/(N-1)$, we have
\begin{eqnarray} 
\tilde{u}=\dfrac{2 \eta^{y}_{N,x} }{N-1} &=& 2V_{d}\tilde{\gamma}\min\Big\{\Big[\min\limits_{k\in\{1,\ldots,N\}\setminus\{i,j\}}\rho^{d}(X_{k},x)\Big],\rho^{d}(y,x)\Big\} \,,\quad N\geq 3\nonumber\\
&\leq & 2V_{d}\tilde{\gamma}\min\Big\{\rho^{d}(X_{3},x),\rho^{d}(y,x)\Big\} \,,\quad N= 3\,,X_{2}=y \nonumber\\
%&\leq & 2V_{d}\tilde{\gamma}\min\Big\{\rho^{d}(X_{3},x),\rho^{d}(X_{2},x)\Big\} \,,\quad N= 3 \label{w<xi_3,x y 1}\\
%&=&\xi_{3,x} \label{w<xi_3,x y 2}
&=&\eta^{y}_{3,x} \,.\label{w<xi_3,x y 2}
\end{eqnarray}
For $ \tilde{u}\in (e,\infty) $, one has $ |\log \tilde{u}|\leq  |\xi_{3,x}| $ and hence, $ \mathsf{E}G|\log \tilde{u}|^{\alpha}\leq \mathsf{E} G|\eta^{y}_{3,x}|^{\alpha}$, $\alpha\geq 0  $.
%%\begin{eqnarray}
%% \mathsf{E}G|\log \tilde{u}|^{\alpha}\leq \mathsf{E} G|\eta^{y}_{3,x}|^{\alpha}\,,\quad \forall\alpha\geq 0
%%\end{eqnarray}
%%Set $\widetilde{\mathcal{A}}_{f,\alpha}:=\{x\in S(f)\}: \mathsf{E}G|\eta^{y}_{3,x}|^{\alpha}<\infty\,,\,\,\alpha=1,2,3,4.\}$. %% where $\mathcal{S} (f) $ represents the support set, i.e.  $ S(f) ：=\{x\in\mathbb{R}^{d}: f(x)>0\}$, and we conclude
\begin{lemma}\label{G_a+b v2}
For $ x\in\mathbb{R}^{d} $ and $ \alpha=1,2,3,4 $, there exists constants $ a,b\geq 0 $, such that:
\begin{align}
G(\mid\log^{\alpha}\eta^{y}_{3,x}\mid)\leq a\,G(\mid\log^{\alpha}{\rho}(x,y)\mid)+b\,.
\end{align}
\end{lemma}
\begin{proof}[Proof of Lemma \ref{G_a+b v2}] 
The proof is analogous to the proof of \textit{Lemma} \ref{G_a+b}, and thus, it is skipped here.
\end{proof}
The \textit{Lemma} \ref{G_a+b v2} implies that $  \mathsf{E}G(\mid\log^{\alpha}\eta^{y}_{3,x}(x,y)\mid)\leq a\,\mathsf{E}G(\mid\log^{\alpha}{\rho}(x,y)\mid)+b$. %, or
%\begin{align}\label{G_a+b exp}
%\int_{\mathbb{R}^{d}}G(\mid\log^{\alpha}\widetilde{\rho}(x,y)\mid)f(y)dy\leq a\int_{\mathbb{R}^{d}}G(\mid\log^{\alpha}{\rho}(x,y)\mid)f(y)dy+b
%\end{align}
Note %%$K_{f,{\color{black} \alpha}}(\varepsilon_{0})<\infty $ is defined in 
from Eq.(\ref{condition K}) that %and from Eq.(\ref{G_a+b exp}), we can conclude 
$ K_{f,{\color{black} \alpha}}(\varepsilon_{0}) <\infty\,\Longrightarrow \mathsf{E}G(\mid\log^{\alpha}\rho(x,y)\mid)<\infty $.
%%\begin{align}\label{G_a+b bounded}
%%K_{f,{\color{black} \alpha}}(\varepsilon_{0}) <\infty\,\Longrightarrow \mathsf{E}G(\mid\log^{\alpha}\widetilde{\rho}(x,y)\mid)<\infty
%%\end{align}
Set $\widetilde{\mathcal{A}}_{f,\alpha}:=\{x\in\mathcal{S}(f)\}: \mathsf{E}G|\eta^{y}_{3,x}|^{\alpha}<\infty\,,\,\,\alpha=1,2,3,4.\}$. By \textit{Lemma} \ref{G_a+b v2} %%and Eq.(\ref{G_a+b bounded})
, one has $ \mu\big(\mathcal{S}(f)\setminus\widetilde{\mathcal{A}}_{f,\alpha}\big)=0 $. Let us define $\widetilde{\mathcal{A}} :=\widetilde{\mathcal{A}}_{f,\alpha}(G)\cap \Lambda (f)\cap \mathcal{S}(f)\cap\mathcal{D}_{f}(R) $, where $\mu\big(\mathcal{S}(f)\setminus\widetilde{\mathcal{A}}\big)=0  $, and for $ x,y\in\widetilde{\mathcal{A}}$, we have
\begin{eqnarray}\label{R_alpha < infty y}
R^{(\alpha)}_{y}(x)<\infty\,,\qquad\alpha=1,2,3,4.\,\,%\text{and}\, x\in  \mathcal{A}_{}
\end{eqnarray}
Substituting Eqs.(\ref{EGlog power y})  % and (\ref{EGlog power3}),(\ref{EGlog power2}) and (\ref{EGlog power1}) 
and (\ref{R_alpha < infty y}) % (\ref{M_J 2.2.2 (2) P1 eq2 y}) 
into (\ref{M_J 2.2.2 (2) P1 eq2 y}), we obtain   
%(\ref{J_2 2.2.2 (2) eq4 y}),  
\begin{align}
\widetilde{\Psi}(N,x)\leq  \widetilde{R}_{y}(N,x) \label{M_J 2.2.2 (2) P1 eq3 y}
\end{align}
%%with
\begin{align}
\text{with}\quad \widetilde{R}_{y}(N,x)=&\frac{1}{16}R^{(4)}_{y}(x)+\frac{R^{(3)}_{y}(x)}{3}\log\dfrac{(N-1)}{2}+\frac{3}{4}R^{(2)}(x)\log^{2}\frac{(N-1)}{2}   \nonumber\\
 &+R^{(1)}(x) \log^{3}\frac{(N-1)}{2} \,.\nonumber
\end{align}
%%Substituting Eqs.(\ref{EGlog power4 y}), (\ref{EGlog power3 y}), (\ref{EGlog power2 y}) and (\ref{EGlog power1 y}) into Eqs.(\ref{M_J 2.2.2 (2) P1 eq2 y}) and (\ref{J_2 2.2.2 (2) eq4 y}),  we obtain
%%\begin{eqnarray}\label{M_J 2.2.2 (2) P1 eq3 y}
%%&&\hspace{-.6cm} J_{2.2.2}^{y,(2.1)}=\mathrm{I}_{ J_{2.2.2}^{y,(2.1)}} +3\log \bigg(\frac{N-1}{2}\bigg)\,\mathrm{II}_{ J_{2.2.2}^{y,(2.1)}}+3\bigg[\log\bigg(\frac{N-1}{2}\bigg)\bigg]^{2}\,\mathrm{III}_{ J_{2.2.2}^{y,(2.1)}}+\mathrm{IV}_{ J_{2.2.2}^{y,(2.1)}} \nonumber\\
%%&&\hspace{-.6cm}\leq \frac{1}{16}\bar{G}_{x}^{y,(4)}+\frac{1}{3}\log \bigg(\frac{N-1}{2}\bigg)\,\bar{G}_{x}^{y,(3)}+\frac{3}{4}\bigg[\log\bigg(\frac{N-1}{2}\bigg)\bigg]^{2}\bar{G}_{x}^{y,(2)}+\bigg[\log\bigg(\frac{N-1}{2}\bigg)\bigg]^{3}\bar{G}_{x}^{y,(1)}  \nonumber \\
%%&&\hspace{-0.6cm}:=\mathbb{G}^{y,\ast}(N,x)
%%\end{eqnarray}
Recall from Eq.(\ref{R_alpha < infty y}) that $R^{(\alpha)}_{y}(x)<\infty$, and this implies:
\begin{eqnarray}\label{P_1R y}
\widetilde{R}_{y}(N,x) =  \mathcal{O}\big([\log(N-1)]^{3}\big)\,. %%\,,\quad N\rightarrow\infty
\end{eqnarray}
Substituting Eqs.(\ref{M_J 2.2.2 (2) P1 eq3 y}) and (\ref{P_1R y}) into (\ref{J_2 2.2.2 (2) eq4 y}), we obtain
\begin{eqnarray}\label{J_2 2.2.2 (2) eq5 final y}
 J_{2.2.2}^{y,(2)}(N,x)&<& \bigg(\widetilde{\Delta}+4\log\log \dfrac{(N-1)}{2}\bigg)\,\widetilde{R}_{y}(N,x)
 %%\bigg[K(\Delta)+4\log\bigg(\log \dfrac{N-1}{2}\bigg)\,\bigg]\mathbb{G}_{x}^{y,\ast}(N) 
 \nonumber\\
%% &:=& J_{2.2.2}^{y,(2),\ast} (N,x)\\
 & = & \mathcal{O}\big\{\log^{3}(N-1)\big[\log\log(N-1)\big]\big\} := J_{2.2.2}^{y,(2),\ast} (N,x)\,.%%\,,\quad N\rightarrow\infty
\end{eqnarray}
We substitute $ J_{2.2.2}^{y,(1)}(N,x) $ in (\ref{J_2 2.2.2 eq3 (1) final y}) and $ J_{2.2.2}^{y,(2)}(N,x) $ in (\ref{J_2 2.2.2 (2) eq5 final y}) into $J_{2.2.2}^{y}(N,x)$ in (\ref{J_2 2.2.2 eq2 y}):
\begin{eqnarray}\label{J_2 2.2.2 final y}
J_{2.2.2}^{y}(N,x)&=&J_{2.2.2}^{(1)}(N,x)+J_{2.2.2}^{(2)}(N,x)\nonumber\\
&< & %% J_{2.2.2}^{y,(1),\ast} + J_{2.2.2}^{y,(2),\ast} \nonumber \\
%%&=&
\log^{3} \frac{e^{1+\Delta}(N-1)}{2}\bigg[\log\log \frac{e^{1+\Delta}(N-1)}{2}+\frac{1}{4}\bigg]\,\Delta \nonumber\\
%%(\log (N-1)e^{1+\Delta})^{3}\Big[\log(\log (N-1)e^{1+\Delta})+\frac{1}{4}\Big]\,\Delta \nonumber\\
&&\hspace{.cm}+\bigg(\widetilde{\Delta}+4\log\log \dfrac{(N-1)}{2}\bigg)\,\widetilde{R}_{y}(N,x)
%%+\bigg[K(\Delta)+4\log\bigg(\log \dfrac{N-1}{2}\bigg)\,\bigg]\mathbb{G}_{x}^{y,\ast}(N) 
\nonumber\\
%%&:=& J_{2.2.2}^{y,\ast}(N,x) \\
&\hspace{.cm}=&\mathcal{O}\big\{\log^{3}(N-1)\big[\log\log(N-1)\big]\big\}:=J_{2.2.2}^{y,\ast}(N,x)\,. %%\,,\quad N\rightarrow\infty\,.
\end{eqnarray}
We substitute $J_{2.2.1}^{y}(N,x) $ in (\ref{J_2 2.2.1 eq3 final y}) and $J_{2.2.2}^{y}(N,x)  $ in (\ref{J_2 2.2.2 final y}) into $  J_{2.2}^{y}(N,x)$ in (\ref{J_2 2.2 1 y}):
%%we obtain: 
%\bigg[\log \bigg(\frac{e}{2}(N-1)\bigg)\bigg]^{3}\bigg[\log\bigg[\log \bigg(\frac{e}{2}(N-1)\bigg)\bigg]+\frac{1}{4}\bigg] \,\bigg[\log \bigg(\dfrac{e\sqrt{N-1}}{2}\bigg)\bigg] \nonumber \\
%&:=& J_{2.2.1}^{y,\ast}(N) \label{J_2 2.2.1 eq3 final y} 
\begin{eqnarray}
J_{2.2}^{y}(N,x)&=&J_{2.2.1}^{y}(N,x)+J_{2.2.2}^{y}(N,x) \nonumber\\
&\leq & %% J_{2.2.1}^{y,\ast}(N,x)+J_{2.2.2}^{y,\ast}(N,x) \nonumber \\
%%&=&
\log^{3} \frac{e(N-1)}{2}\bigg[\log\log \frac{e(N-1)}{2}+\frac{1}{4}\bigg] \,\log \dfrac{e\sqrt{N-1}}{2}\nonumber\\
&&+\log^{3} \frac{e^{1+\Delta}(N-1)}{2}\bigg[\log\log \frac{e^{1+\Delta}(N-1)}{2}+\frac{1}{4}\bigg]\,\Delta \nonumber\\
&&\hspace{.cm}+\bigg(\widetilde{\Delta}+4\log\log \dfrac{(N-1)}{2}\bigg)\,\widetilde{R}_{y}(N,x)   \nonumber \\
%%&=&J_{2.2}^{y,\ast}(N,x) \nonumber \\
&&\hspace{-.6cm}=\mathcal{O}\big\{\log^{4}(N-1)\big[\log\log(N-1)\big]\big\} :=J_{2.2}^{y,\ast}(N,x)\,.
%%\mathcal{O}([\log(N-1)]^{4})\,,\quad N\rightarrow\infty 
\label{J_2 2.2.2 eq2 final y}
\end{eqnarray}
Substituting Eqs.(\ref{J_2 2.1 1 y final}) and (\ref{J_2 2.2.2 eq2 final y}) into (\ref{J_2 2 y}), we obtain
\begin{eqnarray}
J_{2}^{y}(N,x)&\leq & J_{2.1}^{y}(N,x)\cdot J_{2.2}^{y}(N,x)
%%J_{2.1}^{y,\ast}(N,x)\, J_{2.2}^{y}(N,x) 
\nonumber\\
&=& \dfrac{(3e^{\gamma})^{\varepsilon}}{m_{f}^{\varepsilon}(x,R_{2})}\,\dfrac{\mathcal{O}\big\{\log^{4}(N-1)\big[\log\log(N-1)\big]\big\} }{(N-1)^{\varepsilon/2}}\,,\qquad\text{}\, \varepsilon\in(0,e]\nonumber\\
&:=&%\dfrac{\mathcal{O}\big\{\log^{4}(N-1)\big[\log\log(N-1)\big]\big\} }{\mathcal{O}\big((N-1)^{\varepsilon/2}\big)}:=
J_{2}^{y,\ast}(N,x) \,.
%%&< & \infty 
\label{J_2 y infty}
\end{eqnarray}
%%where $ \frac{\mathcal{O}([\log(N-1)]^{4})}{(N-1)^{\varepsilon/2}} $ decreases as $ N $ increases, and  $ \frac{\mathcal{O}([\log(N-1)]^{4})}{(N-1)^{\varepsilon/2}} \xrightarrow[\text{}]{\text{$N\rightarrow \infty$}} 0 $, since the power function grows much faster than the logarithm function when $ N $ increases. Therefore, $ J_{2}^{y}(N,x) $ is bounded. \\
Note that logarithm functions grow slower than power functions, and this implies that
\begin{eqnarray}\label{R_j_2 y}
J_{2}^{y,\ast}(N,x)=\dfrac{\mathcal{O}\big([\log^{4}(N-1)]\log\log (N-1)\big)}{\mathcal{O}((N-1)^{\varepsilon/2})}\longrightarrow 0\,,\quad N\rightarrow\infty\,.  %%\,,\,\, \varepsilon\in(0,e]\,.
\end{eqnarray}
Thus, we have
\begin{eqnarray}\label{J_2 2 final+1 y}
J_{2}^{y}(N,x)\leq J_{2}^{y,\ast}(N,x) <\infty\,.
\end{eqnarray}
%%Substituting Eq.(\ref{J_2 2.2.2 eq2 final y}) and $ J_{2.1}(N,x) $ into Eq.(\ref{J_2 2.1 6 y}) into Eq.(\ref{J_2 2 y}), we obtain
From Eqs.(\ref{I1 x X,y}), (\ref{I1 y N,x bounded}), (\ref{I_2(N,x) 1}), (\ref{J_1 y part2 eq3}) and (\ref{J_2 2 y}), we have
\begin{eqnarray}
\hspace{-.3cm}\frac{\mathsf{E}G\big(\log^{4}\eta_{N,x}^{y}\big)}{4}\, &=&I_{1}^{y}(N,x)+I_{2}^{y}(N,x)  \label{EG log4 y eq1}\nonumber\\
\hspace{-1.cm}&:=&\vartheta_{1}(N)\,M_{f}^{\varepsilon_{1}}(x,R_{1})+\vartheta_{2}(N)\,m_{f}^{-\varepsilon}(x,R_{2})+J_{2}^{y,\ast}(N,x) \,,\label{EG log4 y eq2}
%_{J_{2.2}^{y}(N,x)\,,\,\, \text{Eq.}(\ref{J_2 2.2 3 y})}
\end{eqnarray}
where
\begin{eqnarray}
&&\vartheta_{1}(N):=e^{-\gamma\varepsilon_{1}}
\textstyle L(\varepsilon_{1})%%{\underbrace{\textstyle L(\varepsilon_{1})}_{\mathclap{<\infty ,\,\text{Eq.}(\ref{Lv1 cov})}}}
<\infty\,,\quad \varepsilon_{1}\in(0,1]\,, \label{vartheta 1 N y eq1}\\
&&\vartheta_{2}(N):=(2e^{\gamma})^{\varepsilon}%%{\underbrace{\textstyle L(\varepsilon)}_{\mathclap{<\infty ,\,\text{Eq.}(\ref{Lv1 cov})}}}
\textstyle L(\varepsilon)<\infty\,,\quad \varepsilon \in(0,e]\,.
\end{eqnarray} \label{vartheta 2 N y eq1}
and
\begin{eqnarray}\label{J_2 y ast}
&&\hspace{-.9cm}  J_{2}^{y,\ast}(N,x)= J_{2.1}^{y,\ast}(N,x)\cdot J_{2.2}^{y,\ast}(N,x) \nonumber\\
&&\hspace{-.6cm} = {(3e^{\gamma})^{\varepsilon}}{m_{f}^{-\varepsilon}(x,R_{2})}\,(N-1)^{-\varepsilon/2}\cdot \Big[J_{2.2.1}^{y,\ast}+\Big(J_{2.2.2}^{y,(1),\ast}+ J_{2.2.2}^{y,(2),\ast}\Big)\Big]%%\,,\quad \varepsilon\in(0,e] 
\nonumber\\
&&\hspace{-.6cm} ={m_{f}^{-\varepsilon}(x,R_{2})}\,{(3e^{\gamma})^{\varepsilon}}\, \Bigg[\dfrac{J_{2.2.1}^{y,\ast}+J_{2.2.2}^{y,(1),\ast}}{(N-1)^{\varepsilon/2}}+\dfrac{ J_{2.2.2}^{y,(2),\ast}}{(N-1)^{\varepsilon/2}}\Bigg] \nonumber\\
&&\hspace{-.6cm}:=m_{f}^{-\varepsilon}(x,R_{2})\,\Big(\vartheta_{3}(N)+\vartheta_{4}(N)\,\widetilde{R}_{y}(N,x)  \Big) \,,
\end{eqnarray}
where
\begin{eqnarray}
&&\hspace{-1.2cm}\vartheta_{3}(N)=(3e^{\gamma})^{\varepsilon}\cdot\dfrac{J_{2.2.1}^{y,\ast}+J_{2.2.2}^{y,(1),\ast}}{(N-1)^{\varepsilon/2}} \,,\quad \varepsilon\in(0,e] \nonumber\\
&&=(3\tilde{\gamma})^{\varepsilon}\cdot \dfrac{\mathcal{O}\big\{\log^{4}(N-1)\big[\log\log(N-1)\big]\big\}}{(N-1)^{\varepsilon/2}} \,,\label{vartheta 3 N y eq1} \\
&& \xrightarrow[\text{}]{\text{$N\rightarrow \infty$}} 0 \,. \label{vartheta 3 N y eq2}
\end{eqnarray}
and 
\begin{eqnarray}
&&\hspace{-1.2cm}\vartheta_{4}(N)=(3e^{\gamma})^{\varepsilon}\dfrac{ J_{2.2.2}^{y,(2),\ast}/\widetilde{R}_{y}(N,x)  }{(N-1)^{\varepsilon/2}}\,,\quad \varepsilon\in(0,e] \nonumber\\
&&=(3e^{\gamma})^{\varepsilon}\dfrac{\mathcal{O}\big(\log\log(N-1)\big)}{(N-1)^{\varepsilon/2}}\,,  \label{vartheta 4 N y eq1}\\
&& \xrightarrow[\text{}]{\text{$N\rightarrow \infty$}} 0 \,. \label{vartheta 4 N y eq1}
\end{eqnarray}
Indeed, both $\vartheta_{3}(N) \xrightarrow[\text{}]{\text{$N\rightarrow \infty$}} 0 $ and $\vartheta_{4}(N) \xrightarrow[\text{}]{\text{$N\rightarrow \infty$}} 0 $ hold since the denominator of Eqs.(\ref{vartheta 3 N y eq1}) and (\ref{vartheta 4 N y eq1})  grows much faster than the numerator (i.e. the power function grows much faster than the logarithm functions) when $ N $ increases.\\
Substituting Eqs.(\ref{vartheta 1 N y eq1})-(\ref{vartheta 4 N y eq1}) into Eq.(\ref{EG log4 y eq2}), we obtain
\begin{eqnarray}
\frac{\mathsf{E}\Big\{G\Big[(\log\eta_{N,x}^{y}\big)^{4}\Big]\Big\}}{4} \hspace{-.6cm} &&\leq \vartheta_{1}(N)\,M_{f}^{\varepsilon_{1}}(x,R_{1})+ \Big(\vartheta_{2}(N)+\vartheta_{3}(N)+\vartheta_{4}(N)\,\widetilde{R}_{y}(N,x) \Big)\,m_{f}^{-\varepsilon}(x,R_{2})\nonumber \\
&& < \infty\,. \label{EG log4 y eq3}
\end{eqnarray}
Thanks to the symmetry structure, we can obtain the following result analogously.
\begin{eqnarray}
\frac{\mathsf{E}\Big\{G\Big[(\log\eta_{N,y}^{x}\big)^{4}\Big]\Big\}}{4} \hspace{-.6cm} &&\leq \vartheta_{1}(N)\,M_{f}^{\varepsilon_{1}}(y,R_{1})+ \Big(\vartheta_{2}(N)+\vartheta_{3}(N)+\vartheta_{4}(N)\,\widetilde{R}_{x}(N,y) \Big)\,m_{f}^{-\varepsilon}(y,R_{2})\nonumber \\
&& < \infty \,.\label{EG log4 x eq3}
\end{eqnarray}
Substituting Eqs.(\ref{EG log4 y eq3}) and (\ref{EG log4 x eq3}) into Eq.(\ref{E eta_x^y,2 eta_y^x,2}), we obtain
\begin{eqnarray}\label{E eta_x^y,2 eta_y^x,2 eq2}
&&\mathsf{E} G \Big[\Big| \big(\log\eta^{y}_{N,x}\big)^{2}\big(\log\eta^{x}_{N,y}  \big)^{2}\Big|\Big]  \nonumber\\
 &&\hspace{-.6cm} \leq \dfrac{\mathsf{E} G \big(\log\eta^{y}_{N,x}\big)^{4}+\mathsf{E} G \big(\log\eta^{x}_{N,y}  \big)^{4}}{2}  \nonumber\\
 &&\hspace{-.6cm}\leq 2\,\vartheta_{1}(N)\,\Big(M_{f}^{\varepsilon_{1}}(x,R_{1})+M_{f}^{\varepsilon_{1}}(y,R_{1})\Big)+ 2\vartheta_{2}(N)\Big(m_{f}^{-\varepsilon}(x,R_{2})+m_{f}^{-\varepsilon}(y,R_{2})\Big)+ \nonumber \\
 && 2\vartheta_{3}(N)\Big(m_{f}^{-\varepsilon}(x,R_{2})+m_{f}^{-\varepsilon}(y,R_{2})\Big)+2\vartheta_{4}(N)\Big(m_{f}^{-\varepsilon}(x,R_{2})\widetilde{R}_{y}(N,x) +m_{f}^{-\varepsilon}(y,R_{2})\widetilde{R}_{x}(N,y) \Big)
 \nonumber\\
 &:=& C_{0}(x,y) + C_{1}(x,y)  \nonumber\\
 &<& \infty\,,
\end{eqnarray}
where 
\begin{eqnarray}
&&\hspace{-1.cm} C_{0}(x,y)=2\,\vartheta_{1}(N)\,\Big(M_{f}^{\varepsilon_{1}}(y,R_{1})+M_{f}^{\varepsilon_{1}}(y,R_{1})\Big)+ 2\vartheta_{2}(N)\Big(m_{f}^{-\varepsilon}(x,R_{2})+m_{f}^{-\varepsilon}(y,R_{2})\Big) \nonumber\\
&&\hspace{.6cm} <\infty
\end{eqnarray}
and
\begin{eqnarray}
&& C_{1}(x,y)=2\vartheta_{3}(N)\Big(m_{f}^{-\varepsilon}(x,R_{2})+m_{f}^{-\varepsilon}(y,R_{2})\Big)+ 2\vartheta_{4}(N)\cdot \nonumber\\
&&\hspace{2.cm} \Big( m_{f}^{-\varepsilon}(x,R_{2})\widetilde{R}_{y}(N,x) +m_{f}^{-\varepsilon}(y,R_{2})\widetilde{R}_{x}(N,y) \Big) \nonumber \\
&&\hspace{1.6cm}\xrightarrow[\text{}]{\text{$N\rightarrow \infty$}} 0\,. 
\end{eqnarray}
$ C_{1}(x,y) >0$ is bounded and decreases when $ N $ increases. Hence , $ \exists $ a positive numbers $ \kappa \in(0,\infty)$ and $ \widetilde{N}_{3}:=\widetilde{N}_{3}(x,y)\in (0,\infty)$ that
\begin{eqnarray} \label{C_1 y bounded}
C_{1}(x,y)\leq \kappa\,, \qquad \text{when}\,\, N\geq \widetilde{N}_{3}\,.
\end{eqnarray}
In conclusion, we define $ \widetilde{C}_{0}=C_{0}(x,y) + \kappa  $ and there exists a positive  $ \widetilde{N} $, that
\begin{eqnarray}\label{uniform integrability joint Pt1}
\sup_{N\geq \tilde{N}} \mathsf{E} G \Big[\Big| \big(\log\eta^{y}_{N,x}\big)^{2}\big(\log\eta^{x}_{N,y}  \big)^{2}\Big|\Big]\leq \tilde{C}_{0}(x,y)<\infty
\end{eqnarray}
with
\begin{eqnarray}
\widetilde{N}:=\widetilde{N}(x,y)=\max\{\widetilde{N}_{0}, \widetilde{N}_{1}, \widetilde{N}_{2},\tilde{N}_{3}\}\,.
\end{eqnarray}
%%where $ \tilde{N}_{0} $ defined in Eqs(\ref{N_0 tilde}), $ \tilde{\tilde{N}}_{1} $ in Eq.(\ref{N1 M R1}) , $ N_2 $ in Eq.(\ref{N_2 eq2}), and $ \tilde{ N}_{3} $ in Eq.(\ref{C_1 y bounded}).
We completed the proof of uniform of integrability of $ \Big\{\big(\log\eta^{y}_{N,x}\big)^{2}\big(\log\eta^{x}_{N,y}  \big)^{2}\Big\} $ for $ (x,y)\in \mathcal{A}_{1,M}:=\{(x,y)\in \mathcal{A}_{1}: \rho(x,y)>M\}$.
\begin{eqnarray}
\mathsf{E} \Big\{\big(\log\eta^{y}_{N,x}\big)^{2}\big(\log\eta^{x}_{N,y}  \big)^{2}\Big|\mathds{1}\big\{\rho(x, y)>M\big\}\Big\}\xrightarrow[\text{}]{\text{$N\rightarrow \infty$}} \big(\log^{2}f(x)+\tfrac{\pi^{2}}{6}\big)\big(\log^{2}f(y)+\tfrac{\pi^{2}}{6}\big).\quad 
\end{eqnarray}
%%%%%%%%%%%%%%%%%%%%%%%%%%%%%%%%%%%%%%%%%%%%%%%%%%%%%%%%%%%%%%%%%%%%%%%%%%%%%%%%%%%%%%%%%%%%%%%%%%%%%%%%%%%%%%%%%%%%%%%%%%%%%%%%%%%%
\subsubsection{T1: $ (x,y)\in \mathcal{A}_{1,M}:=\{(x,y)\in \mathcal{A}_{1}: \rho(x,y)>M,\,\,M>0\}$}
%%%%%%%%%%%%%%%%%%%%%%%%%%%%%%%%%%%%%%%%%%%%%%%%%%%%%%%%%%%%%%%%%%%%%%%%%%%%%%%%%%%%%%%%%%%%%%%%%%%%%%%%%%%%%%%%%%%%%%%%%%%%%%%%%%%%
Let us define
\begin{eqnarray} \label{T_N,x,y def}
T_{N}(x,y)&=&\mathsf{E}\Big({\zeta_{i}(N)}^{2}{\zeta_{j}(N)}^{2} \Big| X_{i}=x, X_{j}=y \Big)
\end{eqnarray}
and combing the Eqs.(\ref{expectation joint N}) and (\ref{expectation joint N limit}), we have
\begin{eqnarray} \label{zeta_i zeta_j}
\mathsf{E}\Big({\zeta_{i}(N)}^{2}{\zeta_{j}(N)}^{2} \Big| X_{i}=x, X_{j}=y \Big)&=&\mathsf{E}\Big({\zeta_{1}(N)}^{2}{\zeta_{2}(N)}^{2} \Big| X_{1}=x, X_{2}=y \Big) \nonumber\\
&=&\mathsf{E}\Big[\big(\log\eta^{y}_{N,x}\big)^{2}\big(\log\eta^{x}_{N,y}  \big)^{2}\Big] \nonumber\\
%&\xrightarrow[\text{}]{N \rightarrow \infty }\,&\mathsf{E}\Big[\big(\log\eta_{x}\big)^{2}\big(\log\eta_{y}  \big)^{2}\Big]\,,\quad\quad \nonumber\\
&=&\mathsf{E}\Big[\big(\log\eta_{x}\big)^{2}\big(\log\eta_{y}  \big)^{2}\Big]\,,\quad\quad N\rightarrow\infty \nonumber\\
&=&\mathsf{E}\Big[\big(\log\eta_{x}\big)^{2}\big(\log\eta_{y}  \big)^{2}\Big] \nonumber\\
&=&\big(\log^{2}f(x)+\tfrac{\pi^{2}}{6}\big)\big(\log^{2}f(y)+\tfrac{\pi^{2}}{6}\big)\,.
\end{eqnarray}
From Eqs.(\ref{zeta_i zeta_j}) and (\ref{T_N,x,y def}), we have
\begin{eqnarray}\label{T_N,x,y def expectation}
\mathsf{E}\,T_{N}(x,y)&= &\mathsf{E}\Big\{\mathsf{E}\Big({\zeta_{i}(N)}^{2}{\zeta_{j}(N)}^{2} \Big| X_{i}=x, X_{j}=y \Big) \Big\}  \nonumber\\
&= &\mathsf{E}\Big({\zeta_{i}(N)}^{2}{\zeta_{j}(N)}^{2} \Big)\,.
\end{eqnarray}
We have to prove the uniform of integrability of $  T_{N}(x,y)$ for $ (x,y)\in \mathcal{A}_{1,M}:=\{(x,y)\in \mathcal{A}_{1}: \rho(x,y)>M\}$. It is worth to derive the following result.
\begin{align} 
G\Big(\Big|T_{N}(x,y)\Big|\mathds{1}\big\{\rho(x, y)>M\big\}\Big) &\leq  G\Big(\Big|T_{N}(x,y)\Big|\Big) \nonumber\\
& = G\Big(\Big|\mathsf{E}\big({\zeta_{i}(N)}^{2}{\zeta_{j}(N)}^{2} \Big| X_{i}=x, X_{j}=y \big)\big|\Big) \nonumber\\
&= G\Big(\Big|\mathsf{E}\Big[\big(\log\eta^{y}_{N,x}\big)^{2}\big(\log\eta^{x}_{N,y}  \big)^{2}\Big]\Big|\Big) \nonumber\\
&\leq G\Big(\mathsf{E}\Big|\big(\log\eta^{y}_{N,x}\big)^{2}\big(\log\eta^{x}_{N,y}  \big)^{2}\Big|\Big) \nonumber\\
&\hspace{-1cm}\underbrace{\leq}_{\text{Jensen's inequality}} \mathsf{E}\Big[G\Big(\Big|\big(\log\eta^{y}_{N,x}\big)^{2}\big(\log\eta^{x}_{N,y}  \big)^{2}\Big|\Big)\Big] \label{GT_N,x,y eq1}  \\
&\leq  \tilde{C}_{0}(x,y)\,,\quad \text{defined in Eq.(\ref{uniform integrability joint Pt1})} \nonumber\\
&< \infty\,.
\end{align}
Let us calculate the expectation $ \mathsf{E}\,G\big(\Big|T_{N}(x,y)\Big|\mathds{1}\big\{\rho(x, y)>M\big\}\big) $ by substituting the Eqs.(\ref{T_N,x,y def}),(\ref{GT_N,x,y eq1}) and (\ref{E eta_x^y,2 eta_y^x,2 eq2}) as follows.
\begin{eqnarray}
&&\hspace{.0cm}\mathsf{E}\,G\Big(\Big|T_{N}(x,y)\Big|\mathds{1}\big\{\rho(x, y)>M\big\}\Big) \nonumber\\
&&\hspace{-0.6cm}\leq \mathsf{E}\,G\Big(\Big|T_{N}(x,y)\Big|\Big) \nonumber\\
&&\hspace{-.6cm}=\int_{\mathbb{R}^{d}}\int_{\mathbb{R}^{d}} \,G\Big(\Big|T_{N}(x,y)\Big|\Big) f(x)f(y) dx\,dy \nonumber\\
&&\hspace{-.6cm}\leq \int_{\mathbb{R}^{d}}\int_{\mathbb{R}^{d}} \mathsf{E}\Big[G\Big(\Big|\big(\log\eta^{y}_{N,x}\big)^{2}\big(\log\eta^{x}_{N,y}  \big)^{2}\Big|\Big)\Big] f(x)f(y) dx\,dy \nonumber\\
&&\hspace{-.6cm}\leq 2\,\vartheta_{1}(N)\,\Big(\int_{\mathbb{R}^{d}}M_{f}^{\varepsilon_{1}}(x,R_{1})f(x)dx+\int_{\mathbb{R}^{d}}M_{f}^{\varepsilon_{1}}(y,R_{1})f(y)dy\Big)+ 2\vartheta_{2}(N)\Big(\int_{\mathbb{R}^{d}}   \nonumber\\
&& m_{f}^{-\varepsilon}(x,R_{2})f(x)dx+\int_{\mathbb{R}^{d}}m_{f}^{-\varepsilon}(y,R_{2})f(y)dy\Big)+ 2\vartheta_{3}(N)\Big(\int_{\mathbb{R}^{d}}m_{f}^{-\varepsilon}(x,R_{2})f(x)dx    \nonumber \\
&& +\int_{\mathbb{R}^{d}}m_{f}^{-\varepsilon}(y,R_{2})f(y)dy\Big)+2\vartheta_{4}(N)\Big(\int_{\mathbb{R}^{d}}m_{f}^{-\varepsilon}(x,R_{2})\mathbb{G}^{y,\ast}(N,x)f(x)dx+      \nonumber\\
&&\int_{\mathbb{R}^{d}}m_{f}^{-\varepsilon}(y,R_{2})\mathbb{G}^{x,\ast}(N,y)f(y)dy\Big)    \nonumber\\
&&\hspace{-.6cm}=4\,\vartheta_{1}(N)\,\int_{\mathbb{R}^{d}}M_{f}^{\varepsilon_{1}}(x,R_{1})f(x)dx+ 4\vartheta_{2}(N)\int_{\mathbb{R}^{d}} m_{f}^{-\varepsilon}(x,R_{2})f(x)dx+  \nonumber\\
&& 4\vartheta_{3}(N)\int_{\mathbb{R}^{d}}m_{f}^{-\varepsilon}(x,R_{2})f(x)dx +4\vartheta_{4}(N)\,\int_{\mathbb{R}^{d}}m_{f}^{-\varepsilon}(x,R_{2})\,\mathbb{G}^{y,\ast}(N,x)f(x)dx \nonumber\\
&&\hspace{-.6cm}:= \vartheta_{1}(N)\, Q_{f}(\varepsilon_{1}, R_{1})+\vartheta_{2}(N)\,T_{f}(\varepsilon , R_{2})+\kappa_{f}(N,\varepsilon ,R_{2}) \,,\label{ET_N,x,y rho>M eq1}
\end{eqnarray}
where
\begin{eqnarray}
&& Q_{f}(\varepsilon_{1}, R_{1})=4\,\int_{\mathbb{R}^{d}}M_{f}^{\varepsilon_{1}}(x,R_{1})f(x)dx <\infty \,,\label{Qf eq1} \\
&& T_{f}(\varepsilon , R_{2})=4 \int_{\mathbb{R}^{d}} m_{f}^{-\varepsilon}(x,R_{2})f(x)dx <\infty \,,\label{Tf eq1}
\end{eqnarray}
and
\begin{align}
\hspace{-.4cm}\kappa_{f}(N,\varepsilon ,R_{2})& =4\underbrace{\vartheta_{3}(N)}_{\xrightarrow[\text{}]{\text{$N\rightarrow \infty$}} 0 }\int_{\mathbb{R}^{d}}m_{f}^{-\varepsilon}(x,R_{2})f(x)dx +4\underbrace{\vartheta_{4}(N)}_{\xrightarrow[\text{}]{N\rightarrow \infty} 0 }\,\int_{\mathbb{R}^{d}}m_{f}^{-\varepsilon}(x,R_{2})\,\mathbb{G}^{y,\ast}(N,x)f(x)dx \nonumber\\
&\xrightarrow[\text{}]{\text{$N\rightarrow \infty$}} 0 \,.\label{kappa_f N}
\end{align}
$ \kappa_{f}(N,\varepsilon ,R_{2}) >0$ is bounded and decreases when $ N $ increases. Hence , $ \exists $ a positive numbers $ \tilde{\kappa} \in(0,\infty)$ and $ N_{4}=N_{4}(\tilde{\kappa})\in (0,\infty)$ that
\begin{eqnarray} \label{kappa_f N tilde}
\kappa_{f}(N,\varepsilon ,R_{2}) \leq \tilde{\kappa}\,, \qquad \text{when}\,\, N\geq {N}_{4}\,.
\end{eqnarray}
In conclusion, we define $ \tilde{\tilde{C}}_{0}=Q_{f}(\varepsilon_{1}, R_{1}) + T_{f}(\varepsilon , R_{2})+ \tilde{\kappa}  $ and there exists a positive  $ \tilde{\tilde{N}}=\max\{N_{4}(\tilde{\kappa}), \tilde{N}\} $, that
\begin{eqnarray}\label{uniform integrability joint Pt2 T}
\sup_{N\geq \tilde{\tilde{N}}} \mathsf{E} G\Big(\Big|T_{N}(x,y)\Big|\mathds{1}\big\{\rho(x, y)>M\big\}\Big)\leq \tilde{\tilde{C}}_{0}(x,y)<\infty\,.
\end{eqnarray}
We completed the proof of uniform of integrability of $  T_{N}(x,y)$ for $(x,y)\in \mathcal{A}_{1,M}:=\{(x,y)\in \mathcal{A}_{1}: \rho(x,y)>M\}$.
Thus,
\begin{eqnarray} \label{ET_N,x,y rho>M eq2}
\mathsf{E} \Big[\Big|T_{N}(x,y)\Big|\mathds{1}\big\{\rho(x, y)>M\big\}\Big]\xrightarrow[\text{}]{\text{$N\rightarrow \infty$}} \mathsf{E} \Big[\big(\log^{2}f(x)+\tfrac{\pi^{2}}{6}\big)\big(\log^{2}f(y)+\tfrac{\pi^{2}}{6}\big)\Big]\,.
\end{eqnarray}
%%%%%%%%%%%%%%%%%%%%%%%%%%%%%%%%%%%%%%%%%%%%%%%%%%%%%%%%%%%%%%%%%%%%%%%%%%%%%%%%%%%%%%%%%%%%%%%%%%%%%%%%%%%%%%%%%%%%%%%%%%%%%%%%%%%%
\subsubsection{T2: $(x,y)\in \mathcal{A}_{2,M}:=\{(x,y)\in \mathcal{A}_{2}: \rho(x,y)\leq M,\,\,M>0\}$}
%%%%%%%%%%%%%%%%%%%%%%%%%%%%%%%%%%%%%%%%%%%%%%%%%%%%%%%%%%%%%%%%%%%%%%%%%%%%%%%%%%%%%%%%%%%%%%%%%%%%%%%%%%%%%%%%%%%%%%%%%%%%%%%%%%
Let us considering when $  \rho(x,y)\leq M$ or the domain $(x,y)\in \mathcal{A}_{2,M}:=\{(x,y)\in \mathcal{A}_{2}: \rho(x,y)\leq M\}$ with $ M> 0 $. We are interested in the case when $ M $ is sufficiently small and without loss of generality, we set $ M\ll 1 $. Thus, this implies
$ -\infty <\log \rho(x,y) \leq \log M <0 $ and thus, $ |\log \rho(x,y)|\geq |\log M | $. Moreover, we have the following result in probabilities.
\begin{eqnarray}
\mathsf{Pr} \Big(\rho(x,y)\leq M \Big)&=&\mathsf{Pr}\Big\{G\big(\big|\log \rho(x,y)\big|\big)\geq G\big(\big|\log M \big|\big)\Big\}\,,\quad M\in(0,1) \nonumber\\
&\leq & \dfrac{\mathsf{E}G\big(\big|\log \rho(x,y)\big|\big)}{G\big(\big|\log M \big|\big)}\,,\quad \text{based on Markov's inequality } \nonumber\\
& \xrightarrow[\text{}]{\text{$M\rightarrow 0$}} & 0\,, \label{Prob rho<=M}
\end{eqnarray}
where the numerator is finite, i.e.
\begin{eqnarray}
\mathsf{E}G\big(\big|\log \rho(x,y)\big|\big)=\int_{\mathbb{R}^{d}}\int_{\mathbb{R}^{d}} \,G\Big(\Big|\log \rho(x,y)\Big|\Big) f(x)f(y) dx\,dy <\infty\,.
\end{eqnarray}
The ``$ <\infty $'' is guaranteed by $ K_{f,\alpha}(\varepsilon_{0})<\infty  $ in Eq.(\ref{condition K}), where we set $ \alpha =1 $ and $ \varepsilon_{0}=0 $. The denominator tends to infinity when $ M \rightarrow 0$, i.e.
\begin{eqnarray}
G\big(\big|\log M \big|\big) \xrightarrow[\text{}]{\text{$M\rightarrow 0$}} \infty\,.
\end{eqnarray}
For $ z\in [1,t] $ and $ t\geq 1 $, we have
\begin{eqnarray}
\dfrac{t-1}{t}=\int_{1}^{t}\dfrac{1}{t} dz \leq \int_{1}^{t}\dfrac{1}{z} dz = \log t \leq \int_{1}^{t} 1\,dz= t-1\,.
\end{eqnarray}
Therefore, we have
\begin{eqnarray}
t-1 \leq =t \log t  \,\Longrightarrow\, t\leq G(t)+1\,,\quad t\geq 1\,.
\end{eqnarray}
Obviously, when $ t\in [0,1) $, we have $ G(t)=0\,\,\Rightarrow\,\,t\leq G(t)+1 $. Therefore, we have for all $ t\geq 0 $
\begin{eqnarray}\label{t<G(t)+1 eq1}
t\leq G(t)+1\,,\qquad t\geq 0\,.
\end{eqnarray}
Multiplying $ \mathds{1}\big\{\rho(x, y)\leq M\big\}\geq 0 $ to both RHS and LHS of Eq.(\ref{t<G(t)+1 eq1}), we obtain
\begin{eqnarray}\label{t<G(t)+1 eq2}
t\,\mathds{1}\big\{\rho(x, y)\leq M\big\}\leq G(t)\,\mathds{1}\big\{\rho(x, y)\leq M\big\}+\mathds{1}\big\{\rho(x, y)\leq M\big\}\,,\qquad t\geq 0\,.
\end{eqnarray}
For $ \mathds{1}\big\{\rho(x, y)\leq M\big\}=1 $, we can conclude
\begin{eqnarray}\label{t<G(t)+1 eq3}
G(t)\,\mathds{1}\big\{\rho(x, y)\leq M\big\}&=&(t\log t)\,\mathds{1}\big\{\rho(x, y)\leq M\big\} \nonumber\\
&=&(t\log t)\,\mathds{1}\big\{\rho(x, y)\leq M\big\} \nonumber\\
&=& (t\,\mathds{1}\big\{\rho(x, y)\leq M\big\}) \log t \nonumber \\
&=& (t\,\mathds{1}\big\{\rho(x, y)\leq M\big\}) \log (t\,\mathds{1}\big\{\rho(x, y)\leq M\big\}) \nonumber\\
&=& G\Big(t\,\mathds{1}\big\{\rho(x, y)\leq M\big\}\Big)\,.
\end{eqnarray}
For $ \mathds{1}\big\{\rho(x, y)\leq M\big\}=0 $, we have
\begin{eqnarray}\label{t<G(t)+1 eq4}
G(t)\,\mathds{1}\big\{\rho(x, y)\leq M\big\}=0=G\Big(t\,\mathds{1}\big\{\rho(x, y)\leq M\big\}\Big)\,.
\end{eqnarray}
Substituting Eqs.(\ref{t<G(t)+1 eq3}) and (\ref{t<G(t)+1 eq4}) into Eq.((\ref{t<G(t)+1 eq2})), we obtain 
\begin{eqnarray}
t\,\mathds{1}\big\{\rho(x, y)\leq M\big\}\leq G\Big(t\,\mathds{1}\big\{\rho(x, y)\leq M\big\}\Big)+\mathds{1}\big\{\rho(x, y)\leq M\big\}\,,\qquad t\geq 0\,.
\end{eqnarray}
\begin{eqnarray}
&&\hspace{-.8cm}\Big|\mathsf{E} \Big[T_{N}(x,y)\mathds{1}\big\{\rho(x, y)\leq M\big\}\Big]\Big| \nonumber\\
&&\hspace{-1.2cm}=\Big|\int_{\mathbb{R}^{d}}\int_{\mathbb{R}^{d}} \Big[T_{N}(x,y)\mathds{1}\big\{\rho(x, y)\leq M\big\}\Big]f(x)f(y) dx\,dy \Big|\nonumber\\
&&\hspace{-1.2cm}=\Big|\int_{\mathbb{R}^{d}}\int_{\mathbb{R}^{d}} \mathsf{E}\Big[{\zeta_{1}(N)}^{2}{\zeta_{2}(N)}^{2} \,\mathds{1}\big\{\rho(x, y)\leq M\big\}\Big| X_{1}=x, X_{2}=y \Big]f(x)f(y) dx\,dy \Big|\nonumber\\
&&\hspace{-1.2cm}=\Big|\int_{\mathbb{R}^{d}}\int_{\mathbb{R}^{d}} \mathsf{E}\Big[\big(\log\eta^{y}_{N,x}\big)^{2}\big(\log\eta^{x}_{N,y}  \big)^{2}\,\mathds{1}\big\{\rho(x, y)\leq M\big\}\Big] f(x)f(y) dx\,dy \Big|\nonumber\\
&&\hspace{-1.2cm}\leq \int_{\mathbb{R}^{d}}\int_{\mathbb{R}^{d}} \mathsf{E}\Big[\Big|\big(\log\eta^{y}_{N,x}\big)^{2}\big(\log\eta^{x}_{N,y}  \big)^{2}\Big|\,\mathds{1}\big\{\rho(x, y)\leq M\big\}\Big] f(x)f(y) dx\,dy \nonumber\\
&&\hspace{-1.2cm}\leq \int_{\mathbb{R}^{d}}\int_{\mathbb{R}^{d}} G\Big\{\mathsf{E}\Big[\Big|\big(\log\eta^{y}_{N,x}\big)^{2}\big(\log\eta^{x}_{N,y}  \big)^{2}\Big|\,\mathds{1}\big\{\rho(x, y)\leq M\big\}\Big]\Big\} f(x)f(y) dx\,dy +\nonumber\\
&&\hspace{-.86cm}\int_{\mathbb{R}^{d}}\int_{\mathbb{R}^{d}} \,\mathds{1}\big\{\rho(x, y)\leq M\big\} f(x)f(y) dx\,dy \nonumber\\
&&\hspace{-1.2cm}=\int_{\big\{\rho(x, y)\leq M\big\}} G\Big\{\mathsf{E}\Big[\Big|\big(\log\eta^{y}_{N,x}\big)^{2}\big(\log\eta^{x}_{N,y}  \big)^{2}\Big]\Big\} f(x)f(y) dx\,dy +\mathsf{Pr} \Big(\rho(x,y)\leq M \Big)   \nonumber\\
&&\hspace{-1.2cm}\leq \int_{\big\{\rho(x, y)\leq M\big\}} \mathsf{E}\Big\{G\Big[\Big|\big(\log\eta^{y}_{N,x}\big)^{2}\big(\log\eta^{x}_{N,y}  \big)^{2}\Big] \Big\} f(x)f(y) dx\,dy+\mathsf{Pr} \Big(\rho(x,y)\leq M \Big) . \quad \label{ET_N,x,y rho<=M eq1}
\end{eqnarray}
The inequality in Eq.(\ref{ET_N,x,y rho<=M eq1}) is guaranteed by the Jensen's inequality. Analogous to results in Eqs.(\ref{ET_N,x,y rho>M eq1})-(\ref{Tf eq1}), we derive the 1st part of Eq.(\ref{ET_N,x,y rho<=M eq1}) as follows.
\begin{eqnarray}
&&\int_{\big\{\rho(x, y)\leq M\big\}} \mathsf{E}\Big\{G\Big[\Big|\big(\log\eta^{y}_{N,x}\big)^{2}\big(\log\eta^{x}_{N,y}  \big)^{2}\Big] \Big\} f(x)f(y) dx\,dy\,,\quad N\geq \tilde{\tilde{N}} \nonumber\\
&&\hspace{-.6cm}\leq 4\,\vartheta_{1}(N)\,\int_{\big\{\rho(x, y)\leq M\big\}}M_{f}^{\varepsilon_{1}}(x,R_{1})f(x)dx+ 4\vartheta_{2}(N)\int_{\big\{\rho(x, y)\leq M\big\}} m_{f}^{-\varepsilon}(x,R_{2})f(x)dx+  \nonumber\\
&& 4\vartheta_{3}(N)\int_{\big\{\rho(x, y)\leq M\big\}}m_{f}^{-\varepsilon}(x,R_{2})f(x)dx +4\vartheta_{4}(N)\,\int_{\big\{\rho(x, y)\leq M\big\}}m_{f}^{-\varepsilon}(x,R_{2})\,\widetilde{R}_{y}(N,x) f(x)dx \nonumber\\
&&\hspace{-.6cm}= 4\,\vartheta_{1}(N)\,\mathsf{E}\Big(M_{f}^{\varepsilon_{1}}(x,R_{1}) \,\mathds{1}\big\{\rho(x, y)\leq M\big\} \Big) + 4\vartheta_{2}(N)\mathsf{E}\Big(m_{f}^{-\varepsilon}(x,R_{2}) \,\mathds{1}\big\{\rho(x, y)\leq M\big\} \Big) +  \nonumber\\
&& 4\vartheta_{3}(N)\mathsf{E}\Big(m_{f}^{-\varepsilon}(x,R_{2}) \,\mathds{1}\big\{\rho(x, y)\leq M\big\} \Big) +4\vartheta_{4}(N)\,\mathsf{E}\Big(m_{f}^{-\varepsilon}(x,R_{2})\widetilde{R}_{y}(N,x)  \,\mathds{1}\big\{\rho(x, y)\leq M\big\} \Big) \nonumber\\
&&\hspace{-.6cm}:= \vartheta_{1}(N)\, \tilde{Q}_{f}(M)+\vartheta_{2}(N)\,\tilde{T}_{f}(M)+\vartheta_{3}(N)\,\tilde{T}_{f}(M)+\vartheta_{4}(N)\tilde{D}_{f}(M)  \label{EG_N,x,y rho<=M} \nonumber\\
&& \xrightarrow[\text{}]{\text{$M\rightarrow 0$}} 0\,,
\end{eqnarray} 
where
\begin{eqnarray}
&&\hspace{-1.3cm}\tilde{Q}_{f}(M)=4\, \mathsf{E}\Big(M_{f}^{\varepsilon_{1}}(x,R_{1})\,\mathds{1}\big\{\rho(x, y)\leq M\big\} \Big)    \nonumber\\
&&= 4\,\mathsf{E}\Big(M_{f}^{\varepsilon_{1}}(x,R_{1})\Big)  \,\mathds{1}\big\{\rho(x, y)\leq M\big\}  \nonumber\\
&& =4\,\underbrace{{Q}_{f}(\varepsilon_{1}, R_{1})}_{<\infty}\,\underbrace{\mathds{1}\big\{\rho(x, y)\leq M\big\} }_{\xrightarrow[\text{}]{\text{$M\rightarrow 0$}}0} \quad \xrightarrow[\text{}]{\text{$M\rightarrow 0$}} 0\,, \label{Qf eq2} 
\end{eqnarray}
and
\begin{eqnarray}
&&\hspace{-1.3cm}\tilde{T}_{f}(M)=4\, \mathsf{E}\Big(m_{f}^{-\varepsilon}(x,R_{2})\,\mathds{1}\big\{\rho(x, y)\leq M\big\} \Big)    \nonumber\\
&&= 4\,\mathsf{E}\Big(m_{f}^{-\varepsilon}(x,R_{2})\Big)  \,\mathds{1}\big\{\rho(x, y)\leq M\big\}  \nonumber\\
&& =4\,\underbrace{{T}_{f}(\varepsilon , R_{2})}_{<\infty}\,\underbrace{\mathds{1}\big\{\rho(x, y)\leq M\big\} }_{\xrightarrow[\text{}]{\text{$M\rightarrow 0$}}0} \quad \xrightarrow[\text{}]{\text{$M\rightarrow 0$}} 0\,, \label{Tf eq2} 
\end{eqnarray}
and
\begin{eqnarray}
&&\hspace{-1.3cm}\tilde{D}_{f}(M)=4\,\mathsf{E}\Big(m_{f}^{-\varepsilon}(x,R_{2})\,\widetilde{R}_{y}(N,x)  \,\mathds{1}\big\{\rho(x, y)\leq M\big\} \Big)    \nonumber\\
&&= 4\,\mathsf{E}\Big(\underbrace{m_{f}^{-\varepsilon}(x,R_{2})}_{<\infty}\,\underbrace{\widetilde{R}_{y}(N,x)  }_{<\infty}\Big) \,\underbrace{\mathds{1}\big\{\rho(x, y)\leq M\big\} }_{\xrightarrow[\text{}]{\text{$M\rightarrow 0$}}0}  \nonumber\\
&& \xrightarrow[\text{}]{\text{$M\rightarrow 0$}} 0\,. \label{Df eq1} 
\end{eqnarray}
Moreover,
Recalling from the Eq.(\ref{Prob rho<=M}) that $ \mathsf{Pr} \big(\rho(x,y)\leq M \big) \xrightarrow[\text{}]{\text{$M\rightarrow 0$}}  0 $. Substituting Eqs.(\ref{Prob rho<=M}) and (\ref{EG_N,x,y rho<=M}) into Eq.(\ref{ET_N,x,y rho<=M eq1}), we obtain
\begin{eqnarray}\label{ET_N,x,y rho<=M eq2}
\Big|\mathsf{E} \Big[T_{N}(x,y)\mathds{1}\big\{\rho(x, y)\leq M\big\}\Big]\Big|  \xrightarrow[\text{$N\rightarrow \infty$}]{\text{$M\rightarrow 0$}} 0\,.
\end{eqnarray}
Therefore, for any $ \kappa_{2}>0 $, there exits positive numbers $ M_{1}:=M_{1}(\kappa_{2}) $ and $ N_{5}:=N_{5}(\kappa_{2}) $ such that for $ M>M_{1} $ and $ N>N_{5} $, we have
\begin{eqnarray}
\Big|\mathsf{E} \Big[T_{N}(x,y)\mathds{1}\big\{\rho(x, y)\leq M\big\}\Big]\Big| =\Big|\int_{\big\{\rho(x, y)\leq M\big\}}T_{N}(x,y)f(x)f(y)dxdy \Big|< \kappa_{2}\,.
\end{eqnarray} 
Based on the \textit{Absolute continuity of the Lebesgue integral
}, we assume $ \varphi(x,y)= (\log^{2} f(x)+\sigma^{2})(\log^{2} g(x)+\sigma^{2})$ is integrable in $ \mathbb{R}^{d}\otimes\mathbb{R}^{d} $. Then for all $ \kappa_{2}>0 $, there exists a $ M_{2}=M_{2}(\kappa_{2}) $ such that, if the Lebesgue measure of $ \{\rho(x,y)\leq M\} \subset  \mathbb{R}^{d}\otimes\mathbb{R}^{d}  $ and is less than $ \delta_{M_{2}} $, i.e. $ \mu\{\rho(x,y)\leq M\}\leq \delta_{M_{2}}:=\mu\{\rho(x,y)\leq M_{2}\} $. Then the integral of $  \varphi(x,y) $ over $\{\rho(x,y)\leq M\}   $ is smaller than $ \kappa_{2} $, i.e.
\begin{eqnarray}\label{absolute cts legesgue}
\bigg|\int_{\big\{\rho(x, y)\leq M\big\}}\bigg(\log^{2} f(x)+\frac{\pi^{2}}{6}\bigg)\bigg(\log^{2} f(y)+\frac{\pi^{2}}{6}\bigg)f(x)f(y)dxdy \bigg|< \kappa_{2}\,.
\end{eqnarray}
Therefore, combing the Eqs.(\ref{ET_N,x,y rho>M eq2}),(\ref{ET_N,x,y rho<=M eq2}) and  (\ref{absolute cts legesgue}), we set $ M=\min \{M_{1},M_{2}\} $ and $ N=N_{6}(\tilde{M},\kappa_{2}) $, such that for $ N\geq \max\{ \tilde{\tilde{N}}, N_{6}, N_{6} \}$, we have
\begin{eqnarray}\label{ET_N,x,y rho>M eq3}
&&\hspace{-.6cm}\bigg|\int_{\big\{\rho(x, y)> M\big\}}T_{N}(x,y)f(x)f(y)dxdy -\int_{\big\{\rho(x, y)> M\big\}}\bigg(\log^{2} f(x)+\frac{\pi^{2}}{6}\bigg)\bigg(\log^{2} f(y)+\frac{\pi^{2}}{6}\bigg)
\nonumber\\
&&\hspace{-.6cm} f(x)f(y)dxdy \bigg|< \kappa_{2}
\end{eqnarray}
and
\begin{eqnarray}
&&\hspace{-.6cm}\bigg|\int_{\mathbb{R}^{d}}\int_{\mathbb{R}^{d}}T_{N}(x,y)f(x)f(y)dxdy -\int_{\mathbb{R}^{d}}\int_{\mathbb{R}^{d}}\bigg(\log^{2} f(x)+\frac{\pi^{2}}{6}\bigg)\bigg(\log^{2} f(y)+\frac{\pi^{2}}{6}\bigg)f(x)f(y)dxdy \bigg|  \nonumber\\
&&\hspace{-1.0cm}\leq \bigg|\int_{\big\{\rho(x, y)> M\big\}}T_{N}(x,y)f(x)f(y)dxdy -\int_{\big\{\rho(x, y)> M\big\}}\bigg(\log^{2} f(x)+\frac{\pi^{2}}{6}\bigg)\bigg(\log^{2} f(y)+\frac{\pi^{2}}{6}\bigg)
\nonumber\\
&&\hspace{-.6cm} f(x)f(y)dxdy \bigg|+\Big|\int_{\big\{\rho(x, y)\leq M\big\}}T_{N}(x,y)f(x)f(y)dxdy \Big| \nonumber\\
&&\hspace{-.6cm} +\bigg|\int_{\big\{\rho(x, y)\leq M\big\}}\bigg(\log^{2} f(x)+\frac{\pi^{2}}{6}\bigg)\bigg(\log^{2} f(y)+\frac{\pi^{2}}{6}\bigg)f(x)f(y)dxdy \bigg| \nonumber\\
&&\hspace{-1.0cm} \leq 3\kappa_{2}\,.
\end{eqnarray}
Thus, recall the Eq.(\ref{T_N,x,y def expectation}), we proved
\begin{eqnarray}
&&\mathsf{E}\Big({\zeta_{i}(N)}^{2}{\zeta_{j}(N)}^{2} \Big)
=\mathsf{E}\, T_{N}(x,y) \nonumber\\
&&\xrightarrow[\text{}]{\text{$N\rightarrow \infty$}} \int_{\mathbb{R}^{d}}\int_{\mathbb{R}^{d}}\bigg(\log^{2} f(x)+\frac{\pi^{2}}{6}\bigg)\bigg(\log^{2} f(y)+\frac{\pi^{2}}{6}\bigg)f(x)f(y)dxdy \,,\hspace{2cm} 
\end{eqnarray}
and consequently, we have
\begin{eqnarray}
&&\hspace{-0cm}\mathsf{Cov}\big[ {\zeta}_{i}^{2}(N)\,,\,{\zeta}_{j}^{2}(N)\big]=\mathsf{E}\big[ {\zeta}_{i}^{2}(N)\,{\zeta}_{j}^{2}(N)\big]-\mathsf{E}\big[ {\zeta}_{i}^{2}(N)\big]\mathsf{E}\big[ {\zeta}_{j}^{2}(N)\big] \nonumber\\
&&\xrightarrow[\text{}]{\text{$N\rightarrow \infty$}} \mathsf{E}\bigg(\log^{2} f(x)+\frac{\pi^{2}}{6}\bigg)\bigg(\log^{2} f(y)+\frac{\pi^{2}}{6}\bigg)-\mathsf{E}\bigg(\log^{2} f(x)+\frac{\pi^{2}}{6}\bigg)\bigg(\log^{2} f(y)+\frac{\pi^{2}}{6}\bigg)\nonumber\\
&&\xrightarrow[\text{}]{\text{$N\rightarrow \infty$}} 0\,,\qquad \text{as}\,\, x\perp y \Rightarrow f(x) \perp f(y) \,.
\end{eqnarray}
%%\begin{tcolorbox}
Summary: further to the two subsections above, we have proved the $ L^{2}-$consistence of $ S^{2}_{N}  $ in Eq.(\ref{SNx expansion}). 
%%In specific, we elaborate our proofs in the following formula:
%%\begin{eqnarray}\label{SNx expansion 3}
%%\mathsf{Var}\big(S_{N,x}^{2}\big)=\underbrace{\dfrac{1}{N}\underbrace{\mathsf{Var}\big[ \tilde{\zeta}_{1}^{2}(N)\big]}_{<\infty}}_{\xrightarrow[\text{}]{N\rightarrow \infty}\,0} +\dfrac{2}{N^{2}} \underbrace{\sum\limits_{1\leq i < j\leq N}^{N}\mathsf{Cov}\big[ \tilde{\zeta}_{i}^{2}(N)\,,\,\tilde{\zeta}_{j}^{2}(N)\big]}_{\xrightarrow[\text{}]{N\rightarrow \infty}\,0}
%%\end{eqnarray}
%%\end{tcolorbox}
%%%%%%%%%%%%%%%%%%%%%%%%%%%%%%%%%%%%%%%%%%%%%%%%%%%%%%%%%%%%%%%%%%%%%%%%%%%%%%%%%%%%%%%%%%%%%%%%%%%%%%%%%%%%%%%%%%%%%%%%%%%%%%%%%%%%
%\subsection{$ \mathsf{Var}(H^{2}_{N}) $  }
\subsection{$  \mathsf{Var}(H^{2}_{N}) \,\xrightarrow[\text{}]{N\rightarrow \infty}\, 0$}
%%%%%%%%%%%%%%%%%%%%%%%%%%%%%%%%%%%%%%%%%%%%%%%%%%%%%%%%%%%%%%%%%%%%%%%%%%%%%%%%%%%%%%%%%%%%%%%%%%%%%%%%%%%%%%%%%%%%%%%%%%%%%%%%%%%%
An equivalent statement is 
\begin{eqnarray}
\mathsf{Var}(H^{2}_{N}) \,\xrightarrow[\text{}]{N\rightarrow \infty}\, 0  \quad \Leftrightarrow \quad \mathsf{E}\big[(H_{N})^{4}\big] \,\xrightarrow[\text{}]{N\rightarrow \infty}\mathsf{E}^{2}(H_{N}^{2}) =[\mathsf{E}(H_{N})]^{4}\,,
\end{eqnarray}
where $ \mathsf{E}^{2}(H_{N}^{2}) =[\mathsf{E}(H_{N})]^{4} $ is given by the fact that $ \mathsf{Var}(H_{N})=0 $, $  N\rightarrow\infty$, i.e.
\begin{eqnarray}
\mathsf{Var}(H_{N})=0 \quad \Rightarrow \quad \mathsf{E}(H_{N}^{2})=[\mathsf{E}(H_{N})]^{2} \quad \Rightarrow \quad [\mathsf{E}(H_{N}^{2})]^{2}=[\mathsf{E}(H_{N})]^{4}\,.
\end{eqnarray}
Recall that $ X_{1},\ldots , X_{N} $ are i.i.d. random vectors, and this implies that, for any $ \alpha>0 $, random variables $ \zeta_{1}^{\alpha},\ldots , \zeta_{N}^{\alpha} $ are identically distributed, based on the definitions of $ \zeta_{i} $. Thus, we have $ \mathsf{E}(\zeta_{1}^{\alpha})=\cdots\mathsf{E}(\zeta_{N}^{\alpha}) $, $ \alpha=1,2,3,4 $.
\begin{eqnarray}
\mathsf{E}(H_{N}^{4})&=&\mathsf{E}\bigg[\Big(\frac{1}{N}\sum\limits_{i=1}^{N}\zeta_{i}\Big)^{4}\bigg]\nonumber\\
&=&\frac{1}{N^{4}}\mathsf{E}\bigg(\sum\limits_{i=1}^{N}\zeta_{i}^{4}+4\sum\limits_{i=1}^{N}\sum\limits_{\substack{j=1\\ j\neq i}}^{N}\zeta_{i}^{3}\zeta_{j}+3\sum\limits_{i=1}^{N}\sum\limits_{\substack{j=1\\ j\neq i}}^{N}\zeta_{i}^{2}\zeta_{j}^{2}+ 6\sum\limits_{i=1}^{N}\sum\limits_{\substack{j=1\\ j\neq i}}^{N}\sum\limits_{\substack{k=1\\ k\neq i,j}}^{N}\zeta_{i}^{2}\zeta_{j}\zeta_{k}+ \nonumber\\
&&\sum\limits_{i=1}^{N}\sum\limits_{\substack{j=1\\ j\neq i}}^{N}\sum\limits_{\substack{k=1\\ k\neq i,j}}^{N}\sum\limits_{\substack{l=1\\ k\neq i,j,k}}^{N}\zeta_{i}\zeta_{j}\zeta_{k}\zeta_{l}\bigg) \nonumber\\
&=&\frac{1}{N^{4}}\bigg\{\sum\limits_{i=1}^{N}\mathsf{E}\big(\zeta_{i}^{4}\big)+4\sum\limits_{i=1}^{N}\sum\limits_{\substack{j=1\\ j\neq i}}^{N}\mathsf{E}\big(\zeta_{i}^{3}\zeta_{j}\big)+3\sum\limits_{i=1}^{N}\sum\limits_{\substack{j=1\\ j\neq i}}^{N}\mathsf{E}\big(\zeta_{i}^{2}\zeta_{j}^{2}\big)+ \nonumber\\
&& 6\sum\limits_{i=1}^{N}\sum\limits_{\substack{j=1\\ j\neq i}}^{N}\sum\limits_{\substack{k=1\\ k\neq i,j}}^{N}\mathsf{E}\big(\zeta_{i}^{2}\zeta_{j}\zeta_{k}\big) +\sum\limits_{i=1}^{N}\sum\limits_{\substack{j=1\\ j\neq i}}^{N}\sum\limits_{\substack{k=1\\ k\neq i,j}}^{N}\sum\limits_{\substack{l=1\\ k\neq i,j,k}}^{N}\mathsf{E}\big(\zeta_{i}\zeta_{j}\zeta_{k}\zeta_{l}\big)\bigg\}\,.\label{EH_N power4 eq1}
\end{eqnarray}
We consider a generalised H\"older's inequality for the product of $ m $ random variables, where \(Z_1, Z_2, \ldots, Z_m\) are random variables, and \(q_1, q_2, \ldots, q_m\) are positive real numbers such that \(\frac{1}{q_1} + \frac{1}{q_2} + \ldots + \frac{1}{q_m} = 1\). Then, the $ m $-dimensional H\"older's inequality for \(E(Z_1Z_2 \ldots Z_N)\) is given by:
\begin{eqnarray}\label{Holder inequality extension m dimension}
E(|Z_1Z_2 \ldots Z_m|) \leq \left(E(|Z_1|^{q_1})\right)^{1/q_1} \cdot \left(E(|Z_2|^{q_2})\right)^{1/q_2} \cdot \ldots \cdot \left(E(|Z_N|^{q_N})\right)^{1/q_N}\,.
\end{eqnarray}
Let $ m=2 $, $ Z_{1}=\zeta_{i}^{\beta}$, $Z_{2}=\zeta_{j}^{(\alpha-\beta)} $ for $ \alpha\,,\beta\in\mathbb{N} $ and $ 1\leq \beta\leq \alpha $, $q_{1}= \frac{\alpha}{\beta} $ and  $q_{2}= \frac{\alpha}{\alpha-\beta} $.
\begin{eqnarray}\label{Holder inequality zeta_i,j alpha+beta}
\mathsf {E}\Big[\zeta_{i}^{\beta}\zeta_{j}^{(\alpha-\beta)}\Big]\leq \mathsf {E} \Big[\big|\zeta_{i}^{\beta}\zeta_{j}^{(\alpha-\beta)}\big|\Big]&\leq &\Big[\mathsf{E} \Big(\big|\zeta_{i}^{\beta}\big|^{\frac{\alpha}{\beta}}\Big)\Big]^{\frac {\beta}{\alpha}}\Big[\mathsf{E} \Big(\big|\zeta_{j}^{(\alpha-\beta)}\big|^{\frac{\alpha}{\alpha-\beta}}\Big)\Big]^{\frac {\alpha-\beta}{\alpha}} \nonumber\\
&=& \big[\mathsf{E} \big(\big|\zeta_{i}^{\alpha}\big|\big)\big]^{\frac {\beta}{\alpha}}\big[\mathsf{E} \big(\big|\zeta_{j}^{\alpha}\big|\big)\big]^{\frac {\alpha-\beta}{\alpha}} \nonumber\\
&=&  \mathsf{E} \big(\big|\zeta_{1}^{\alpha}\big|\big)\,.
\end{eqnarray} 
Recall that $ X_{1},\ldots , X_{N} $ are i.i.d. random vectors, and this implies that, for any $ \alpha>0 $, random variables $ \zeta_{1}^{\alpha},\ldots , \zeta_{N}^{\alpha} $ are identically distributed, based on the definitions of $ \zeta_{i}$. Thus, we have $ \mathsf{E}(\zeta_{1}^{\alpha})=\cdots\mathsf{E}(\zeta_{N}^{\alpha}) $. 
For $ \alpha=4 $ and $ \beta=1,2 $ respectively, we have
\begin{eqnarray}
&&\hspace{-1.2cm}\mathsf {E}(\zeta_{i}^{3}\zeta_{j})\leq \mathsf {E} (|\zeta_{i}^{3}\zeta_{j}|)\leq \Big[\mathsf{E} \Big(\big|\zeta_{i}^{3}\big|^{\frac{4}{3}}\Big)\Big]^{\frac {3}{4}}\Big[\mathsf{E} \Big(\big|\zeta_{j}^{1}\big|^{\frac{4}{1}}\Big)\Big]^{\frac {1}{4}}=\big[\mathsf{E} \big(\big|\zeta_{i}^{4}\big|\big)\big]^{\frac {3}{4}}\big[\mathsf{E} \big(\big|\zeta_{j}^{4}\big|\big)\big]^{\frac {1}{4}} =\mathsf {E} \big(\zeta_{1}^{4}\big),\quad \label{Holder inequality zeta_i,j 3+1}\\
&&\hspace{-1.2cm}\mathsf {E}(\zeta_{i}^{2}\zeta_{2}^{2})= \mathsf {E} (|\zeta_{i}^{2}\zeta_{j}^{2}|)\leq \Big[\mathsf{E} \Big(\big|\zeta_{i}^{2}\big|^{2}\Big)\Big]^{\frac {1}{2}}\Big[\mathsf{E} \Big(\big|\zeta_{j}^{2}\big|^{2}\Big)\Big]^{\frac {1}{2}}=\big[\mathsf{E} \big(\big|\zeta_{i}^{4}\big|\big)\big]^{\frac {1}{2}}\big[\mathsf{E} \big(\big|\zeta_{j}^{4}\big|\big)\big]^{\frac {1}{2}} =\mathsf {E} \big(\zeta_{1}^{4}\big).\quad \label{Holder inequality zeta_i,j 2+2}
\end{eqnarray} 
Substituting $ m=3 $, $ Z_{1}=\zeta_{i}^{2}$, $Z_{2}=\zeta_{j}$, $Z_{3}=\zeta_{k}$  for $q_{1}= 2$,  $q_{2}= 4 $, $q_{3}= 4 $ into Eq.(\ref{Holder inequality extension m dimension}), we obtain
\begin{eqnarray}\label{Holder inequality zeta_i,j,k 2+1+1}
\mathsf {E}(\zeta_{i}^{2}\zeta_{j}\zeta_{k})\leq \mathsf {E} (|\zeta_{i}^{2}\zeta_{j}\zeta_{k}|) &\leq &\Big[\mathsf{E} \Big(\big|\zeta_{i}^{2}\big|^{2}\Big)\Big]^{\frac {1}{2}}\Big[\mathsf{E} \Big(\big|\zeta_{j}\big|^{4}\Big)\Big]^{\frac {1}{4}}\Big[\mathsf{E} \Big(\big|\zeta_{k}\big|^{4}\Big)\Big]^{\frac {1}{4}}  \nonumber\\
&=&  \mathsf{E} \big(\zeta_{1}^{4}\big)\,.
\end{eqnarray} 
%Substituting $ m=4 $, $ Z_{1}=\zeta_{i}$, $Z_{2}=\zeta_{j}$, $Z_{3}=\zeta_{k}$, $Z_{4}=\zeta_{l}$ for $q_{1}=q_{2}=q_{3}=q_{4} =4 $ into Eq.(\ref{Holder inequality extension m dimension}), we obtain
%\begin{eqnarray}\label{Holder inequality zeta_i,j,k,l 1+1+1+1}
%\mathsf {E}(\zeta_{i}\zeta_{j}\zeta_{k}\zeta_{l})\leq \mathsf {E} (|\zeta_{i}\zeta_{j}\zeta_{k}\zeta_{k}|) &\leq &\Big[\mathsf{E} \Big(\big|\zeta_{i}\big|^{4}\Big)\Big]^{\frac {1}{4}}\Big[\mathsf{E} \Big(\big|\zeta_{j}\big|^{4}\Big)\Big]^{\frac {1}{4}}\Big[\mathsf{E} \Big(\big|\zeta_{k}\big|^{4}\Big)\Big]^{\frac {1}{4}}\Big[\mathsf{E} \Big(\big|\zeta_{l}\big|^{4}\Big)\Big]^{\frac {1}{4}}  \nonumber\\
%&=&  \mathsf{E} \big(\zeta_{1}^{4}\big)
%\end{eqnarray} 
Substituting Eqs.(\ref{Holder inequality zeta_i,j 3+1})-(\ref{Holder inequality zeta_i,j,k 2+1+1}) into Eq.(\ref{EH_N power4 eq1}), we have
\begin{eqnarray}
\mathsf{E}(H_{N}^{4})&\leq &\frac{1}{N^{4}}\bigg\{\sum\limits_{i=1}^{N}\mathsf{E}\big(\zeta_{i}^{4}\big)+4\sum\limits_{i=1}^{N}\sum\limits_{\substack{j=1\\ j\neq i}}^{N}\mathsf{E}\big(\zeta_{i}^{4}\big)+3\sum\limits_{i=1}^{N}\sum\limits_{\substack{j=1\\ j\neq i}}^{N}\mathsf{E}\big(\zeta_{i}^{4}\big)+ 6\sum\limits_{i=1}^{N}\sum\limits_{\substack{j=1\\ j\neq i}}^{N}\sum\limits_{\substack{k=1\\ k\neq i,j}}^{N}\mathsf{E}\big(\zeta_{i}^{4}\big)   \nonumber\\
&&+\sum\limits_{i=1}^{N}\sum\limits_{\substack{j=1\\ j\neq i}}^{N}\sum\limits_{\substack{k=1\\ k\neq i,j}}^{N}\sum\limits_{\substack{l=1\\ k\neq i,j,k}}^{N}\mathsf{E}\big(\zeta_{i}\zeta_{j}\zeta_{k}\zeta_{l}\big)\bigg\}\nonumber \\
&=&\bigg[\dfrac{1}{N^{3}}+\dfrac{4(N-1)}{N^{3}}+\dfrac{3(N-1)}{N^{3}}+\dfrac{6(N-1)(N-2)}{N^{3}}\bigg]\mathsf{E}(\zeta_{1}^{4})+\nonumber\\
&&\dfrac{(N-1)(N-2)(N-3)}{N^{3}}\mathsf{E}\big(\zeta_{i}\zeta_{j}\zeta_{k}\zeta_{l}\big) \nonumber\\
&=& [o(N^{-3})+o(N^{-2})+o(N^{-1})] \mathsf{E}(\zeta_{1}^{4})+O(1)\mathsf{E}\big(\zeta_{i}\zeta_{j}\zeta_{k}\zeta_{l}\big)  \nonumber\\
&\rightarrow & \mathsf{E}\big(\zeta_{i}\zeta_{j}\zeta_{k}\zeta_{l}\big) \,,\quad N\rightarrow\infty\,.
\end{eqnarray}
Therefore, $\mathsf{E}(H_{N}^{4})  $ is bounded by $\mathsf{E}\big(\zeta_{i}\zeta_{j}\zeta_{k}\zeta_{l}\big)   $ when $ N $ tends to infinity. We set $ \widetilde{\zeta}_{i,j}= \zeta_{i}\zeta_{j}$ and $\widetilde{\zeta}_{k,l}= \zeta_{k}\zeta_{l}$, for any $ 1\leq i<j<k<l\leq N $. {\color{black}We can conclude that $ \widetilde{\zeta}_{i,j} $ and $ \widetilde{\zeta}_{k,l} $ are identically distributed because $\zeta_{1},\ldots ,\zeta_{N}  $ are identically distributed, i.e.
\begin{eqnarray}
X_{1},\ldots ,X_{N}\, \text{ i.i.d. }\,\,\Rightarrow \,\,\zeta_{1},\ldots ,\zeta_{N}\, \text{ identical} \,\,\Rightarrow \,\,  \widetilde{\zeta}_{i,j}, \widetilde{\zeta}_{k,l}\text{ identical}\,.
\end{eqnarray}
}
Thus, we aim to prove that 
\begin{eqnarray}\label{zeta_i,j,k,l exp}
\mathsf{E}\big(\zeta_{i}\zeta_{j}\zeta_{k}\zeta_{l}\big) =\mathsf{E}\big(\widetilde{\zeta}_{i,j}\,\widetilde{\zeta}_{i,j}\big)\,\rightarrow\,\mathsf{E}\big(\widetilde{\zeta}_{i,j}\big)\,\mathsf{E}\big(\widetilde{\zeta}_{l,k}\big)=\big[\mathsf{E}\big(\widetilde{\zeta}_{i,j}\big)\big]^{2} \,,\quad N\rightarrow\infty\,.
\end{eqnarray}  
Or equivalently
\begin{eqnarray}
\mathsf{Cov}\big(\widetilde{\zeta}_{i,j}\,,\widetilde{\zeta}_{i,j}\big)\rightarrow 0 \,,\quad N\rightarrow\infty  \,.
\end{eqnarray}
Recall that we have proved $ \mathsf{Cov} ({\zeta}_{i}\,,{\zeta}_{j})\rightarrow 0\,,\, N\rightarrow\infty $, and thus 
\begin{eqnarray}\label{zeta_i,j exp}
\mathsf{E}\big(\widetilde{\zeta}_{i,j}\big)=\mathsf{E}\big({\zeta}_{i}{\zeta}_{i}\big)\longrightarrow \mathsf{E}\big({\zeta}_{i}\big)\mathsf{E}\big({\zeta}_{j}\big)=\big[\mathsf{E}\big({\zeta}_{i}\big)\big]^{2} \,,\quad N\rightarrow\infty  \,.
\end{eqnarray}
%%\begin{eqnarray}\label{•}
%%\mathsf{E}\big(\widetilde{\zeta}_{i,j}\big)\rightarrow \big[\mathsf{E}\big({\zeta}_{i}\big)\big]^{2} \,,\quad N\rightarrow\infty 
%%\end{eqnarray}
Substituting Eq.(\ref{zeta_i,j exp}) into (\ref{zeta_i,j,k,l exp}), we obtain:
\begin{eqnarray}
\mathsf{E}\big(\zeta_{i}\zeta_{j}\zeta_{k}\zeta_{l}\big) =\mathsf{E}\big(\widetilde{\zeta}_{i,j}\,\widetilde{\zeta}_{k,l}\big)\rightarrow \big[\mathsf{E}\big(\widetilde{\zeta}_{i,j}\big)\big]^{2}=\big[\mathsf{E}\big({\zeta}_{i}\big)\big]^{4} \,,\quad N\rightarrow\infty\,,
\end{eqnarray}
and this implies:
\begin{eqnarray}
\mathsf{E}(H_{N}^{4})=[\mathsf{E}(H_{N})]^{4}\quad \Longleftrightarrow \quad\mathsf{Var}(H_{N}^{2})=0 \,\quad N\rightarrow\infty\,.
\end{eqnarray}\,.
\section*{Acknowledgements}
Nikolai Leonenko (NL) would like to thank for support and hospitality during the programme ``Fractional Differential Equations'' and the programme ``Uncertainly Quantification and Modelling of Materials'' in Isaac Newton Institute for Mathematical Sciences, Cambridge. Also NL was partially supported under the ARC Discovery Grant DP220101680 (Australia), LMS grant 42997 (UK) and grant FAPESP 22/09201-8 (Brazil).
\section*{Appendix}
Supplementary Material to `Varentropy Estimation via Nearest Neighbor Graphs'.
%%In order to prove the above two subsections about Variance and Covariance, some \underline{technical proofs are provided in {\color{red} Note-4 Appendix A,B,C.I and C.II} (attached)}.

%%%%%%%%%%%%%%%%%%%%%%%%%%%%%%%%%%%%%%%%%%%%%%%%%%%%%%%%%%%%%%%%%%%%%%%%%%%%%%%%%%%%%%%%%%%%%%%%%%%%%%%%%%%%%%%%%%%%%%%%%%%%%%%%%%%%%%%%%%%%
\newpage
%%%%%%%%%%%%%%%%%%%%%%%%%%%%%%%%%%%%%%%%%%%%%%%%%%%%%%%%%%%%%%%%%%%%%%%%%%%%%%%%%%%%%%%%%%%%%%%%%%%%%%%%%%%%%%%%%%%%%%%%%%%%%%%%%%%%%%%%%%%%
\section*{Supplementary Material to `Varentropy Estimation via Nearest Neighbor Graphs'}

\author{\hspace{2cm} Nikolai Leonenko  \hspace{2cm}   Yu Sun \hspace{2cm}    Emanuele Taufer \hspace{-2cm}}
\subsection*{Appendix: Proofs and auxiliary results in the main text.}
%%\subsection*{A.1  Proofs of auxiliary results of Theorem 2 in the main text.}
\subsection*{A.1 Proofs of Lemma 2 and Lemma 3}
\begin{proof}%[Proof of Lemma 2 and Lemma 3]
\textbf{Part.1} Following the integration by parts, we have
\begin{eqnarray}\label{lemma power 3 proof}
&&\int\limits_{(0,\frac{1}{e}]}(-\log u)^{3}\,\log (-\log u) d F(u)\nonumber\\
&=&\bigg[(-\log u)^{3}\,\log (-\log u)F(u)\bigg]_{0}^{\frac{1}{e}}-\int\limits_{(0,\frac{1}{e}]} F(u)d\bigg[(-\log u)^{3}\,\log (-\log u) \bigg]
\end{eqnarray}
The 1st part of the LHS $\bigg[(-\log u)^{3}\,\log (-\log u)F(u)\bigg]_{0}^{\frac{1}{e}} $ is 0 because $ F(0)=0 $ and $ \log (-\log \frac{1}{e})=0 $. The 2nd part of the LHS is 
\begin{eqnarray}\label{lemma power 3 proof LHS}
-\int\limits_{(0,\frac{1}{e}]} F(u)d\bigg[(-\log u)^{3}\,\log (-\log u)\bigg]=3\int\limits_{(0,\frac{1}{e}]} F(u)\dfrac{(\log u)^{2}}{u}\bigg[\log(-\log u)+\dfrac{1}{3}\bigg]du\quad
\end{eqnarray}
Substitute Eq.(\ref{lemma power 3 proof LHS}) into Eq.(\ref{lemma power 3 proof}), we obtain the statement (1) of \textit{Lemma} 2\\%\ref{lemma power 3}. \\
\textbf{Part.2} We substitute $ F(u) \longrightarrow 1-F(u) $ in the deferential part as follows
\begin{eqnarray}\label{lemma power 3 proof (2.1)}
\int\limits_{[e,\infty)}(\log u)^{3}\,\log (\log u) d F(u)=-\int\limits_{[e,\infty)}(\log u)^{3}\,\log (\log u) d [1-F(u)]
\end{eqnarray}
Applying the rule of integration by parts, we have
\begin{eqnarray}\label{lemma power 3 proof (2.2)}
\hspace{-1.5cm}&&\int\limits_{[e,\infty)}(\log u)^{3}\,\log (\log u) d [1-F(u)]\nonumber\\
\hspace{-1.5cm}&=&-\bigg[\int\limits_{[e,\infty)}(\log u)^{3}\,\log (\log u)[1- F(u)]\bigg]_{e}^{\infty}+\int\limits_{[e,\infty)} [1-F(u)]d\bigg[(\log u)^{3}\,\log (\log u) \bigg]
\end{eqnarray}
The 1st part of the LHS $-\bigg[\int_{[e,\infty)}(\log u)^{3}\,\log (\log u)[1- F(u)]\bigg]_{e}^{\infty} $ is 0 because $ 1-F(\infty)=0 $ and $ \log (\log e)=0 $. The 2nd part of the LHS is 
\begin{eqnarray}\label{lemma power 3 proof LHS (2)}
\int\limits_{[e,\infty)}[1- F(u)]d\bigg[(\log u)^{3}\,\log (\log u)\bigg]=3\int\limits_{[e,\infty)} F(u)\dfrac{(\log u)^{2}}{u}\bigg[\log(\log u)+\dfrac{1}{3}\bigg]du\quad
\end{eqnarray}
Substitute Eq.(\ref{lemma power 3 proof LHS (2)}) into (\ref{lemma power 3 proof (2.2)}) and (\ref{lemma power 3 proof (2.1)}) , we obtain the statement (2) of \textit{Lemma} 2.\\
\textbf{Part.3} The proof of \textit{Lemma} 3 is analogous to the proof of \textit{Lemma} 2, and thus we skipped the details.
\end{proof}
%%%%%%%%%%%%%%%%%%%%%%%%%%%%%%%%%%%%%%%%%%%%%%%%%%%%%%%%%%%%%%%%%%%%%%%%%%%%%%%%%%%%%%%%%%%%%%%%%%%%%%%%%%%%%%%%%%%%%%%%%%%%%%%%%%%%%%%%%%
\subsection*{A.2 Auxiliary results for the proof of finiteness of $ I_{1}(N,x) $}\label{section I1}
%%%%%%%%%%%%%%%%%%%%%%%%%%%%%%%%%%%%%%%%%%%%%%%%%%%%%%%%%%%%%%%%%%%%%%%%%%%%%%%%%%%%%%%%%%%%%%%%%%%%%%%%%%%%%%%%%%%%%%%%%%%%%%%%%%%%%%%%%%%%
Note that
%\begin{corollary}\label{corollary bernulli inequality xi X N}
$ \varepsilon\in(0,1] $, $ x\in[0,1] $, then for all $ N>1 $, we have
\begin{eqnarray}\label{bernulli inequality xi X N}
1-(1-x)^{N}\leq (Nx)^{\varepsilon}
\end{eqnarray}
%\end{corollary}
%\begin{proof}
Really by the \textit{Bernoulli Inequality}: $ 1-(1-x)^{N}\leq Nx\,,\,\,x\in[0,1] $. 
Let us consider $x\in[0,\frac{1}{N}]$ and $x\in[\frac{1}{N}, 1]\,,\,\,N>1  $ separately.
\begin{itemize}
\item for $x\in[0,\frac{1}{N}]$, we have $ Nx\in[0,1] \,\,\Rightarrow\, Nx\leq (Nx)^{\varepsilon}\,,\,\,\varepsilon\in(0,1] $, and substitute it into \textit{Bernoulli Inequality}, we have $ 1-(1-x)^{N}\leq (Nx)^{\varepsilon}$.
\item for $x\in[\frac{1}{N},1]$, we have $ Nx>1\, \Rightarrow\,(Nx)^{\varepsilon}>1 $, and thus $ 1-(1-x)^{N}\leq 1 \leq (Nx)^{\varepsilon}$.
\end{itemize}
Thus, Eq.(\ref{bernulli inequality xi X N}) is proven.
\subsection*{A.3 Auxiliary results for the proof of finiteness of $J_{1}(N,x)$}
%%%%%%%%%%%%%%%%%%%%%%%%%%%%%%%%%%%%%%%%%%%%%%%%%%%%%%%%%%%%%%%%%%%%%%
%%%%%%%%%%%%%%%%%%%%%%%%%%%%%%%%%%%%%%%%%%%%%%%%%%%%%%%%%%%%%%%%%%%%%%
%\begin{proof}
Really, for $ t\in[0,1] $, we have $0\leq 1-t\leq e^{-t} $, and thus power function $ (\cdot)^{N} $ satisfies 
\begin{align}\label{(1-t)N<e-tN} 
(1-t)^{N}\leq e^{-tN} 
\end{align}
We consider
%\end{proof}
\begin{eqnarray}
1-F_{N,x}(u)&=&\big[1-\mathsf{P}_{N,x}(u)\big]^{N-1}\nonumber\\
&\leq & e^{-(N-1)\,\mathsf{P}_{N,x}(u)} \label{1-F_Nx(u)}\\
%&=& e^{-(N-1)\,\mathsf{P}_{N,x}(u)} \nonumber\\
&=& \exp{\Bigg(-\dfrac{u}{e^{\gamma}}\,\dfrac{\int_{B\big(x,r_{N}(u)\big)}f(y)dy}{r_{N}^{d}(u)V_{d}}\Bigg)}\nonumber\\
&\leq & \exp{\big(-\frac{u}{e^{\gamma}}\,m_{f}(x,R_{2})\big)} \label{1-F_Nx(u) 1}
\end{eqnarray}
%%where $ m_{f}(x,R_{2}) $ is defined as follows.
%%\begin{eqnarray}\label{m_f(x,R_2)}
%%\dfrac{\mathsf{P}_{N,x}(u)}{\mid B\big(x,r_{N}(u)\big)\mid}=\dfrac{\int_{B\big(x,r_{N}(u)\big)}f(y)dy}{r_{N}^{d}(u)V_{d}}\geq\inf\limits_{r\in(0,R_{2}]}\dfrac{\int_{B\big(x,r\big)}f(y)dy}{r^{d}V_{d}}:=m_{f}(x,R_{2})
%%\end{eqnarray}
%%For a density $ f $ and $ R_{2}>0 $, we define $\mathcal{D}_{f}(R_{2})=\{x\in S(f): m_{f}(x,R_{2})>0\}  $ with $ m_{f}(x,R_{2}) $ defined in Eq.(\ref{m_f(x,R_2)}). Then, we have $ \mu\big(\mathcal{S}(f)\setminus \mathcal{D}_{f}(R_{2})\big)=0 $, see, Lemma 3.2 of Bulinski and Dimitrov (2019).\\
%%The following results are remarkable.
%\begin{corollary}\label{corollary e_-t < t_-delta}
Note that $ t>0 $ and $ \delta\in(0,e] $,  we have
\begin{eqnarray}\label{e_-t < t_-delta}
e^{-t}\leq t^{-\delta}
\end{eqnarray}
%\end{corollary}
%\begin{proof}
Really, for $ t>0 $ and $ \delta\in(0,e] $, we have $ 0<t^{-e}\leq t^{-\delta}<1 $. Therefore, the original statement is guaranteed if we prove the following statement
\begin{eqnarray}\label{e_-t < t_-e}
e^{-t}\leq t^{-e}\quad\Longleftrightarrow\quad \dfrac{\ln t}{t}\leq \dfrac{1}{e}
\end{eqnarray}
Let us define $ g(t)=\dfrac{\ln t}{t}\,,\,t>0 $. To prove Eq.(\ref{e_-t < t_-e}) by separating the domain of $ t $ into $t\in[0,1]$ and $t\in(1, \infty) $ respectively.
\begin{itemize}
\item for $t\in(0,1]$, it is worth nothing that $\ln t\leq 0 $. Considering the fact that $ \dfrac{1}{e}\geq 0 $, and consequently, we have $ \dfrac{\ln t}{t}\leq \dfrac{1}{e} $.
\item for $t\in(1,\infty)$, we calculate the 1st order of $g(t) $ as follows
\begin{eqnarray}\label{gx}
g'(t)=\frac{1}{t^{2}}(1-\ln t)
    \begin{cases}
      >0, & \text{if}\ t\in(1,e] \\
      =0, & \text{if}\ t=e\\
      <0, & \text{if}\ t\in(e,\infty)
    \end{cases}
\end{eqnarray}
Therefore, we have $ g(x)\leq \sup\limits_{t\in(1,\infty)} g(x)=g(e)=\dfrac{1}{e}$, which implies Eq.(\ref{e_-t < t_-e}).
\end{itemize}
Therefore, we proved $e^{-t}\leq t^{-t}\leq t^{-\delta}  $ and (\ref{e_-t < t_-delta}) is proven.
%\end{proof}
From the Eqs.(\ref{1-F_Nx(u) 1})-(\ref{e_-t < t_-delta}), we have:
\begin{eqnarray}\label{1-F_Nx(u) 2}
1-F_{N,x}(u) &\leq & \bigg[\frac{u}{e^{\gamma}}\,m_{f}(x, R_{2})\bigg]^{-\varepsilon}\,,\quad \varepsilon\in(0,e]
\end{eqnarray}
%%%%%%%%%%%%%%%%%%%%%%%%%%%%%%%%%%%%%%%%%%%%%%%%%%%%%%%%%%%%%%%%%%%%%%
%%%%%%%%%%%%%%%%%%%%%%%%%%%%%%%%%%%%%%%%%%%%%%%%%%%%%%%%%%%%%%%%%%%%%%
%%\subparagraph*{A.1.4 Auxiliary results for the investigation of $J_{2.1}(N,x)$}\label{subparagraph J_2 2.1}
\subsection*{A.4 Auxiliary results for the investigation of $J_{2.1}(N,x)$}\label{subparagraph J_2 2.1}
%%%%%%%%%%%%%%%%%%%%%%%%%%%%%%%%%%%%%%%%%%%%%%%%%%%%%%%%%%%%%%%%%%%%%%
%%%%%%%%%%%%%%%%%%%%%%%%%%%%%%%%%%%%%%%%%%%%%%%%%%%%%%%%%%%%%%%%%%%%%%
Substituting Eq.(\ref{(1-t)N<e-tN}) into $J_{2.1}(N,x)$, we have 
%Eqs.(\ref{(1-t)N<e-tN}) and (\ref{J_2 2.1 1}), we obtain:
\begin{eqnarray}
J_{2.1}(N,x)
&\leq & e^{-(N-2)\,\mathsf{P}_{N,x}(\sqrt{N-1})} \label{J_2 2.1 2} \nonumber\\
&=& \exp\left\{-(N-2)\dfrac{\mathsf{P}_{N,x}(\sqrt{N-1})}{\mid B\big(x,r_{N}(\sqrt{N-1})\big)\mid}\,\big| B\big(x,r_{N}(\sqrt{N-1})\big)\big|\right\}               \nonumber\\
&=& \exp\left\{-\dfrac{N-2}{N-1}\dfrac{\sqrt{N-1}}{e^{\gamma}}\,\dfrac{\mathsf{P}_{N,x}(\sqrt{N-1})}{\mid B\big(x,r_{N}(\sqrt{N-1})\big)\mid}\right\}
 \label{J_2 2.1 3}
\end{eqnarray}
Using $ u=\sqrt{N-1} $, we obtain % from Eq.(\ref{m_f(x,R_2)}) that
\begin{eqnarray}\label{m_f(x,R_2) 2}
\dfrac{\mathsf{P}_{N,x}(\sqrt{N-1})}{\mid B\big(x,r_{N}(\sqrt{N-1})\big)\mid}&=&\dfrac{\int_{B\big(x,r_{N}(\sqrt{N-1})\big)}f(y)dy}{r_{N}^{d}(\sqrt{N-1})V_{d}}\nonumber\\
&\geq &\inf\limits_{r\in(0,R_{2}]}\dfrac{\int_{B\big(x,r\big)}f(y)dy}{r^{d}V_{d}}\nonumber\\
&:=&m_{f}(x,R_{2})
\end{eqnarray}
%where $ R_{2} $ is defined in Eq.(\ref{R2}). Note that for  $ N>10 $ in Eq.(\ref{N>N2>10}):
Note that
\begin{eqnarray}\label{N-2 N-1}
\dfrac{N-2}{N-1}\in\Big[\frac{1}{2}, 1\Big]\,,\qquad \forall N\geq 3
\end{eqnarray}
From Eqs.(\ref{J_2 2.1 3}), (\ref{m_f(x,R_2) 2}) and (\ref{N-2 N-1}), we obtain.
%Substituting Eq.(\ref{N-2 N-1}) and (\ref{m_f(x,R_2) 2}) into Eq.(\ref{J_2 2.1 3}), we obtain
\begin{eqnarray}\label{J_2 2.1 4}
J_{2.1}(N,x)
&\leq & e^{-\frac{1}{2}\frac{\sqrt{N-1}}{e^{\gamma}}\,m_{f}(x,R_{2})} \label{J_2 2.1 4} \nonumber\\
&\leq & \bigg[\frac{\sqrt{N-1}}{2e^{\gamma}}\,m_{f}(x,R_{2})\bigg]^{-\varepsilon}\,,\qquad \varepsilon\in(0,e]  \label{J_2 2.1 5}\nonumber\\
&=& \dfrac{(2e^{\gamma})^{\varepsilon}}{m_{f}(x,R_{2})^{\varepsilon}}\dfrac{1}{(N-1)^{\varepsilon/2}}<\infty\,,\qquad \varepsilon\in(0,e] \label{J_2 2.1 6}
\end{eqnarray}
Note that with $ t=\frac{1}{2}\frac{\sqrt{N-1}}{\tilde{\gamma}}\,m_{f}(x,R_{2})\in(0,\infty) $ in Eq.(\ref{J_2 2.1 5}), we used together with Eq.(\ref{e_-t < t_-delta}).
%%%%%%%%%%%%%%%%%%%%%%%%%%%%%%%%%%%%%%%%%%%%%%%%%%%%%%%%%%%%%%%%%%%%%%%%%%%%%%%%%%%%%%%%%%%%%%%%%%%%%%%%%%%%%%%%%%%%%%%%%%%%%%%%%%%%%%%%%%%%
%%\subparagraph*{A.1.5 Auxiliary results for the investigation of $J_{2.2.2}^{(2)}(N,x)$}
\subsection*{A.5 Proof of Lemma 4}
%%%%%%%%%%%%%%%%%%%%%%%%%%%%%%%%%%%%%%%%%%%%%%%%%%%%%%%%%%%%%%%%%%%%%%%%%%%%%%%%%%%%%%%%%%%%%%%%%%%%%%%%%%%%%%%%%%%%%%%%%%%%%%%%%%%%%%%%%%%%
%%Note that for $N\geq 10 $, we have $ \log (N-1) \geq 1 $ and consequently we have
%%\begin{eqnarray}\label{loglog(N-1)w}
%%\log(\log (N-1)w)&=&\log \Big\{\log (N-1)\Big[1+\dfrac{\log w}{\log (N-1)}\Big]\Big\} \,,\quad N>10,\,\, w>e^{1+\Delta}\nonumber\\
%%&\leq & \log\log (N-1)+\log(1+\log w)\,,\quad N>10,\,\, w>e^{1+\Delta}.
%%\end{eqnarray}
%%which leads to the following results.
%%\begin{lemma}\label{corollary K(Delta)}
%%For an arbitrage small number $ \Delta>0 $, and $ w\in[e^{1+\Delta},\infty) $, we have $ \dfrac{\log(1+\log w)}{\log(\log w)} $ is bounded and precisely, we have
%%\begin{eqnarray}
%%&&\dfrac{\log(1+\log w)}{\log(\log w)}\leq \frac{\log(2+\Delta)}{\log(1+\Delta)} :=K(\Delta) \label{K(Delta)} 
%%\end{eqnarray}
%%Moreover, we have $1\leq K(\Delta)< \infty  $ and equality holds if and only if $ \Delta\rightarrow 0  $.
%%\end{lemma}
\begin{proof}[Proof of Lemma 4]
Let us define $ g(w)=\dfrac{\log(1+\log w)}{\log(\log w)} $ and clearly, we derive the 1st order derivative of $ g(w) $ as follows
\begin{eqnarray}\label{1st order g(w)}
\dfrac{d g(w)}{d w}= \dfrac{1}{w\,\log\log w}\bigg(\dfrac{\log\log w}{\log w+1}-\dfrac{\log(\log w+1)}{\log w}\bigg) 
\end{eqnarray}
For $ w\in(e,\infty) $, we have $ \log w >1 $ and $ \log\log w >0 $. Moreover, we have $ \log\log w<\log(\log w+1) $ and $\log w+1>\log w $, which implies 
\begin{eqnarray}\label{loglogw}
\dfrac{\log\log w}{\log w+1}-\dfrac{\log(\log w+1)}{\log w}<0
\end{eqnarray}
Substituting Eq.(\ref{loglogw}) into (\ref{1st order g(w)}), we obtain
\begin{eqnarray}\label{1st order g(w) eq2}
\dfrac{d g(w)}{d w}<0 
\end{eqnarray}
This implies that for $ w\in[e^{1+\Delta},\infty) $ with an arbitrage small $ \Delta $>0, we obtain
\begin{eqnarray}
&& \widetilde{\Delta} :=\sup\limits_{w\in[e^{1+\Delta},\infty)} g(w)=g(w)\bigg|_{w=e^{1+\Delta}}=\frac{\log(2+\Delta)}{\log(1+\Delta)} 
\end{eqnarray}
Obviously, $ 0<\log(1+\Delta)\leq\log(2+\Delta)\,,\,\forall \Delta>0$, and thus we have 
$$
1 \leq \widetilde{\Delta}=\frac{\log(2+\Delta)}{\log(1+\Delta)}<\infty \,.
$$
with the equality holds only if $ \Delta\rightarrow 0 $. The \text{Lemma} 4 is proven. 
\end{proof}
%%%%%%%%%%%%%%%%%%%%%%%%%%%%%%%%%%%%%%%%%%%%%%%%%%%%%%%%%%%%%%%%%%%%%%%%%%%%%%%%%%%%%%%%%%%%%%%%%%%%%%%%%%%%%%%%%%%%%%%%%%%%%%%%%%%%%%%%%%%%
%%\subparagraph*{A.1.5 Auxiliary results for the investigation of $J_{2.2.2}^{(2)}(N,x)$}
\subsection*{A.6 Proof of Lemma 5}
%%%%%%%%%%%%%%%%%%%%%%%%%%%%%%%%%%%%%%%%%%%%%%%%%%%%%%%%%%%%%%%%%%%%%%%%%%%%%%%%%%%%%%%%%%%%%%%%%%%%%%%%%%%%%%%%%%%%%%%%%%%%%%%%%%%%%%%%%%%%
%%%%%%%%%%%%%%%%%%%%%%%%%%%%%%%%%%%%%%%%%%%%%%%%%%%%%%%%%%%%%%%%%%%%%%%%%%%%%%%%%%%%%%%%%%%%%%%%%%%%%%%%%%%%%%%%%%%%%%%%%%%%%%%%%%%%%%%%%%%%
\begin{proof}[Proof of Lemma 5]
%\begin{lemma}\label{G_a+b}
For $ x\in\mathbb{R}^{d} $, $ d\geq 1 $, $ \alpha=1,2,3,4\ldots$, we prove that \\
\textbf{Step 1} 
there exits positive constant $ a_{0} $, such that
\begin{align}\label{G_rho^d eq1}
G(\mid\log{\xi_{2,x}}\mid^{\alpha})\leq a_{0}\,G(\mid\log{\rho}^{d}\mid^{\alpha})+b_{0}
\end{align}
Recall the definition of $ \xi_{N,x}=(N-1)V_{d}\,e^{\gamma}\min\limits_{j=2,\ldots ,N}\rho^{d}(x, X_{j}) $ and the random variables $ X_{1},X_{2}\cdots ,X_{N} $ are i.i.d and are assumed follow the same law of $ \xi $. Set $ N=2 $:
\begin{align}
\xi_{2,x}=V_{d}\,e^{\gamma}\rho^{d}(x,\xi),\,\, \text{and thus}
\,\,\,  \log \xi_{2,x} =  \log \rho^{d}+c_{0}\,,\,\,\, c_{0}=\log \big(V_{d}e^{\gamma}\big)
\end{align}
and  
\begin{align}\label{xi_2 alpha eq1}
|\log \xi_{2,x}|^{\alpha} &=\big| \log \rho^{d}\big|^{\alpha}\cdot\bigg|1+\dfrac{c_{0}}{(\log \rho^{d})}\bigg|^{\alpha}  \nonumber\\
&\leq \big| \log \rho^{d}\big|^{\alpha}\cdot2^{(\alpha-1)}\bigg(1+\dfrac{|c_{0}|^{\alpha}}{\big| \log \rho^{d}\big|^{\alpha}} \bigg) 
\end{align}
where we apply the Jensen inequality to the convex function $ f(y)=y^{\alpha} $, $ \alpha\geq 1 $:
\begin{align}
\bigg(\dfrac{a+b}{2}\bigg)^{\alpha}\leq \dfrac{a^{\alpha}+b^{\alpha}}{2} \quad &\Longrightarrow\quad (a+b)^{\alpha}\leq 2^{(\alpha-1)}\big(a^{\alpha}+b^{\alpha}\big)  \nonumber\\
&\Longrightarrow\quad |a+b|^{\alpha}\leq 2^{(\alpha-1)}\big|a^{\alpha}+b^{\alpha}\big|\leq 2^{(\alpha-1)}\big(|a|^{\alpha}+|b|^{\alpha}\big)
\end{align}
We define $\widetilde{d}= \min\limits_{\rho\in (0,1/e) \cup(e,\infty)} |\log \rho^{d}|  $. Then, one has $1\leq d=\log e^{d}< \widetilde{d}<\infty $ and
\begin{align}\label{xi_2 alpha eq2}
\sup_{\rho\in (0,1/e) \cup(e,\infty)}\bigg(1+\dfrac{|c_{0}|^{\alpha}}{\big| \log \rho^{d}\big|^{\alpha}} \bigg) =1+\dfrac{|c_{0}|^{\alpha}}{\big(\tilde{d}\big)^{\alpha}}:=c_{1}
\end{align}
and $ 1\leq c_{1}<\infty$. From Eqs.(\ref{xi_2 alpha eq1}) and (\ref{xi_2 alpha eq2}), we obtain
\begin{align}\label{xi_2 alpha eq3}
|\log \xi_{2,x}|^{\alpha} &\leq 2^{(\alpha-1)}c_{1}\big| \log \rho^{d}\big|^{\alpha}  \nonumber\\
& =\tilde{c}_{1}\big|\log \rho^{d}\big|^{\alpha} ,\quad \tilde{c}_{1}=2^{(\alpha-1)}c_{1}
\end{align}
Clearly, $ |\log \xi_{2,x}|^{\alpha} \geq 1 $ and we calculate the logarithm of Eq.(\ref{xi_2 alpha eq3}) as follows.
\begin{align}\label{xi_2 alpha eq4}
\log|\log \xi_{2,x}|^{\alpha} &\leq \log\tilde{c}_{1}+\log\big| \log \rho^{d}\big|^{\alpha}        \nonumber\\
&=\log\big| \log \rho^{d}\big|^{\alpha} \bigg(1+\dfrac{\log\tilde{c}_{1}}{\log\big| \log \rho^{d}\big|^{\alpha}}\bigg)  \nonumber\\
&\leq c_{2}\log\big| \log \rho^{d}\big|^{\alpha} 
\end{align}
where
\begin{align}\label{xi_2 alpha eq4}
c_{2}:=\sup_{\rho\in (0,1/e) \cup(e,\infty)}\bigg(1+\dfrac{\log\tilde{c}_{1}}{\log\big| \log \rho^{d}\big|^{\alpha}}\bigg)=1+\dfrac{\log \tilde{c}_{1}}{\alpha\log \widetilde{d}}
\end{align}
and obviously, $ 1<c_{2}<\infty $. From Eqs.(\ref{xi_2 alpha eq3}), we have
\begin{align}\label{xi_2 alpha eq5}
G(|\log \xi_{2,x}|) \leq \tilde{c}_{1}c_{2}\,G\big(\big|\log \rho^{d}\big|\big)  
\end{align}
%%with $ c_{1},c_{2} \in (1,\infty) $.\\
\textbf{Step 2} there exits positive constants $ a_{1},b_{1} $, such that
\begin{align}\label{G_rho^d^alpha eq1}
G(\mid\log{\rho}^{d}\mid^{\alpha})\leq a_{1}\,G(\mid\log{\rho}\mid^{\alpha})+b_{1}
\end{align}
According to the definition of $ G(\cdot) $, we have 
\begin{align}\label{G_rho^d^alpha eq2}
G(\mid\log{\rho}^{d}\mid^{\alpha})&=\big|\log{\rho}^{d}\big|^{\alpha}\log\big(\big|\log{\rho}^{d}\big|^{\alpha}\big)    \nonumber\\
&=d^{\alpha}\big|\log{\rho}\big|^{\alpha} \alpha\big(\log d+ \log\big(\big|\log{\rho}\big|\big) \big)   \nonumber\\
&=d^{\alpha}G(\mid\log{\rho}\mid^{\alpha})+d^{\alpha}\log d^{\alpha}|\log\rho|^{\alpha}
\end{align}
Note that $\mid\log{\rho}^{d}\mid^{\alpha}= \mid\log^{\alpha}{\rho}^{d}\mid $. From the monotonicity of the power function ($ \alpha>0 $), one has 
\begin{enumerate}[(i)]
\item when $ \rho\in[1/e,e] $, we have $ 0\leq |\log{\rho}|^{}\leq 1 $ and equivalently $ 0\leq |\log{\rho}|^{\alpha}\leq 1 $. In this case, $ G(\mid\log{\rho}\mid^{})=G(\mid\log{\rho}\mid^{\alpha})=0$, and $G(\mid\log{\rho}^{d}\mid^{\alpha})=d^{\alpha}\log d^{\alpha}|\log\rho|^{\alpha}\geq 0  $. Obviously, Eq.(\ref{G_rho^d^alpha eq1}) holds.
\item when $ \rho\in(0,1/e)\cup (e,\infty) $, we have $ G(\mid\log{\rho}\mid^{\alpha}) \neq 0 $ and 
\begin{align}\label{G_rho^d^alpha eq3}
\dfrac{G(\mid\log{\rho}^{d}\mid^{\alpha})}{G(\mid\log{\rho}\mid^{\alpha})}=d^{\alpha}+d^{\alpha}\log d\dfrac{1}{\log |\log\rho|}
\end{align}
\textbf{(ii.1)} when $ \rho\in(e,\infty):= (e,T)\cup [T,\infty)$ with $T\in(e,\infty) $, we have $ \log\rho>1$, and hence  $  |\log{\rho}|^{\alpha}=(\log{\rho})^{\alpha}$, and $ G(\mid\log{\rho}\mid^{\alpha})>0 $.  \\
\textbf{(ii.1.a)} when $ \rho\in (e,T)$, we have $ 1< \log{\rho}<\log T $ and 
\begin{align}\label{sup_rho 1}
\sup_{\rho\in (e,T)}\big\{{\log |\log\rho|}\big\}=\log(\log T)
\end{align}
%%\end{enumerate}
From Eqs.(\ref{sup_rho 1}) and (\ref{G_rho^d^alpha eq2}), we obtain:
\begin{align}
G(\mid\log{\rho}^{d}\mid^{\alpha})&\leq d^{\alpha}G(\mid\log{\rho}\mid^{\alpha})+d^{\alpha}\log d^{\alpha}\log(\log T)  \nonumber\\
&=a_{1}G(\mid\log{\rho}\mid^{\alpha})+b_{1}
\end{align}
with $ a_{1}:=d^{\alpha} >0$ and $ b_{1}:=d^{\alpha}\log d^{\alpha}\log(\log T)>0$. Thus, we proved Eq.(\ref{G_rho^d^alpha eq1})\\
\textbf{(ii.1.b)} when $ \rho\in [T,\infty)$, we have $ 1< \log{\rho}\leq \log T $ and  
\begin{align}\label{sup_rho 2}
\sup_{\rho}\bigg\{\dfrac{1}{\log |\log\rho|}\bigg\}=\dfrac{1}{\log |\log T|}
\end{align}
From Eqs.(\ref{sup_rho 1}) and (\ref{G_rho^d^alpha eq3}), we obtain:
\begin{align}
\dfrac{G(\mid\log{\rho}^{d}\mid^{\alpha})}{G(\mid\log{\rho}\mid^{\alpha})}&\leq d^{\alpha}+d^{\alpha}\log d\dfrac{1}{\log |\log T|}:=a_{1} <\infty \nonumber\\
\Longleftrightarrow\,\, G(\mid\log{\rho}^{d}\mid^{\alpha})&\leq a_{1}G(\mid\log{\rho}\mid^{\alpha})
\end{align}
Thus, the Eq.(\ref{G_rho^d^alpha eq1}) is proved.\\
\textbf{(ii.2)} when $ \rho\in(0,1/e) $, we set $ \rho=1/\bar{\rho}$ with $\bar{\rho}\in (e,\infty)$ and hence, $ |\log \rho|= |\log \bar{\rho}| $. Obviously, the Eqs.(\ref{G_rho^d^alpha eq2}) and (\ref{G_rho^d^alpha eq3}) stay unchanged after replacing $ \rho $ with $ \bar{\rho} $. From \textbf{(ii.1.a)} and \textbf{(ii.1.b)}, the Eq.(\ref{G_rho^d^alpha eq1}) is proven.
\end{enumerate}
From (i) and (ii), the Eq.(\ref{G_rho^d^alpha eq1}) is proven.\\
\end{proof}
%%%%%%%%%%%%%%%%%%%%%%%%%%%%%%%%%%%%%%%%%%%%%%%%%%%%%%%%%%%%%%%%%%%%%%%%%%%%%%%%%%%%%%%%%%%%%%%%%%%%%%%%%%%%%%%%%%%%%%%%%%%%%%%%%%%%%%%%%%%%
\subsection*{A.7 Proof of Lemma 6}
%\subsubsection*{A.2.1 Proof of Lemma 6}
%%%%%%%%%%%%%%%%%%%%%%%%%%%%%%%%%%%%%%%%%%%%%%%%%%%%%%%%%%%%%%%%%%%%%%%%%%%%%%%%%%%%%%%%%%%%%%%%%%%%%%%%%%%%%%%%%%%%%%%%%%%%%%%%%%%%%%%%%%%%
%\begin{lemma}\label{eta>0 both}
%For each $ (x,y)\in A $, $ x\neq y $, one has $ \eta^{y}_{N,x}>0 $ and $ \eta^{x}_{N,y}>0 $ almost sure.
%\end{lemma}
\begin{proof}[Proof of Lemma 6]
Recall that $ X_{1},\ldots,X_{N}\,\overset{\text{i.i.d.}}{\sim}\xi  $ and substituting $ u=0 $ into Eq.(125) in the main text, we have
%Eq.(\ref{F Nx u 1}) and (\ref{F Nx u 2}), and $ N\geq 3 $, we obtain
\begin{eqnarray}
\mathsf{Pr}\big(\eta^{y}_{N,x}> 0\big)&=&\mathds{1}\Big[\rho(x, y)>0\Big] \Big[\mathsf{Pr}\Big(\rho(\xi,x)>0\Big)\Big]^{N-2}\,\nonumber\\ 
&=&\Big\{1-\mathds{1}\Big[\rho(x, y)=0\Big]\Big\} \Big[1-\mathsf{Pr}\Big(\rho(\xi,x)=0\Big)\Big]^{N-2}\,\nonumber\\ 
&=&\Big\{1-\underbrace{\mathds{1}\big[x=y\big]}_{=0}\Big\} \Big[1-\underbrace{\mathsf{Pr}\big(\xi=x\Big)}_{=0}\big]^{N-2}\,\nonumber\\
&=&1
\end{eqnarray}
Analogously, we can obtain $ \mathsf{Pr}\big(\eta^{x}_{N,y}> 0\big) =1$. Therefore, we exclude a set of zero probability for the random variable is zero, i.e. $ \eta^{}_{N,x,y}=\big(\eta^{y}_{N,x}, \eta^{x}_{N,y}\big)=(0,0) $. The \text{Lemma} 6 is proved.
\end{proof}


\begin{thebibliography}{10}
%\bibitem{Archer, Il Memming Park 2}Archer, E., Park, I.M., Pillow, J. (2013) Bayesian entropy estimation for binary spike train data using parametric prior knowledge
%. Advances in Neural Information Processing Systems, 26, 1700-1708.

%\bibitem{Archer, Il Memming Park 1}Archer, E., Park, I.M., Pillow, J. (2014) Bayesian entropy estimation for countable discrete. Journal of Machine Learning Research, 15, 2833-2868.

\bibitem{Batsidis, Zografos 2013} Batsidis, A. and Zografos, K. (2013). A necessary test of fit of specific elliptical distributions based on an estimator of Song's
measure. The Journal of Multivariate Analysis, 113, 91-105.

\bibitem{Berrett, Samworth} Berrett, T. B. and Samworth, R. J. (2023). Efficient functional estimation and the super-oracle phenomenon. The Annals of Statistics, 51(2), 668-690.

\bibitem{Berrett, Samworth, Yuan} Berrett, T. B., Samworth, R. J. and Yuan, M. (2019). Efficient multivariate entropy estimation via k-nearest neighbour distances. The Annals of Statistics, 36(5), 2153-2182.

\bibitem{Biau, Devroye} Biau, G. and Devroye, L. (2015). Lectures on the Nearest Neighbor Method. Springer.
 

\bibitem{Bobkov}Bobkov, B. and Madiman, M. (2011). Dimensional behaviour of entropy and information. C. R. Acad. Sci. Paris, Ser. I, 349, 201-204.

\bibitem{Borkar} Borkar, V.S.(1995). Probability Theory: An Advanced Course. Springer-Verlag, New York. 
%\bibitem{Bonev} Bonev, B.I. (2010). Feature selection based on information theory. PhD thesis, University of Alicante.


%\bibitem{Chao}Bailey, J., Houle, M.E., Ma. X. (2022). Relationships between tail entropies and local intrinsic dimensionality and their use for estimation and feature representation. Entropy , 24, 1220.







\bibitem{Bulinski, Dimitrov} Bulinski, A. and Dimitrov, D. (2019). Statistical Estimation of Shannon Entropy. Acta Mathematica Sinica, English Series, 35(1), 17-46.


%\bibitem{Buono e Longobardi 2020} Buono, F., Longobardi, M. (2020). Varentropy of past lifetimes. Preprint, arXiv:2008.07423.

%\bibitem{Chao}Chao, A., Shen T. (2003). Nonparametric estimation of Shannon's index of diversity when there are unseen species in sample. Environmental and Ecological Statistics, 10, 429-443.

%\bibitem{Cover and Thomas} Cover, T.M., Thomas, J. (2006). Elements of Information Theory (second Edition). John Wiley and Sons.

\bibitem{Cadirci, Leonenko, Makogin} Cadirci, M. S., Evans, D., Leonenko, N. and Makogin, V. (2022). Entropy-based test for generalized Gaussian distributions.  Comput. Stat. Data.
Anal., 173, 107502.

\bibitem{Crescenzo 2021a} Di Crescenzo, A. and Paolillo, L. (2021). Analysis and applications of the residual varentropy of random lifetimes. Probability in the Engineering and Informational Sciences, 35, 680-698.

%\bibitem{Crescenzo 2021b} Di Crescenzo, A., Paolillo, L., Su\'a rez-Llorens, Alfonso. (2021). Stochastic comparisons, differential entropy and varentropy for distributions induced by probability density functions. Preprint, arXiv:2103.11038v1 [math.PR] 19 Mar 2021.



%\bibitem{De Hek} De Hek, P.A. (1999). On endogenous growth under uncertainty. International Economic Review, 40(3), 727-744.
\bibitem{Delattre e Fournier} Delattre, S. and Fournier, N. (2017). On the Kozachenko-Leonenko entropy estimator. J. Statist. Planning Inference, 185, 69-93.

\bibitem{Devroye, Gyorfi} Devroye, L. and Gy\"orfi, L. (2022). On the Consistency of the Kozachenko-Leonenko Entropy Estimate. 
IEEE Transactions on Information Theory, 68(2), 1178-1185.

%\bibitem{Hausser, Strimmer} Hausser, J. and Strimmer, K. (2009) Entropy inference and the James-Stein estimator, with application to nonlinear gene association networks. The Journal of Machine Learning Research, 10, 1469-1484.

%\bibitem{Golan} Golan, A.(2002). Information and Entropy Econometrics-Editor's View. Journal of Econometrics, 107, 1-15.

\bibitem{Gao e Oh e Viswanath} Gao, W., Oh, S. and Viswanath, P. (2017). Density functional estimators with k-nearest neighbor bandwidths. Proc. IEEE Int. Symp. Inf. Theory (ISIT), Aachen, Germany, 1351-1355.


\bibitem{Goria} Goria, M. N., Leonenko, N. N., Mergel, V. V. and Novi Inverardi, P. L. (2005) A new class of random vector entropy estimators and its applications in testing statistical hypotheses. J. Nonparam. Statist., 17(3), 277-297.

\bibitem{Kozachenko e Leonenko}Kozachenko, L. F. and Leonenko, N. (1987). Sample estimate of the entropy of a random vector. Problems. Information Transmission, 23(1), 95-101.

\bibitem{Leonenko, Makogin, Cadirci}Leonenko, N., Makogin, V. and Cadirci, M. S. (2021). The entropy based goodness of fit tests for generalized von Mises-Fisher distributions and beyond. Electron. J. Statist., 15(2), 6344-6381.

\bibitem{Leonenko et al., 2008} Leonenko, N., Pronzato, L. and Savani, V. (2008). A class of r\'enyi information estimators for multidimensional densities. The Annals of Statistics, 36(5), 2153-2182.

\bibitem{Leonenko et al., 2010} Leonenko, N. and Pronzato, L. (2010). Correction: a class of r\'enyi information estimators for multidimensional densities. The Annals of Statistics, 38(6), 3837-3838.
%\bibitem{Haurie} Haurie, A. (2003). Integrated assessment modelling for global climate change: an infinite horizon optimization viewpoint. Environmental Modeling and Assessment, 8, 117-132.


%%\bibitem{Maadani et al.} Maadani S., Mohtashami Borzadaran G. R., Rezaei Roknabadi A.H. (2022) Varentropy of order statistics and some stochastic comparisons. Communications in Statistics - Theory and Methods, 51(18), 6447-6460.


%\bibitem{Nemenman} Nemenman, I., Shafee, F., Bialek, W. (2002). Entropy and Inference, Revisited. Preprint: arXiv:physics/0108025.


%\bibitem{Laporte} Laporte, A., Ferguson, B. (2007). Investment in health when health is stochastic. Journal of Population Economics, 20(2), 423-444.


%\bibitem{Liu} Liu, J. (2007). Information Theoretic Content and Probability. PhD thesis, University of Florida.

%%\bibitem{Lord e Sun e Bollt} Lord, W. M., Sun, J., Bollt, E.M. (2018). Geometric k-nearest neighbor estimation of entropy and mutual information. Chaos, Interdiscipl. J. Nonlinear Sci.,, 28(3), Art. no. 033114.

%\bibitem{Martin} Martin, I.W.R., Pindyck, R.S. (2015). Averting catastrophes: the strange economics of Scylla and Charybdis. American Economic Review, 105(10), 2947-2985.


%\bibitem{Mc} McKinsey\&Company (2009). Pathways to a low carbon economy. Version 2 of the Global Greenhouse Gas Abatement Cost Curve. Available online: http://www.mckinsey.com/business-functions/sustainability-and-resource-productivity/our-insights/pathways-to-a-low-carbon-economy


%\bibitem{Mariani} Mariani, F., P\'erez-Barahona, A., Raffin, N. (2010). Life expectancy and the environment. Journal of Economic Dynamics and Control, 34, 798-815.


%\bibitem{Muller} M\"uller-Furstenberger, G., Schumacher, I. (2015). Insurance and climate-driven extreme events. Journal of Economic Dynamics and Control, 54(C), 59-73.

%\bibitem{Nordhaus} Nordhaus, W.D. (2008). A question of balance: weighing the options on global warming policies. Yale University Press, New Haven.

%\bibitem{Paninski} Paninski, L. (2003) Estimation of Entropy and Mutual Information. Neural Computation, 15(6), 1191-1253.

%\bibitem{Picone} Picone, G., Uribe, M., Wilson, R.M. (1998). The effect of uncertainty on the demand for medical care, health capital and wealth. Journal of Health Economics, 17(2), 171-185.


\bibitem{Penrose e Yukich 2011} Penrose, M. D. and Yukich, J. E. (2011). Laws of large numbers and nearest neighbor distances. In Advances in Directional and Linear Statistics, Springer, 189-199.

\bibitem{Penrose e Yukich 2013} Penrose, M. D. and Yukich, J. E. (2013). Limit theory for point processes in manifolds. Ann. Appl. Probab. 23(6), 2161-2211.


%%\bibitem{Raqab 2021} Raqab, M.Z., Bayoud, H.A., Qiu, G. (2021) Varentropy of inactivity time of a random variable and its related applications. IMA J. Math. Control. Info., 39, 132-154.

%\bibitem{Rossana} Rossana, R.J. (1985). Delivery lags and buffer stocks in the theory of investment by the firm. Journal of Economic Dynamics and Control, 9, 153-193.

%%\bibitem{Shannon} Shannon, C. (1948). A Mathematical Theory of Communication. The Bell System Technical Journal, 27, 379-423.

%%\bibitem{Singh, H} Singh, H., Misra, N., Hnizdo, V., Fedorowicz, A., Demchuk, E. (2003). Nearest neighbor estimates of entropy. Amer. J. Math. Manage. Sci., 23(3-4), 301-321.

%%\bibitem{Singh, S} Singh, S., P{\'o}czos, B. (2016). Finite-sample analysis of fixed-k nearest neighbor density functional estimators. Proc. Adv. Neural Inf. Process. Syst., 1217-1225.


%%\bibitem{Sricharan1} Sricharan, K., Raich, R., Hero, A.O. (2012). Estimation of nonlinear functionals of densities with confidence. IEEE Trans. Inf. Theory, 58(7), 4135-4159.

%%\bibitem{Sricharan2} Sricharan, K., Wei, D.,Hero, A.O. (2013). Ensemble estimators for multivariate entropy estimation. IEEE Trans. Inf. Theory, 59(7), pp. 4374-4388.

\bibitem{Song 2001} Song, K.S. (2001). R{\'e}nyi information, loglikelihood and an intrinsic distribution measure.  Journal
of Statistical Planning and Inference, 93, 51-69.

%\bibitem{Tsue} Tsur, Y., Withagen, C. (2011). Preparing for catastrophic climate change. Center for Agricultural Economic Research, Hebrew University of Jerusalem, Discussion Paper, No.5.11.
\bibitem{Tsybakov e Meulen 1996} Tsybakov, A.B. and Van der Meulen, E.C. (1996). Root-n consistent estimators
of entropy for densities with unbounded support. Scand. J. Statist, 23 (1), 75-83.

%\bibitem{United} United Nations Office of the High Commissioner for Human Rights. (2008). The right to health. Printed at United Nations, Geneva. 

%\bibitem{United} United Nations. (2015). World population ageing 2015. New York: Department of Economic and Social Affairs, Population Division.

%\bibitem{Vu} Vu, V.Q., Yu, B., Kass, R.E. (2007) Coverage adjusted entropy estimation. Statistics in Medicine, 26(21), 4039-4060.


\bibitem{Yeh} Yeh, J. (2014). Real Analysis: Theory of Measure and Integration (3rd edition). World Scientific, Singapore.

\bibitem{Zografos 2008} Zografos, K. (2008). On Mardia's and Song's measures of kurtosis in elliptical distributions. J. Multivariate Analysis, 99, 858-879.


\end{thebibliography}
\end{document}